\documentclass[reqno]{memo-l}
\usepackage[T1]{fontenc}
\usepackage{mathtools}
\usepackage{xifthen}
\usepackage{amssymb}
\usepackage{enumitem}
\usepackage{url}

\usepackage{mathrsfs}  
\usepackage{graphicx} 
\usepackage{float}
\usepackage{hyperref}
\usepackage[table]{xcolor}

\usepackage{varioref}
\let\vref\ref

\usepackage{xspace}
\usepackage[normalem]{ulem}
\usepackage[only, llbracket, rrbracket, ocircle]{stmaryrd}
\usepackage{imakeidx}

\usepackage{booktabs}
\usepackage{array}
\usepackage{multirow}
\usepackage{colortbl}
\usepackage{boldline}
\usepackage{cellspace}
\usepackage{makecell}
\usepackage[usestackEOL]{stackengine}

\usepackage{leading}
\leading{13.5pt}

\hypersetup{
  colorlinks = true, 
  linkcolor = red!70!black, 
  citecolor = blue, 
  menucolor = green, 
  urlcolor = blue,
  hypertexnames = false,
}

\usepackage{tikz}
\usetikzlibrary{cd}
\usetikzlibrary{calc}
\usetikzlibrary{snakes}
\usetikzlibrary{arrows}
\usetikzlibrary{arrows.meta}
\usetikzlibrary{backgrounds}
\usetikzlibrary{decorations.pathmorphing}
\usetikzlibrary{decorations.markings}

\theoremstyle{plain}
\newtheorem{proposition}{Proposition}[chapter]
\newtheorem{theorem}[proposition]{Theorem}
\newtheorem{lemma}[proposition]{Lemma}
\newtheorem{corollary}[proposition]{Corollary}

\theoremstyle{definition}
\newtheorem{definition}[proposition]{Definition}

\newtheorem{question}[proposition]{Question}

\theoremstyle{remark}
\newtheorem{remark}[proposition]{Remark}

\newlist{tfae}{enumerate}{1}
\setlist[tfae]{label=\textup{(\alph*)}, ref=\alph*}
\makeatletter
\newcommand\implication[2]{$(\@alph #1)\implies(\@alph #2)$}
\newcommand\tfaeitem[1]{\textup{(\@alph#1)}}
\makeatother

\newlist{thmlist}{enumerate}{1}
\setlist[thmlist]{label=\textup{(\roman*)}, ref=(\roman*)}
\newcommand\thmitem[1]{\textup{(\emph{\romannumeral #1})}}

\usepackage[skins, breakable, hooks]{tcolorbox}

\tcbset{common/.style={
  breakable,
  enhanced jigsaw,
  before skip=\medskipamount,
  after skip=\medskipamount,
  colback=#1,
  grow to right by=0.25cm,
  grow to left by=0.25cm,
  if odd page={right=0.25cm,left=0.25cm}{right=0.25cm,left=0.25cm},
  toggle enlargement,
  boxsep=0pt,
  toprule=0.07mm,
  bottomrule=0.07mm,
  leftrule=0mm,
  rightrule=0mm,
  arc=0mm,
  }
}

\newlength{\boxesnormalparindent}
\AtBeginDocument{\setlength{\boxesnormalparindent}{\parindent}}
\tcbset{before upper={\setlength{\parindent}{\boxesnormalparindent}}} 

\colorlet{thm-color}{blue!5}
\colorlet{proof-color}{gray!8}
\foreach \t in {proposition, theorem, lemma, corollary}
  {
    \expandafter\tcolorboxenvironment\expandafter{\t}{
        common={thm-color},
    }
  }
\foreach \t in {proof}
  {
    \expandafter\tcolorboxenvironment\expandafter{\t}{
        common={proof-color},
        after skip=\bigskipamount,
    }
  }

\providecommand*\lettergroup[1]{%
      \par\par
      \nopagebreak
  }

\makeindex[title=Index, 
        program=texindy,
        options=-q -M ./index.xdy
        ]
\makeindex[title=Notations,
        name=notations,
        program=texindy,
        options=-q -M ./index.xdy
        ]

\newcounter{indexcount}
\newcommand{\threedigit}[1]{\ifnum #1<100 0\fi\ifnum #1<10 0\fi#1}
\newcommand{\indexsymbol}[1]{%
  \ifcsname\detokenize{SYM@@#1}\endcsname
    \index[notations]{\csname\detokenize{SYM@@#1}\endcsname @#1}%
  \else
    \stepcounter{indexcount}%
    \expandafter\xdef\csname SYM@@\detokenize{#1}\endcsname{%
      \expandafter\threedigit\expandafter{\romannumeral-`Q\theindexcount}%
    }%
    \index[notations]{\threedigit{\theindexcount}@\unexpanded{\unexpanded{#1}}}%
  \fi
}

\newcommand{\NN}{\mathbb{N}}
\newcommand{\ZZ}{\mathbb{Z}}
\newcommand{\kk}{\Bbbk}

\newcommand{\cB}{\mathcal{B}}

\newcommand{\cF}{\mathcal{F}}

\newcommand{\cG}{\mathcal{G}}

\newcommand{\cR}{\mathcal{R}}

\newcommand{\ssC}{\mathsf{C}}

\newcommand\I{\mathcal{I}}
\let\emptyset\varnothing

\newcommand\m{\mathfrak{m}}
\newcommand\dR{\mathrm{dR}}
\DeclareMathOperator\charact{char} 
\DeclareMathOperator\coker{coker}
\DeclareMathOperator{\HC}{HC}
\DeclareMathOperator{\BC}{\mathcal BC}

\DeclareMathOperator{\ev}{ev}
\DeclareMathOperator{\rot}{rot}
\DeclareMathOperator{\rad}{rad}
\DeclareMathOperator{\rk}{rk}

\DeclareMathOperator{\Hom}{Hom}

\DeclareMathOperator{\Ext}{Ext}
\DeclareMathOperator{\Tor}{Tor}

\DeclareMathOperator{\gldim}{gl.dim}

\DeclareMathOperator{\HH}{HH}
\let\H\relax
\DeclareMathOperator{\H}{H}
\newcommand\place{\mathord{-}}
\newcommand\C{\mathscr{C}}
\newcommand\G{\mathscr{G}}
\newcommand\id{\mathrm{id}}

\newcommand\lfact[2]{{#2}\if1#1'\fi_{\mathsf{lt}}}
\newcommand\rfact[2]{{#2}\if2#1'\fi_{\mathsf{rt}}}

\DeclareMathAlphabet{\mathbx}{U}{bbold}{m}{n}
\newcommand\one{\mathbx{1}}

\let\oldparallel\parallel
\renewcommand\parallel{\mathbin{\oldparallel}}

\mathchardef\mhyphen="2D

\newcommand{\rsmod}[1]{{#1\mhyphen\mathrm{mod}}}

\newcommand\cyclespace{\mspace{-5mu}}
\newcommand\cycle[1]{%
   \langle\cyclespace\langle#1\rangle\cyclespace\rangle}
\newcommand\hcycle[1]{\llbracket#1\rrbracket}
\newcommand\Cprim{\C^{\,\mathsf{basic}}}
\newcommand\Crep{\overline\C}
\newcommand\Crepprim{\overline\C^{\,\mathsf{basic}}}
\DeclarePairedDelimiter\abs{\lvert}{\rvert}

\newcommand\claim[2][0.8]{%
  \ifthenelse{\dimtest{#1\displaywidth}>{0\displaywidth}}{
    \begin{minipage}{#1\displaywidth}
    \itshape #2
    \end{minipage}
  }{%
    \begin{minipage}{-#1\displaywidth}
    \itshape\centering #2
    \end{minipage}
  }
}

\def\spiral[#1](#2)(#3:#4)(#5:#6)[#7]{%
	\pgfmathsetmacro{\domain}{#4+#7*360}
	\pgfmathsetmacro{\growth}{180*(#6-#5)/(pi*(\domain-#3))}
	\draw [
              #1,
	      shift={(#2)},
	      domain=#3*pi/180:\domain*pi/180,
	      variable=\t,
	      smooth,
	      samples=int(\domain/5)
              ] 
           plot ({\t r}: {#5+\growth*\t-\growth*#3*pi/180})
}

\tikzset{
  snake it/.style={decorate, decoration=snake},
  bl/.style={circle, draw=white, thin,fill=black!100, scale=0.5},
  cg/.style={circle, draw, thin,fill=black!10, scale=0.8},
  cw/.style={circle, draw, thin,fill=white, scale=0.8},
  B/.style={circle, draw,fill=black!30, scale=1.6},
  b/.style={circle, draw,fill=black, scale=0.3},
  r/.style={circle, draw=red,fill=red, scale=0.3},
}

\newcommand{\newterm}[1]{\textbf{#1}}

\newcommand\fd{f.d.\@\xspace}

\definecolor{alizarin}{rgb}{0.82, 0.1, 0.26}
\definecolor{candyapplered}{rgb}{1.0, 0.03, 0.0}

\definecolor{darkgreen}{rgb}{0.5, 0.5, 0.26}
\newenvironment{newsib}{\color{black}}{}
\definecolor{blue(pigment)}{rgb}{0.2, 0.2, 0.6}

\definecolor{byzantine}{rgb}{0.74, 0.2, 0.64}

\definecolor{blue-violet}{rgb}{0.54, 0.17, 0.89}

\begin{document}

\frontmatter

\title{The Hochschild cohomology and the Tamarkin-Tsygan calculus    of gentle algebras}

\date{\today}

\author{Cristian Chaparro}
\address{Department of Mathematics with Computer Science\\
 Guangdong Technion-Israel Institute of Technology (GTIIT)\\
 241 Daxue Road - Jinping District\\
 Shantou, Guangdong Province\\
China
 }
\email{cristian.chaparro@gtiit.edu.cn}

\author{Sibylle Schroll}
\address{Department of Mathematics\\
  University of Cologne\\
  Weyertal 86--90\\
  Cologne, Germany
  }
\email{schroll@math.uni-koeln.de}

\author{Andrea Solotar}
\address{Departamento de Matem\'atica\\
  FCEyN Universidad de Buenos Aires\\  
  and IMAS-CONICET\\
  Pabellon I -- Ciudad Universitaria\\
  (1428) Buenos Aires\\ 
  Argentina
  }
\email{asolotar@dm.uba.ar}

\author{Mariano Su\'arez-\'Alvarez}
\address{Departamento de Matem\'atica\\
  FCEyN Universidad de Buenos Aires\\ 
  and IMAS-CONICET\\
  Pabellon I -- Ciudad Universitaria\\
  (1428) Buenos Aires\\ 
  Argentina
  }
\email{mariano@dm.uba.ar}


\begin{abstract}
The Tamarkin Tsygan calculus of a finite dimensional algebra is a differential calculus given by the comprehensive data of the Hochschild cohomology, its structure both as a graded commutative algebra under the cup product and as a graded Lie algebra under the Gerstenhaber bracket, together with the Hochschild homology and its module structure over the Hochschild cohomology given by the cap product as well as the Connes differential. In this paper, we calculate the whole of the Tamarkin  Tsygan calculus for the class of gentle algebras. Apart from some isolated calculations, this is, to our knowledge, the first complete calculation of this calculus for a family of finite dimensional algebras. Gentle algebras appear in many different areas of mathematics such as the theory of cluster algebras, N=2 gauge theories and homological mirror symmetry of surfaces. For the latter Haiden-Katzarkov-Kontsevich show that the partially wrapped Fukaya category of a graded surface is triangle equivalent to the perfect derived category of a graded gentle algebra. Thus the symplectic cohomology of the surface should be linked to the Hochschild cohomology of the gentle algebra. We give a description of how this connection might be visualised, in that, we  show how these structures are
encoded in the geometric surface model of the bounded derived category
associated to a gentle algebra via its ribbon graph. 
 Finally, we show how to recover structural results for a gentle algebra given its Tamarkin-Tsygan calculus.

\end{abstract}

\maketitle

\tableofcontents

\mainmatter

\chapter{Introduction}




Given a field $\kk$ and an associative  $\kk$-algebra $A$,  the Hochschild cohomology $\H^*(A, M )$ and Hochschild homology $\H_*(A, M )$ of $A$ with coefficients in  an  $A$-bimodule $M$ are the
graded vector spaces:
\[
\H^*(A, M ) = \Ext^*_{A^e} (A, M ) \mbox{ and } \H_*(A, M ) = \Tor^{A^e}_* (M, A).
\]
Arbitrary projective resolutions of $A$  as  $A$-bimodule, such as the bar resolution, can  theoretically  be used to compute these spaces.
However, in practice the bar resolution is not suitable for computations since the spaces involved quickly become too big.
Hochschild homology and cohomology, as  defined by G.\@ Hochschild in 1945 
in terms of the bar resolution, are important tools to study the
representation theory of a given associative algebra. Both Hochschild homology and cohomology are invariant under
Morita equivalence and more generally under derived equivalence \cite{Rickard}.
Let us just mention that $\HH^0(A)$ is the center of $A$ and $\HH^1(A)$ is isomorphic to the quotient of the $\kk$-linear derivations of $A$ by the
inner derivations, while $\HH^2(A)$ provides a classification of infinitesimal deformations of the algebra~$A$.


The whole structure given by the Hochschild cohomology and homology together with the cup and cap products, the Gerstenhaber bracket and the Connes differential is called the Tamarkin Tsygan calculus.  It is invariant under derived equivalence \cite{AK:invariance-calculus} and --- if we can compute all these invariants --- provides a lot of information. The main problem is that these invariants are not easy to compute, since the cup
product and the Gerstenhaber bracket have been initially defined in terms of the bar resolution and on the other hand the Hochschild cohomology is seldom computed using the bar resolution.
The cup product, which coincides with the  Yoneda product of classes of extensions, is in general well understood, since it can be expressed in terms
of arbitrary resolutions via a diagonal map. But this is not  the case of the Gerstenhaber bracket. In spite of the interpretations given by Stasheff \cite{Stasheff} and by Schwede
\cite{Schwede},
it remains mysterious.
This bracket is --- despite its formula --- a geometrical object, as can be seen in the case of algebras of infinitely differentiable functions on a manifold, where it coincides with the Schouten and Nijenhuis bracket \cite{Kosmann-Schwarzbach}.  

The calculation of the whole Tamarkin Tsygan calculus is very difficult and generally not even possible for particular algebras. However, there exist some calculations for individual algebras, see for example \cite{AK:invariance-calculus}. The problem, in general, is that the minimal projective bimodule resolutions are difficult to find and even if one is able to compute such a resolution, it might be so complicated that the computation of the Tamarkin Tsygan calculus is not within reach. For monomial algebras the minimal projective bimodule resolution is known \cite{Bardzell} and in the case of quadratic monomial algebras it is simple enough, to embark on the extensive calculations of the Tamarkin Tsygan calculus. Yet even for quadratic monomial algebras, the combinatorial level of the calculations is such that it is too complicated to calculate the whole calculus. On the other hand for gentle algebras, the additional constraints on their structure are such that the calculations become  possible. Indeed, for these algebras it is possible to read this structure from the quiver, and thus  
also from the  oriented surface that can be associated to a gentle algebra.

In this article we study the homological invariants for gentle algebras that arise from Hochschild theory and the Tamarkin Tsygan calculus,
that is, Hochschild homology and cohomology with all their structure: the Gerstenhaber structure of the cohomology and the module structure of the homology as a module over the cohomology ring as well as the cyclic homology and the Connes differential.

Gentle algebras were introduced in the 80's by I. Assem, D. Happel and A. Skowronski \cite{Assem-Skowronski}. Since their inception they have been extensively studied from many different points of view.    There are many reasons for this interest. For example, their structure and representation theory is very amenable to computations. They are of tame representation type and their indecomposable representations have been classified in terms of string and band modules \cite{Wald-Waschbuesch}, \cite{BR}. Furthermore, morphism spaces of their indecomposable representations have been 
completely described  in 
\cite{CB, Krause}. Remarkably the class of gentle algebras is closed under derived equivalence \cite{Schroer-Zimmermann} and they are derived tame and the indecomposables in the derived category have been classified in terms of homotopy strings and bands \cite{BM}. Furthermore, a basis of the morphism spaces between indecomposable objects have been given \cite{Arnesen-Laking-Pauksztello}, and their mapping cones have been described \cite{Canakci-Paukzstello-Schroll}. However, gentle 
algebras are not only interesting in terms of their representation theory but they also have connections to many other areas of mathematics and physics. For example, they appear as Jacobian algebras of triangulations of surfaces \cite{ABCJP, LF1} and as such are 
instrumental in the categorification of cluster algebras \cite{Amiot}.  Generalising the geometric model for the module category of  Jacobian gentle algebras in \cite{ABCJP, LF1}, a geometric surface model for the module category of any gentle algebras has been given in \cite{BCS}.  A geometric surface model has also been developed for their bounded derived categories in \cite{OPS}. This is directly linked to another example of the ubiquity of gentle algebras, namely the recently established connection of graded gentle algebras and partially wrapped Fukaya categories in the work of \cite{HKK, LP} in the context of homological mirror symmetry where Fukaya categories are on the A side of the theory. The corresponding B side is also given in terms of gentle algebras in \cite{BD, LP}. Finally, a last connection we will give is the role that gentle algebras play in the context of  $A_1$ Gaiotto theories in  $4d$ $N=2$ gauge theories  in \cite{Cecotti}.


Previously, Ladkani \cite{Ladkani1} has computed the dimensions of the Hochschild cohomology spaces of gentle algebras and related them to the AAG-invariants. Redondo and Rom\'an \cite{Redondo-Roman} have also computed these dimensions and provided a description in terms of Bardzell's resolution \cite{Bardzell}. They also proved that the cup product of two classes of cocycles of odd degrees is zero, and that there are some non zero cup products of elements of even degrees. Moreover, they prove analogous results for the Gerstenhaber bracket, namely, that they proved that the brackets of two classes of cocycles of even positive degrees annihilate and that there are non zero brackets. Valdivieso \cite{Valdivieso} gave geometric interpretations for the dimension of the Hochschild cohomology space of gentle Jacobian algebras. In a series of recent papers, Bocklandt and van de Krekke, have studied deformations of a large family of gentle $A_\infty$-algebras as well as homological mirror symmetry in this context \cite{BvdK1, BvdK2, vdK1, vdK2, vdK3}. In particular, in \cite{vdK3} and \cite{BvdK1} for gentle  $A_\infty$-algebras arising from surface dimers, the authors calculate the Hochschild cohomology as well as the Gerstenhaber structure of these algebras. For the gentle algebras considered in these papers, all arrows are contained in cycles and they  correspond to  punctured surfaces with no boundary, or, in the language of \cite{HKK}, surfaces where the only boundary components are non-stopped. 

In addition to the calculation of the Tamarkin Tsygan calculus for any gentle algebra, we prove several results that show that it is possible to recover knowledge about the gentle algebra given its Tamarkin Tsygan calculus such as the finiteness of the global dimension and the dimension of the algebra.

\bigskip

The main
results of this paper are the following:
\begin{itemize}
\item Theorem \ref{thm:cohomology:basis}, which provides the description of the Hochschild cohomology of a gentle algebra as a vector space.
\item Theorem  \ref{thm:basishomology}, where the description of the Hochschild homology of a gentle algebra as a vector space is detailed.
\item  In Theorem \ref{thm:Connesspectral} we obtain the cyclic homology of gentle algebras.
\item Theorem \ref{thm:algebra} gives the algebra structure of Hochschild cohomology.
\item Theorem \ref{thm:cap}, which  provides the description of the cap product.
\item In Theorems \ref{thm:HH1Lie} and \ref{thm:Lie-cohomology}  we obtain the structure of the first cohomology space as a Lie algebra.
\item Finally, Theorems \ref{thm:geometric1} and \ref{thm:generators in gentle surface} provide 
a geometric description of the Hochschild cohomology and homology in terms of the surface associated to the bounded derived category of the gentle algebra.
\end{itemize}

\vspace{1cm}

{\bf Conventions:} Given a set of paths $X$, we will write $X_n$ the set of homogeneous paths of length $n$ in $X$, 
and $\abs{\hbox{ }}$ denotes the length of a path. 
Paths will be written from right to left, as compositions of functions.

\vspace{1cm}
\textbf{Acknowledgments:} C.C. thanks CONICET, Instituto de Investigaciones Ma\-te\-m\'aticas 
Luis A. Santal\'o (IMAS), and the University of Buenos Aires for the scholarship awarded to pursue a postdoctoral position (2020-2023). A substantial portion of this research was carried out during the period of the scholarship. S is partially supported by the DFG through the project SFB/TRR 191 Symplectic Structures in Geometry, Algebra and Dynamics (Projektnummer 281071066 - TRR 191).
A.S. and M.S.A. are invited professors at Guangdong Technion Israel Institute of Technology. She also thanks University of Cologne for an invitation during which she worked with S.S. on this project. A.S. and 
M.S.A are partially supported by PIP-CONICET 11220200101855CO.


\chapter{Gentle algebras and the Tamarkin-Tsygan calculus}
\label{chapter:algebras}

\section{Gentle algebras}

We start by recalling the definitions of gentle algebras and some of the
objects and notations associated to them.

\begin{definition}
If $Q=(Q_0,Q_1,s,t)$ is a quiver and~$I$ an ideal in the path algebra~$\kk
Q$, then we say that the pair $(Q,I)$\indexsymbol{$(Q,I)$} is a
\newterm{gentle presentation}\index{Gentle presentation} if
\begin{enumerate}[label=(\emph{\alph*})]

\item every vertex $v \in Q_0$ is the source of at most two arrows, and the
target of at most two arrows,

\item for every arrow $a \in Q_1$ there exists at most one arrow $b \in
Q_1$ such that $ab \notin I$, and at most one arrow $c \in Q_1$ such that
$ca \notin I$,

\item for every arrow $a \in Q_1$ there exists at most one arrow $b \in
Q_1$ such that $t(a) = s(b)$  and $ba \in I$, and at most one arrow $c \in
Q_1$ such that $t(c) = s(a)$ and $ac \in I$,

\item the ideal~$I$ is generated by paths of length~$2$.

\end{enumerate}
If additionally the ideal~$I$ is admissible, so that 
\begin{enumerate}[resume*]

\item there is a positive integer $N$ such that every path of length~$N$
in~$Q$ is in~$I$,

\end{enumerate}
then we say that the pair $(Q,I)$ is \newterm{finite dimensional
gentle}, or, for short \newterm{\fd gentle}.

An algebra $A$ is \newterm{gentle}\index{Gentle algebra} or \newterm{\fd
gentle} if it is Morita equivalent to a quotient of the form $\kk Q/I$ with
$(Q,I)$ a gentle or \fd gentle presentation, respectively.
\end{definition}

The representation theory of \fd gentle algebras is
well-understood. Their categories of modules have been described by Butler
and Ringel \cite{BR}, Crawley-Boevey \cite{CB} and Krause~\cite{Krause} in
terms of indecomposable modules, the maps between them, and also the almost
split short exact sequences connecting them, given in terms of the
so-called \emph{strings} and~\emph{bands} in the quiver. For their derived
categories, there are several descriptions  \cite{BM, BD}, of which the one
of the indecomposable objects due to Bekkert and Merklen, in terms of
\emph{homotopy strings} and \emph{homotopy bands}, is the one that we will
use. Moreover, the derived categories of \fd gentle algebras
are described by geometric models due to Opper, Plamondon and Schroll
\cite{OPS} ---see also \cite{HKK} and \cite{LP}--- based on the ribbon
surface attached to the algebra and the embedded ribbon graph initially
constructed by Schroll \cite{S}. We will briefly describe this model in
Chapter~\ref{chapter:geometric} and then use it to provide a geometric
interpretion of a generating system of the Hochschild cohomology 
of a gentle algebra as a graded-commutative algebra as well as a basis of
the Hochschild homology, and the cup and cap products.

\bigskip

We fix a gentle presentation $(Q,I)$ where $Q$ is a connected quiver. We
begin by computing the Hochschild cohomology of the algebra $A\coloneqq\kk
Q/I$ in a very explicit way that will allow us to determine the
Gerstenhaber algebra structure with which it is endowed and also to obtain
a geometric interpretation in terms of the ribbon surface of the
presentation. For this we fix a projective resolution  of $A$ as a bimodule
over itself, then use it to compute the cohomology in degrees~$0$ and~$1$,
and later in all higher degrees. While all calculations of the Hochschild
cohomology are similar, they do differ in low degrees and therefore we
treat these cases separately.

\bigskip

In the sections that follow we will compute various invariants of gentle
algebras, starting with Hochschild cohomology and homology. Most of our
calculations are made possible by the fact that gentle algebras are in
particular quadratic monomial algebras and that for such algebras there is
a  nice description of their minimal projective resolutions when viewed as
bimodules over themselves. We describe this resolution in detail and fix
the notation that we will use throughout when working. We do this more
generally in terms of quadratic monomial algebras, since this is going to
be useful when we talk about homology in Chapter~\ref{chapter:homology}.

\bigskip

We now fix a \newterm{quadratic monomial presentation}\index{Quadratic
monomial presentation} $(Q,I)$, so that $Q$~is a quiver and $I$~an ideal in
the path algebra~$\kk Q$ generated by a set~$R$ of paths of length~$2$, and
consider the algebra~$A=\kk Q/I$. We note that, for example, a
gentle presentation is of this form. We write~$\cB$ for
the set of paths in~$Q$ that do not have subpaths in~$R$, which is clearly
a basis for~$A$. On the other hand, we let $E$ be the subalgebra of~$A$
spanned by the (classes of) the trivial paths and for conciseness we write,
whenever useful, a bar~$|$ for the tensor product~$\otimes_E$ taken
over~$E$. If $P$ is a set of paths in~$Q$, then we write $\kk P$ for the
vector space freely spanned by~$P$ and always view it as an $E$-bimodule in
the obvious way. For each integer~$m\geq0$ we write $P_m$ for the subset of
those paths in~$P$ that have length exactly~$m$.

For a monomial algebra~$A$, Bardzell in~\cite{Bardzell} gives the
construction of a  projective resolution of~$A$ as an $A$-bimodule  from
the presentation~$(Q,I)$. Since the algebra we are considering is
quadratic, that projective resolution has a particularly simple form, which
we now describe. We let $\Gamma$\indexsymbol{$\Gamma$} be the set of those
paths in~$Q$ all of whose subpaths of length~$2$ are in~$R$.  So, in
particular, $\Gamma_0=Q_0$ and $\Gamma_1=Q_1$. The Bardzell resolution is
the chain complex~$\cR$\indexsymbol{$\cR$}
  \begin{multline}
  \begin{tikzcd}[column sep=1.25em, ampersand replacement=\&]
  \cdots \arrow[r]
    \& A\otimes_E \kk\Gamma_m\otimes_E A \arrow[r, "d_m"]
    \& A\otimes_E \kk\Gamma_{m-1}\otimes_E A \arrow[r]
    \& \cdots 
  \end{tikzcd}
  \\
  \begin{tikzcd}[column sep=1.5em, ampersand replacement=\&]
  \cdots \arrow[r]
    \& A\otimes_E \kk\Gamma_2\otimes_E A \arrow[r, "d_2"]
    \& A\otimes_E \kk\Gamma_1\otimes_E A \arrow[r, "d_1"]
    \& A\otimes_E \kk\Gamma_0\otimes_E A 
  \end{tikzcd}
  \end{multline}
with differentials given by
  \[
  d_m(1\otimes\gamma\otimes 1) 
        = a_m\otimes a_{m-1}\cdots a_1\otimes 1
          + (-1)^m\,1\otimes a_m\cdots a_2\otimes a_1
  \]
for all $m\geq1$ and all paths $\gamma=a_m\cdots a_1\in\Gamma_m$, and
augmentation 
  \[
  \mu:A\otimes_E\kk\Gamma_0\otimes_E A\to A
  \]
such that $\mu(1\otimes e\otimes 1)=e$ for all $e\in\Gamma_0$. It is easy
to check that this is indeed a complex of projective $A$-bimodules, and its
exactness is most conveniently proved by exhibiting a contracting homotopy,
as done by Sk\"oldberg in~\cite{Skoldberg}.

The quadratic monomial algebra~$A$ has a $\NN_0$-grading in which the
vertices and arrows of the quiver~$Q$ have degrees~$0$ and~$1$,
respectively, and with respect to this grading the algebra is Koszul, see
for example, Corollary~4.3 in \cite{PP} giving this in the case in which
there is only one vertex in~$Q$, the general case being an easy extension.
The Koszul dual algebra~$A^!$\indexsymbol{$A^"!$} has a presentation
$(Q,I^!)$ with the same quiver~$Q$ as that of~$A$ and ideal~$I^!$ generated
by all paths of length~$2$ in~$Q$ that are \emph{not} in~$I$. The Bardzell
resolution~$\cR$ that we described above is in fact the Koszul bimodule
resolution of~$A$ and, in a way, this explains why it has this simple form.

When the presentation~$(Q,I)$ is gentle, then so is the dual
presentation~$(Q,I^!)$, but in general the dual of a \fd gentle
presentation is not \fd gentle. Dualization interchanges the class of \fd
gentle presentations and that of gentle presentations of finite global
dimension, and therefore the class of \fd gentle presentations of finite
global dimension is closed under Koszul duality.

\section{Tamarkin-Tsygan calculus of an associative algebra}

The Tamarkin-Tsygan calculus of an associative algebra $A$ over a field $\kk$ is the comprehensive data of the Hochschild cohomology $\HH^*(A)$, the cup product $\smile$ and its Gerstenhaben bracket $[-,-]$, as well as the Hochschild homology $\HH_*(A)$
with the cap product $\frown$ and the Connes differential $B$. 
More precisely,  
the cup product
$$ \smile: \HH^p(A) \otimes_{\kk} \HH^q(A) \to \HH^{p+q}(A)$$ induces a graded commutative algebra structure on $\HH^*(A)$. The Gerstenhaber bracket
$$ [-,-]: \HH^p(A) \otimes_{\kk} \HH^q(A) \to \HH^{p+q-1}(A)$$ is such that $(\HH^*(A)[1], [-,-])$ is a graded Lie algebra. 
The Hochschild cohomology as graded commutative algebra acts on the Hochschild homology via the cap product
$$\frown:  \HH_p(A) \otimes_{\kk} \HH^q(A) \to \HH_{p-q}(A).$$  Finally, the Connes differential induces a structure of a complex of $k$-vector spaces on $\HH_*(A)$
$$B:  \HH_p(A) \to \HH_{p+1} (A)$$ with $B^2=0$ and satisfying the following identity  
$$ [B i_\alpha - (-1)^p
i_\alpha B, i_\beta] = i_{[\alpha, \beta]},$$
where $\alpha \in \HH^p(A), \beta \in \HH^q(A)$ and $z \in \HH_r(A)$ with $ i_\alpha: \HH_r(A) \to \HH_{r-p}(A)$ given by $ z \mapsto (-1)^{pr}
z\frown \alpha$. 

Thus the whole of the Tamarkin-Tsygan calculus is given by 
$$(\HH^*(A), \smile, [-,-], \HH_*(A), \frown, B).$$

While, in general, not many explicit calculations for the whole of the Tamarkin-Tsygan calculus are known, the strong combinatorial constraints for gentle algebras make such a calculation possible. Furthermore, the bounded derived category of a gentle algebra has a geometric surface model (see Chapter~\ref{chapter:geometric}) and we show that we can represent the entire Tamarkin-Tsygan calculus in terms of curves and operations on curves on the associated surface. An indication that this might be possible is based on the fact that the surface (genus, number of boundary components, number of marked points on the boundary components) is a derived invariant of gentle algebras \cite{OPS, APS, Op, HKK, LP} in combination with the derived invariance result of the Tamarkin-Tsygan calculus by Armenta and Keller.

\begin{theorem}\cite{AK:invariance-calculus}
The Tamarkin-Tsygan calculus of an associative $\kk$-algebra is invariant under derived equivalence. 
\end{theorem}



\chapter{Hochschild cohomology}
\label{chapter:cohomology}

We fix a gentle presentation~$(Q,I)$ and write $A\coloneqq\kk
Q/I$ for the algebra it presents. Our objective in this section is to
compute the Hochschild cohomology~$\HH^*(A)$ of~$A$. In
Chapter~\ref{chapter:algebras} we described a particular projective
resolution~$\cR$ of~$A$ as an $A$-bimodule, and we will realize~$\HH^*(A)$
as the cohomology of the complex~$\Hom_{A^e}(\cR,A)$. We will start by
giving a  description of this complex and introducing some terminology.

\bigskip

We say that two paths~$\alpha$ and~$\beta$ in~$Q$ are
\newterm{parallel}\index{Parallel paths} if $s(\alpha)=s(\beta)$
and~$t(\alpha)=t(\beta)$. If~$X$ and~$Y$ are sets of paths in~$Q$, then we
denote by~$X\parallel Y$ the set of all pairs~$(\alpha,\beta)$ of~$X\times
Y$ with~$\alpha$ and~$\beta$ parallel. Let~$\kk(X\parallel
Y)$\indexsymbol{$\kk(X\parallel Y)$} be the vector space it freely spans,
considered as an $E$-bimodule. For each $m\geq0$, we let $\cR_m
\coloneqq A\otimes_E\kk\Gamma_m \otimes_E A$ be the 
homogeneous component of degree~$m$ in the resolution~$\cR$
described in Chapter~\ref{chapter:algebras}.

\begin{proposition}\cite[Proposition 2.2.1.8]{Strametz:thesis}
\label{prop:strametz}
The graded vector space $\kk(\Gamma\parallel\cB)$ is a complex
with differentials
$d^m:\kk(\Gamma_m\parallel\cB)\to\kk(\Gamma_{m+1}\parallel\cB)$ given by 
  \begin{equation} \label{eq:differential}
  d^{m}(\gamma,\alpha) 
        = \sum_{\substack{b\in Q_1\\
                (b\gamma,b\alpha)\in\Gamma_{m+1}\parallel\cB}
               }
            (b\gamma,b\alpha)
        - (-1)^m 
          \sum_{\substack{a\in Q_1\\
                (\gamma a,\alpha a)\in\Gamma_{m+1}\parallel\cB}
                }
            (\gamma a,\alpha a)
  \end{equation}
for all $(\gamma,\alpha)\in\Gamma_m\parallel\cB$. \qed  
\end{proposition}

\bigskip

The collection of isomorphisms of vector spaces 
  \[
  \bigl(\Phi_m:\kk(\Gamma_m\parallel\cB) \to \Hom_{A^e}(\cR_m,A)\bigr)_{m\geq0} 
  \]
such that for all $(\alpha,\beta)\in\Gamma_m\parallel\cB$ and all
$\gamma\in\Gamma_m$ we have
  \[
  \Phi_m(\alpha,\beta)(1\otimes \gamma\otimes 1)
        = \begin{cases*}
          \beta & if $\gamma=\alpha$; \\
          0 & in any other case.
          \end{cases*}
  \]
induces an isomorphism of complexes
  \[
  \Phi:\kk(\Gamma\parallel\cB) \to \Hom_{A^e}(\cR,A).
  \]

The \newterm{weight}\index{Weight of a pair} of a pair~$(\gamma,\alpha)$
in~$\Gamma\parallel\cB$ is the integer~$\abs{\alpha}-\abs{\gamma}$. If
$\ell\in\ZZ$, we write $\kk(\Gamma\parallel\cB)_\ell$ for the subspace
of~$\kk(\Gamma\parallel\cB)$ spanned by the pairs of weight~$\ell$. We
clearly have a direct sum decomposition $\kk(\Gamma\parallel\cB) =
\bigoplus_{\ell\in\ZZ}\kk(\Gamma\parallel\cB)_\ell$, and it is preserved by
the differentials of the complex~$\kk(\Gamma\parallel\cB)$, so that this
complex is in fact one of graded vector spaces.

The \newterm{rank}\index{Rank of a (co)cycle} of a non-zero element~$u$
of~$\kk(\Gamma_m\parallel\cB)$ is the number~$\rk(u)$\indexsymbol{$\rk(u)$}
of elements of~$\Gamma_m\parallel\cB$ that appear in it with non-zero
coefficient. On the other hand, two elements in~$\kk(\Gamma_m\parallel\cB)$
are \newterm{disjoint}\index{Disjoint (co)cycles} if no element
of~$\Gamma_m\parallel\cB$ appears with non-zero coefficient in both, and a
cocycle in~$\kk(\Gamma_m\parallel\cB)$ is
\newterm{irreducible}\index{Irreducible (co)cycles} if it is not the sum of
two non-zero cocycles that are disjoint. Every cocycle in the
complex~$\kk(\Gamma_m\parallel\cB)$ is a sum of pairwise disjoint
irreducible cocycles, although possibly in several different ways, and
every irreducible cocycle is weight-homogeneous.

\bigskip

A path~$C=c_m\cdots c_1$ in~$Q$ is a \newterm{cycle}\index{Cycle} if its
length~$m$ is positive and it ends at the vertex where it begins, and in
that case it is \newterm{primitive}\index{Primitive cycle} if it is not a
proper power of another cycle. We put\indexsymbol{$\rot(C)$} 
  \[
  \rot(C) \coloneqq c_{m-1}\cdots c_1c_m, 
  \] 
the cycle obtained from~$C$ by `rotating it one step to the left'. The
following lemma connects these notions.

\begin{lemma}\label{lemma:primitive}
If $m\in\NN$ and $C$ is a cycle of length $m$ in~$Q$, then there are a
primitive cycle~$D$ and an integer~$k$, both uniquely determined by~$C$,
such that $C=D^k$. In fact, the set $\{i\in\NN:\rot^i(C)=C\}$ is not empty,
its smallest element~$l$ divides~$m$, the cycles
$C$,~$\rot(C)$,~\dots,~$\rot^{l-1}(C)$ are pairwise different, the
cycle~$D$ is the suffix of length~$l$ of~$C$ and $k=m/l$. \qed 
\end{lemma}

In the situation of this lemma, we call the length of the primitive
cycle~$D$ the \newterm{period}\index{Period of a cycle} of the cycle~$C$
with which we start. 

\bigskip

We say that two cycles in~$Q$ are \newterm{conjugate}\index{Conjugate
cycles} if one can be obtained from the other by repeated rotation. This is
clearly an equivalence relation on the set of all cycles in the quiver, and
we call its equivalence classes the \newterm{circuits}\index{Circuit}
of~$Q$ and write~$\C$\indexsymbol{$\C$} for the set of all circuits in~$Q$.
The \newterm{length}\index{Length of a circuit} of a circuit is the length
of any of the cycles it contains, and a circuit is
\newterm{primitive}\index{Primitive circuit} if the cycles it contains are
all primitive or, equivalently, if any one of them is.

\section{Cohomology in low degrees}
\label{sect:low-degrees}

In~\cite[Lemma 3.3 and Theorem 3.4]{CSS} we computed the Hochschild
cohomology spaces $\HH^0(A)$ and~$\HH^1(A)$ in the case where the
presentation~$(Q,I)$ is \fd gentle.  In this section we extend these results to
the gentle case, and reformulate them in order to align them
with the calculations in the later parts of this paper. 

\bigskip

A path~$\alpha$ is \newterm{$\cB$-maximal}\index{Bmaximal
path@$\cB$-maximal path} if it belongs to~$\cB$ and is not a proper subpath
of another element of~$\cB$; since the quiver is connected and not just a
vertex, such a path has positive length. A cycle~$\alpha$ in~$Q$ is
\newterm{cocomplete}\index{Cocomplete cycle} if $\alpha^2\in\cB$, in which
case we also have~$\alpha\in\cB$. A circuit is
\newterm{cocomplete}\index{Cocomplete circuit} if the cycles it contains
are cocomplete. If $C$ is a cycle in~$Q$, then either all the powers~$C^k$
with $k\geq1$ are cocomplete or none of them are --- in particular, if $D$
is the primitive cycle whose power~$C$ is, then $D$ is cocomplete exactly
when~$C$ is.

We let $\C(\cB)$\indexsymbol{$\C(\cB)$} be the set all cocomplete circuits
of~$(Q,I)$, and choose a set~$\Crep(\cB)$\indexsymbol{$\Crep(\cB)$} of
representatives of the elements of~$\C(\cB)$ as follows: we first choose a
representative for each primitive cocomplete circuit, and then let
$\Crep(\cB)$ be the set of all positive powers of those. In this way, the
set $\Crep(\cB)$ has the following useful property: if $\alpha$ and $\beta$
are two elements of~$\Crep(\cB)$ and $\alpha\beta$ is a cocomplete cycle
of~$(Q,I)$, then $\alpha\beta$ also belongs to~$\Crep(\cB)$.

If the presentation~$(Q,I)$ is \fd gentle, then there
are no cocomplete cycles in~$(Q,I)$ and the sets $\C(\cB)$ and~$\Crep(\cB)$
are both empty. With this set-up, we  describe the $0$th Hochschild
cohomology space of the gentle algebra~$A$:

\begin{proposition}\label{prop:hh:0}
The vector space~$\HH^0(A)$ is freely generated by the collection of the
following elements of~$\kk(\Gamma_0\parallel\cB)$:
\begin{itemize}

\item the element\indexsymbol{$\one$}
  \[
  \one\coloneqq\sum_{i\in Q_0}(e_i,e_i),
  \]

\item the pairs~$(s(\alpha),\alpha)$ in~$\Gamma_0\parallel\cB$ with
$\alpha$ a $\cB$-maximal path, and

\item the elements of the form\indexsymbol{$\cycle{\alpha}$}
  \[
  \cycle{\alpha} \coloneqq
        \sum_{i=0}^{r-1}
                (s(\rot^i(\alpha)),\rot^i(\alpha))
  \]
with $\alpha$ a cycle in~$\Crep(\cB)$ and $r$~its period.

\end{itemize}
\end{proposition}

\begin{proof}
A simple calculation shows that the elements listed in the statement of the
theorem are in the kernel of the differential~$d^0$ and, since they are
pairwise disjoint, it is clear that they are linearly independent. On the
other hand, every element of that kernel is a pairwise disjoint sum of
irreducible ones, so to prove the theorem it will be enough to show that
every irreducible element of~$\ker d^0$ is a scalar multiple of one of the
listed elements.

With that in mind, let us then consider an element $z$ of~$\ker d^0$ that
is irreducible and, in particular, weight-homogeneous of some
weight~$\ell\geq0$. If this weight is~$0$, then there are scalars $x_i$,
one for each vertex $i\in Q_0$, such that $z=\sum_{i\in Q_0}
x_i\cdot(e_i,e_i)$ and we have that
  \begin{align*}
  0 &= d^0(z) 
     =   \sum_{i\in Q_0}
         \sum_{\substack{b\in Q_1\\s(b)=i}}
            x_i\cdot(b,b)
       - \sum_{i\in Q_0}
         \sum_{\substack{a\in Q_1\\t(a)=i}}
            x_i\cdot(a,a)
     \\
    &= \sum_{b\in Q_1} (x_{s(b)}-x_{t(b)}) \cdot (b,b).
  \end{align*}
As the quiver~$Q$ is connected, this implies that there is a scalar
$x\in\kk$ such that $x=x_i$ for all vertices $i\in Q_0$ and therefore that
our cocycle~$z$ is a scalar multiple of the element~$\one$ described in the
proposition.

Let us suppose now that $\ell\geq1$. We can write
  \[
  z=x_1\cdot(e_1,\alpha_1)+\cdots+x_r\cdot(e_r,\alpha_r)
  \]
with $r\geq1$, the pairs $(e_1,\alpha_1)$,~\dots,~$(e_r,\alpha_r)$ all
in~$\Gamma_0\parallel\cB_\ell$ and pairwise different, and the scalars
$x_1$,~\dots,~$x_r\in\kk$ all non-zero. If there is an index
$i\in\{1,\dots,r\}$ such that the path~$\alpha_i$ is $\cB$-maximal, then
$d^0(e_i,\alpha_i)=0$ and, since $z$ is irreducible, we have that $z$~is in
fact a scalar multiple of~$(e_i,\alpha_i)$, which is one of the elements
listed in the proposition. We may therefore suppose that none of the
paths~$\alpha_1$,~\dots,~$\alpha_r$ appearing in~$z$ is $\cB$-maximal. 

Let $i\in\{1,\dots,r\}$ and suppose that there is an arrow~$b$ in~$Q$ such
that $b\alpha_i\in\cB$. As~$d^0(z)=0$, it follows from this that there is
an index $j\in\{1,\dots,r\}$ and an arrow~$a$ in~$Q$ such that
$(b,b\alpha_i)=(a,\alpha_ja)$, and then $a=b$ and $b$ is the first arrow
in~$\alpha_i$: as $b\alpha_i\in\cB$, we see that $\alpha_i$ is cocomplete.
A similar argument applies if there is an arrow~$a$ in~$Q$ with
$\alpha_ia\in\cB$, of course, and this shows that the
cycles~$\alpha_1$,~\dots,~$\alpha_r$ are all cocomplete.

For each $i\in\{1,\dots,r\}$ let $a_i$ and~$b_i$ be the last and the first
arrow in~$\alpha_i$, respectively. Because the presentation~$(Q,I)$ is
gentle and the paths~$\alpha_1$,~\dots,~$\alpha_r$ all have
length~$\ell \geq1$, we have that
  \[
  0 = d^0(z) = \sum_{i=1}^rx_i\cdot(b_i,b_i\alpha_i)
             - \sum_{i=1}^rx_i\cdot(a_i,\alpha_i a_i).
  \]
Moreover, as $(Q,I)$ is gentle and the
paths~$\alpha_1$,~\dots,~$\alpha_r$ are pairwise different and all of the
same positive length, we see that the arrows~$a_1$,~\dots,~$a_r$ are
pairwise different, as are the arrows~$b_1$,~\dots,~$b_r$, and this implies
that there is a permutation~$\pi$ of~$\{1,\dots,r\}$ such that
$(a_i,\alpha_ia_i)=(b_{\pi(i)},b_{\pi(i)}\alpha_{\pi(i)})$ and
$x_i=x_{\pi(i)}$ for all $i\in\{1,\dots,r\}$. In particular, we see that
$\alpha_{\pi(i)}=\rot(\alpha_i)$ for all $i\in\{1,\dots,r\}$.

Let now $t$ be the period of~$\alpha_1$, so that the
paths~$\alpha_1$,~$\rot(\alpha_1)$,~\dots,~$\rot^{t-1}(\alpha_1)$ are
pairwise different and $\rot^t(\alpha_1)=\alpha_1$. The integers
$1$,~$\pi(1)$,~\dots,~$\pi^{t-1}(1)$ are then pairwise different, there is
a unique $\alpha\in\Crep(\cB)$ that is conjugate to~$\alpha_1$, and we have
that
  \[
  \sum_{i=0}^{t-1}
        x_{\pi^i(1)}\cdot(e_{\pi^i(1)},\alpha_{\pi^i(1)})
    = x_1\cdot\cycle{\alpha},
  \]
which is in the kernel of~$d^0$: as $z$ is irreducible, this implies that
this sum is in fact equal to~$z$. The proof of the theorem is thus
complete.
\end{proof}

We now give a description of~$\HH^1(A)$ and, in order to that, we start by
giving a description of the $1$-cocycles that have rank~$1$.
Just as for $\cB$-maximal paths, we say that a path~$\gamma$ is
\newterm{$\Gamma$-maximal}\index{Gmaximal@$\Gamma$-maximal} if it belongs
to~$\Gamma$ and is not a proper subpath of another element of~$\Gamma$. We
remark that, since $Q$ is connected, a path is simultaneously $\cB$- and
$\Gamma$-maximal if and only if it has length~$1$ and goes from a source to
a sink.

\begin{lemma}\label{lemma:one:rank-1}
Let $u$ be an element of~$\Gamma_1\parallel\cB$. If $u$ is cocycle and not
a coboundary, then one of the following conditions holds:
\begin{enumerate}[label=(\roman*), ref=(\roman*)]

\item\label{it:one:loop} there is a loop $b$ in~$Q$ such that $b^2\in R$
and $u=(b,s(b))$, and the characteristic of the field~$\kk$ is~$2$,

\item\label{it:one:max} there is a $\Gamma$-maximal arrow~$c$ and a
path~$\alpha\in\cB$ that neither begins nor ends with~$c$ such that
$u=(c,\alpha)$,

\item\label{it:one:cdelta} there is a cocomplete cycle~$\delta$ starting
with an arrow~$c$ such that $u=(c,c\delta)$,

\item\label{it:one:diag} there is an arrow~$c$ such that $u=(c,c)$.

\end{enumerate}
On the other hand, if $u$ satisfies one of these conditions then it is a
cocycle.
\end{lemma}

We note that it might be the case that $u$ satisfies one of these four
conditions and is a coboundary.

\begin{proof}
The second claim of the lemma follows from a direct computation that we
omit. Let us prove the first one. Let $u=(c,\alpha)$ be an element
of~$\Gamma_1\parallel\cB$ that is a cocycle and not a coboundary. Since
$d^1(u)=0$  one of the following two possibilities occurs:
\begin{enumerate}[label=(\Alph*), ref=(\Alph*)]

\item\label{it:one:1} either the characteristic of~$\kk$ is~$2$, and there
are arrows~$a$ and~$b$ in~$Q$ such that 
  \begin{equation}\label{eq:one:1}
  (bc,b\alpha)
        = (ca,\alpha a)
        \in \Gamma_{2}\parallel\cB,
  \end{equation}

\item\label{it:one:2} or there is no arrow~$b$ such that
$(bc,b\alpha)\in\Gamma_{2}\parallel\cB$ and there is no arrow $a$ such that
$(ca,\alpha a)\in\Gamma_{2}\parallel\cB$.

\end{enumerate}
If \ref{it:one:1} holds, then $b=c=a$, $b$ is a loop, $b^2\in R$, and
$\alpha=b^l$  for some integer $l\geq0$: as~$b^2\in R$ and $\alpha\in\cB$,
we have in fact that $l\leq1$, and one of conditions~\ref{it:one:loop}
or~\ref{it:one:diag} of the lemma holds.

Let us now suppose that \ref{it:one:2} holds. If $\alpha$ has length~$0$,
then \ref{it:one:2} means that the arrow~$c$ is $\Gamma$-maximal, so that
condition~\ref{it:one:max} holds. If $\alpha$ has length~$1$, then either
$\alpha=c$ and condition~\ref{it:one:diag} holds, or $\alpha$ is an arrow
different from~$c$, and in that case the gentleness of~$(Q,I)$
together with~\ref{it:one:2} implies at once that the arrow~$c$ is
$\Gamma$-maximal, so that condition~\ref{it:one:max} holds. We are thus
left with considering the case in which the path~$\alpha$ has length at
least~$2$.

Suppose that $\alpha$ ends with the arrow~$c$, so that there is a cycle
$\delta\in\cB$ of positive length and such that $\alpha=c\delta$. There is
an arrow~$a$ such that $\delta a\in\cB$, for otherwise
$d^0(s(\delta),\delta)=u$, and $u$ is not a coboundary. Now $\alpha
a=c\delta a$ is in~$\cB$. Since $c\delta$ and $\delta a$ are and $\delta$
has positive length, and then \ref{it:one:2} implies then that $ca\not\in
R$: because $\delta a\in\cB$, the gentleness of~$(Q,I)$ then allows
us to deduce that $c$ is the first arrow of~$\delta$ and, because
$c\delta\in\cB$, that $\delta$ is a cocomplete cycle: as $u=(c,c\delta)$,
we thus see that condition~\ref{it:one:cdelta} holds.

Suppose next that the path~$\alpha$ now starts with the arrow~$c$.
Reasoning as we have just done we can show that there is a cocomplete
cycle~$\epsilon$ ending with~$c$ such that $u=(c,\epsilon c)$, but then
putting $\delta\coloneqq\rot(\epsilon)$, which is a cocomplete cycle
starting with~$c$, we see that $u=(c,c\delta)$ and, thus, that
condition~\ref{it:one:cdelta} holds also in this case.

Finally, if the path~$\alpha$ neither starts nor ends with the arrow~$c$
then \ref{it:one:2} and the gentleness of~$(Q,I)$ imply at once that the
arrow~$c$ is $\Gamma$-maximal and then that condition~\ref{it:one:max}
holds.
\end{proof}

Knowing all the $1$-cocycles of rank~$1$, we can now complete the
description of the first cohomology space~$\HH^1(A)$. As usual, a
\newterm{spanning tree}\index{Spanning tree} of the quiver~$Q$ is the set
of arrows of a connected, acyclic subquiver of~$Q$ that contains all its
vertices.

\begin{proposition}\label{prop:hh:1}
Let $T$ be a spanning tree of the quiver~$Q$. The vector space $\HH^1(A)$
is freely generated by the collection of the cohomology classes of the
following elements of~$\Gamma_1\parallel\cB$:
\begin{itemize}

\item the pairs of the form $(c,c)$ with $c\in Q_1\setminus T$,

\item the pairs of the form $(c,\alpha)$ with $c$ a $\Gamma$-maximal arrow
and $\alpha\in\cB$ a path that neither begins not ends with~$c$, 

\item the pairs $(c,c\delta)$ with $\delta\in\Crep(\cB)$ and $c$ the first
arrow in~$\delta$,

\item the pairs $(b,s(b))$ with $b$ a loop in~$Q$ such that $b^2\in R$, if
the characteristic of~$\kk$ is~$2$.

\end{itemize}
\end{proposition}

\begin{proof}
Every $1$-cocycle in our complex~$\kk(\Gamma\parallel\cB)$ is cohomologous
to a pairwise disjoint sum of irreducible cocycles that are not
coboundaries. Lemma~\ref{lemma:one:rank-1} describes those of these that
have rank~$1$, and we will now show that there are none of rank greater
than~$1$. The proposition will then immediately follow.

Let $z$ be an irreducible $1$-cocycle in our
complex~$\kk(\Gamma\parallel\cB)$ that is not a coboundary, and to reach a
contradiction let us suppose that its rank $r\coloneqq\rk(z)$ is greater
than~$1$. Since $z$ is irreducible, it is weight-homogeneous: let $\ell$ be
its weight, which is at least~$-1$. We can write
  \[
  z
        = x_1\cdot(c_1,\alpha_1) + \cdots + x_r\cdot(c_r,\alpha_r)
  \]
for some pairwise different pairs $(c_1,\alpha_1)$,~\dots,~$(c_r,\alpha_r)$
in~$\Gamma_1\parallel\cB_{\ell+1}$ and some non-zero scalars
$x_1$,~\dots,~$x_r\in\kk$, and compute that
  \begin{align}
  d^1(z) 
       &= \sum_{\substack{1\leq i\leq r,\;b\in Q_1\\
                bc_i\in R,\;b\alpha_i\in\cB}}
            x_i\cdot(bc_i,b\alpha_i)
        + \sum_{\substack{1\leq i\leq r,\;a\in Q_1\\
                c_ia\in R,\;\alpha_ia\in\cB}}
            x_i\cdot(c_ia,\alpha_ia) 
        \notag
        \\
       &= \sum_{b\in Q_1}
          \Biggl[
            \sum_{\substack{1\leq i\leq r\\bc_i\in R,\;b\alpha_i\in\cB}}
                x_i\cdot(bc_i,b\alpha_i)
          \Biggr]
        + \sum_{a\in Q_1}
          \Biggl[
            \sum_{\substack{1\leq i\leq r\\c_ia\in R,\;\alpha_ia\in\cB}}
                x_i\cdot(c_ia,\alpha_ia)
          \Biggr]
        , \label{eq:hh1:1}
  \end{align}
with each of the sums in brackets having at most one summand because of the
gentleness of the presentation~$(Q,I)$ and the fact that the
paths~$\alpha_1$,~\dots,~$\alpha_r$ all have the same length~$\ell+1$. As
$z$ is irreducible and we are supposing that $r>1$, the pair
$(c_1,\alpha_1)$ is not a cocycle and therefore at least one of the
following two conditions holds:
\begin{enumerate}[label=(\Alph*), ref=(\Alph*)]

\item\label{it:reg:1} there exists an arrow~$b$ such that
$(bc_1,b\alpha_1)\in\Gamma_2\parallel\cB$,

\item\label{it:reg:2} there exists an arrow~$a$ such that
$(c_1a,\alpha_1a)\in\Gamma_2\parallel\cB$.

\end{enumerate}
Let us suppose that the first one holds, so that the pair
$(bc_1,b\alpha_1)$ appears in the first sum of~\eqref{eq:hh1:1} with
coefficient~$x_1$: as $d^1(z)=0$, we see that that pair has to appear also
in the second sum with opposite coefficient, and that then there is an
arrow~$a$ and an index $i\in\{1,\dots,r\}$ such that   
  \begin{equation}\label{eq:hh1:2}
  x_1+x_i=0,
  \qquad \qquad
  (bc_1,b\alpha_1)=(c_ia,\alpha_ia). 
  \end{equation}
If we had $i=1$, then these equalities imply that $2x_1=0$, so that the
characteristic of~$\kk$ is~$2$, and that $a=c_1=b$, that $b$ is a loop with
$b^2\in R$, and that $\alpha_1=b^r$ with $r\in\{0,1\}$: but then
$d^1(c_1,\alpha_1)=0$, which we know not to be the case. We thus see that
$i>1$.

If $\ell=-1$, then using the second equality in~\eqref{eq:hh1:2} we find
that $a=b$, that $b$ is a loop, and therefore that
$(c_1,\alpha_1)=(c_i,\alpha_i)$, which is impossible since $i\neq1$. If
$\ell=0$, then the second equality in~\eqref{eq:hh1:2} implies that
$(c_1,\alpha_1)=(a,a)$ and, in particular, that $(c_1,\alpha_1)$ is a
cocycle: this is again impossible. We therefore have that $\ell\geq1$, so
that $\alpha_1$ has length at least~$2$. Using that and~\eqref{eq:hh1:2} we
see that there is a cycle~$\delta\in\cB$ of positive length such that
$\alpha_1=\delta a$ and $\alpha_i=b\delta$, and therefore
  \[
  d^0\bigl(x_i\cdot (s(\delta),\delta)\bigr) 
        = x_i\cdot(b,b\delta) - x_i\cdot(a,\delta a)
        = x_i\cdot(c_i,\alpha_i) + x_1\cdot(c_1,\alpha_1),
  \]
so that $r=i=2$ and in fact $d^0\bigl(x_i\cdot(s(\delta),\delta)\bigr)=z$:
this is absurd.

This shows that the condition~\ref{it:reg:1} cannot actually hold, and a
similar reasoning shows that the same happens with
condition~\ref{it:reg:2}. We can conclude from this that the hypothesis
that the rank~$r$ of~$z$ is larger than~$1$ is untenable, as we wanted.

At this point, we know that every $1$-cocycle is cohomologous to a linear
combination of elements of~$\Gamma_1\parallel\cB$ that satisfy one of the
conditions of Lemma~\ref{lemma:one:rank-1} and, of course, the set of those
cocycles forms a basis of the space~$Z$ they span. Let now $z$ be a
non-zero coboundary belonging to~$Z$ that is weight-homogeneous of some
weight~$\ell$. Because of the form of the differential~$d^0$, it is clear
that no element of~$\Gamma_1\parallel\cB$ satisfying one of the
conditions~\ref{it:one:loop} or~\ref{it:one:max} of
Lemma~\ref{lemma:one:rank-1} appears in~$z$ with non-zero coefficient and
that, in particular, $\ell\geq0$. 
\begin{itemize}

\item If $\ell>0$, then $z$ is a linear combination of pairs of the
form~$(c,c\delta)$ with $\delta$ a cocomplete cycle and $c$ the first arrow
of~$\delta$. A pair of that form appears in the coboundary of a pair
$(e,\eta)$ in~$\Gamma_0\parallel\cB$ if and only if on one hand $e=s(\eta)$
and, on the other, $\eta=\delta$ or $\eta=\rot^{-1}(\delta)$, and when this
is the case all the pairs that appear in~$d^0(e,\eta)$ are of the form of
those in~$z$. It follows from this that $z$ is a linear combination of the
elements
  \[
  d^0(s(\epsilon),\epsilon) 
        = (b,b\epsilon) - (a,a\rot(\epsilon))
  \]
with $\epsilon$ a cocomplete path, $b$ its first arrow and $a$ the first
arrow of~$\rot(\epsilon)$, which is the last arrow of~$\epsilon$.

\item If the weight~$\ell$ of~$z$ is~$0$, then $z$ is in fact a linear
combination of elements of~$\Gamma_1\parallel\cB$ of the form~$(c,c)$ with
$c$ an arrow of~$Q$, and therefore a linear combination of the elements
  \[
  d^0(e_i,e_i) 
        = \sum_{\substack{b\in Q_1\\s(b)=i}}(b,b)
        - \sum_{\substack{a\in Q_1\\t(a)=i}}(a,a)
  \]
with $i\in Q_1$. It is well-known then that $z$ is cohomologous to a unique
linear combination of pairs $(c,c)$ with $c\in Q_1\setminus T$; the details
of the argument to check this can be found in the proof of Theorem~3.4
of~\cite{CSS}.

\end{itemize}
Putting everything together, we see at once that the claim of the
proposition is true.
\end{proof}

Counting dimensions shows that this description is compatible with the results of \cite{Cibils-Saorin}.

\section{Cohomology in higher degrees}

We now compute the cohomology of our algebra~$A$ in degrees greater
than~$1$. We will use several times the following key remark that is a
consequence of the form of the differential of our
complex~$\kk(\Gamma\parallel\cB)$.

\begin{remark}\label{rem:key}
If a pair~$(\gamma,\alpha)\in\Gamma_m\parallel\cB$ is such that  at least one
of~$\gamma$ or~$\alpha$ has positive length, then there is at most one
arrow~$b$ such that $(b\gamma,b\alpha)$ is in~$\Gamma\parallel\cB$, and at
most one arrow~$a$ such that $(\gamma a,\alpha a)$ is
in~$\Gamma\parallel\cB$. This follows from the fact that the
presentation~$(Q,I)$ is gentle, and implies that in each of the
sums that appear in the right hand side of the
equality~\eqref{eq:differential} of Proposition~\ref{prop:strametz} there
is at most one summand. Moreover, these sums have no terms in common unless
there is a loop~$b$ in~$Q$ and an integer~$l\geq0$ such that $\gamma=b^m$
and~$\alpha=b^l$, and when this is the case then
$d^m(\gamma,\alpha)=(1-(-1)^m)(b^{m+1},b^{l+1})$. 
\end{remark}

A cycle $C$ in~$Q$ is \newterm{complete}\index{Complete cycle} in~$(Q,I)$
if $C^2$ is in~$\Gamma$, and in that case $C$ is also there. A
straightforward calculation using the fact that the presentation~$(Q,I)$ is
gentle proves the following observation:

\begin{lemma}\label{lemma:sum}
Let $C$ be a complete cycle of length~$m$ and period~$r$ in~$(Q,I)$, and
let $b$ be the first arrow in~$C$. If we set
  \[
  \cycle{C} 
        \coloneqq \sum_{i=0}^{r-1} (-1)^{im}\cdot(\rot^i(C),s(\rot^i(C)))
        \in \kk(\Gamma_{m}\parallel\cB),
  \]
then we have that
  \[
  d^{m}(\cycle{C}) = (1-(-1)^{m})\cdot(bC,b).
  \]
In particular, the cochain~$\cycle{C}$ is an $m$-cocycle in the
complex~$\kk(\Gamma\parallel\cB)$ if and only if either its degree~$m$ is
even or the characteristic of the field~$\kk$ is~$2$. \qed
\end{lemma}

This lemma allows us to produce a cochain from a complete cycle in the
bound quiver. Up to scalars, this cochain only depends on the circuit that
contains that cycle, as the following lemma shows.

\begin{lemma}\label{lemma:sum:span}\mbox{}
\begin{thmlist}

\item If $C$ is a complete cycle of length~$m$ in~$(Q,I)$ and $i\in\ZZ$,
then 
  \[
  \cycle{\rot^i{C}} = (-1)^{im}\cdot \cycle{C}.
  \]

\item If $C$ and $D$ are two complete cycles of length~$m$ in~$(Q,I)$ that belong
to the same circuit, then the cochains $\cycle{C}$ and $\cycle{D}$ span the
same $1$-dimensional subspace of~$\kk(\Gamma_m\parallel\cB)$. \qed

\end{thmlist}
\end{lemma}

To describe the cocycles of higher degree in our complex we will follow the
same course of action as in the case of those of lower degree: every
$m$-cocycle in the complex~$\kk(\Gamma\parallel\cB)$ is a sum of pairwise
disjoint irreducible cocycles, and we will describe these according to
their rank. For variety, we start this time with those of rank larger
than~$1$:

\begin{proposition}\label{prop:cocycles:rank-r}
Let $m\geq2$. If $u$ is an irreducible cocycle
in~$\kk(\Gamma_m\parallel\cB)$ that has rank $r\coloneqq \rk(u)>1$ and is
not a coboundary, then 
\begin{thmlist}

\item the degree~$m$ of~$u$ is even or the characteristic of~$\kk$ is~$2$,
and

\item that degree $m$ is divisible by~$r$ and there is a complete cycle~$C$
of length~$m$ and period~$r$ in~$(Q,I)$ such that $u$ is a non-zero scalar
multiple of~$\cycle{C}$.

\end{thmlist}
\end{proposition}

\begin{proof} 
Let $u$ be an irreducible cocycle with rank $r\coloneqq\rk(u)$ greater
than~$1$ that is not a coboundary. We can write
$u=x_1\cdot(\gamma_1,\alpha_1)+\cdots+x_r\cdot(\gamma_r,\alpha_r)$ with the
pairs
  \begin{equation}\label{eq:ps}
  (\gamma_1,\alpha_1),\quad \dots, \quad (\gamma_r,\alpha_r)
  \end{equation}
all in~$\Gamma_m\parallel\cB$ and pairwise different, and scalars
$x_1$,~\dots,~$x_r\in\kk$ all non-zero. 

Let us start by showing that 
  \begin{equation}\label{eq:not-loop}
  \claim[-0.8]{for all $i\in\{1,\dots,r\}$ the path $\gamma_i$ is not a
  power of a loop.} 
  \end{equation}
To see this, let us suppose that there is an $i\in\{1,\dots,r\}$ and a
loop~$b$ in~$Q$ such that $\gamma_i=b^m$. Since $u$~is irreducible, we have
$d^m(\gamma_i,\alpha_i)\neq0$, and the gentleness of~$(Q,I)$ implies
that one of the pairs $(b\gamma_i,b\alpha_i)$ or~$(\gamma_ib,\alpha_ib)$
appears in~$d^m(\gamma_i,\alpha_i)$, and up to symmetry we can suppose that
the first one does. 

As $d^m(u)=0$, we see that there is a $j\in\{1,\dots,r\}\setminus\{i\}$
such that the pair $(b\gamma_i,b\alpha_i)$ also appears in
$d^m(\gamma_j,\alpha_j)$ and, in fact, because the pairs in~\eqref{eq:ps}
are pairwise different, a unique one. This implies that there is an
arrow~$a$ such that $(b\gamma_i,b\alpha_i)=(\gamma_ja,\alpha_ja)$, and then
$\gamma_j=b^m$, $a=b$ and $b\alpha_i=\alpha_jb$. If the length
of~$\alpha_i$ were~$0$ or~$1$, then we would have that
$\alpha_i=\alpha_j=s(b)$ or that $\alpha_i=\alpha_j=b$, respectively, so
that in both cases $(\gamma_i,\alpha_i)=(\gamma_j,\alpha_j)$, which is
a contradiction. The path~$\alpha_i$ thus has length at least~$2$ and there is a
path~$\zeta$ of positive length such that $\alpha_i=\zeta b$ and
$\alpha_j=b\zeta$. We then have $d^m(\gamma_i,\alpha_i)=(b^{m+1},b\zeta
b)$ and $d^m(\gamma_j,\alpha_j)=(-1)^{m+1}\cdot(b^{m+1},b\zeta b)$, and,
because the pair $(b^{m+1},b\zeta b)$ does not appear in
$d^m(\gamma_k,\alpha_k)$ if $k\not\in\{i,j\}$, the fact that $d^m(u)=0$
implies that $x_i=(-1)^{m}x_j$. Now
  \begin{align*}
  d^{m-1}\bigl(x_j\cdot(b^{m-1},\zeta)\bigr)
       &= x_j\cdot (b^m,b\zeta) - (-1)^{m-1}x_j\cdot (b^m,\zeta b) \\
       &= x_j\cdot (\gamma_j,\alpha_j) + x_i\cdot (\gamma_i,\alpha_i).
  \end{align*}
As $u$ is an irreducible cocycle, this tells us that $r=2$, that
$\{i,j\}=\{1,2\}$, and that $u$ is a coboundary, which is a contradiction and hence our claim~\eqref{eq:not-loop} is proved.

Let us next show that
  \begin{equation}\label{eq:zero}
  \claim[-0.8]{the paths $\alpha_1$,~\dots,~$\alpha_r$ all have length~$0$.}
  \end{equation}
To do so, let us suppose that there is an $i\in\{1,\dots,r\}$ such that the
path~$\alpha_i$ has positive length. Since $u$ is irreducible and~$r>1$, we
have that $d^m(\gamma_i,\alpha_i)\neq0$ and then, up to the obvious
left-right symmetry, we can suppose that there is an arrow~$b$ such that
the pair $(b\gamma_i,b\alpha_i)$ is in~$\Gamma_{m+1}\parallel\cB$ and
appears in~$d^m(\gamma_i,\alpha_i)$ with non-zero coefficient.
As~$d^m(u)=0$, there is a $j\in\{1,\dots,r\}\setminus\{i\}$ such that the
pair $(b\gamma_i,b\alpha_i)$ also appears with non-zero coefficient in
$d^m(\gamma_j,\alpha_j)$. The pairs in~\eqref{eq:ps} are pairwise
different, so there is in fact exactly one such~$j$, and there is  an
arrow~$a$ such that 
  \begin{equation}\label{eq:qs}
  (b\gamma_i,b\alpha_i) = (\gamma_ja,\alpha_ja).
  \end{equation}
As $\gamma_i$ and $\gamma_j$ are not powers of loops, it follows now from
Remark~\ref{rem:key} that the coefficients with which the
pairs~$(b\gamma_i,b\alpha_i)$ and $(\gamma_ja,\alpha_ja)$ appear
in~$d^m(\gamma_i,\alpha_i)$ and in~$d^m(\gamma_j,\alpha_j)$ are $1$
and~$(-1)^{m+1}$, respectively, and then, since $(b\gamma_i,b\alpha_i)$
certainly does not appear in~$d^m(u)$, we have that
  \begin{equation}\label{eq:rs}
  x_i + (-1)^{m+1}x_j = 0.
  \end{equation}
On the other hand, since $m\geq2$, from the equality~\eqref{eq:qs} we see
that there exist $\delta\in\Gamma_{m-1}$ and $\zeta\in\cB$ such that
$\gamma_i=\delta a$, $\gamma_j=b\delta$, $\alpha_i=\zeta a$ and
$\alpha_j=b\zeta$, and the gentleness of the presentation~$(Q,I)$
implies that
  \[
  d^{m-1}(\delta,\zeta) 
        = (b\delta,b\zeta) - (-1)^{m-1}\cdot(\delta a,\zeta a)
        = (\gamma_i,\alpha_i) - (-1)^{m-1}\cdot(\gamma_j,\alpha_j).
  \]
Together with~\eqref{eq:rs}, this implies that $v\coloneqq
x_i\cdot(\gamma_i,\alpha_i)+x_j\cdot(\gamma_j,\alpha_j)$ is a coboundary and,
\emph{a~fortiori}, a cocycle. As $u$ is irreducible, we must have~$u=v$,
and this is a contradiction since the cocycle~$u$ with which we started is not a
coboundary. This contradiction proves our claim~\eqref{eq:zero}. It follows that $\alpha_i=s(\gamma_i)$ for all
$i\in\{1,\dots,r\}$.

The next observation that we need to make is that
  \begin{equation}
  \claim[-0.8]{for each $i\in\{1,\dots,r\}$ the path~$\gamma_i$ is a
  complete cycle.}
  \end{equation}
Let $i\in\{1,\dots,r\}$. As $u$ is an irreducible cocycle, we have that
$d^m(\gamma_i,\alpha_i)\neq0$, and therefore there exists an arrow~$b$ such
that $(b\gamma_i,b)$ appears in~$d^m(\gamma_i,\alpha_i)$, or there exists
an arrow~$a$ such that $(\gamma_ia,a)$ appears
in~$d^m(\gamma_i,\alpha_i)$, or, of course, both, and by symmetry we may
just as well suppose the first of these possibilities occurs. Now
$d^m(u)=0$, so there is a $j\in\{1,\dots,r\}\setminus\{i\}$ such that the
pair~$(b\gamma_i,b)$ appears in $d^m(\gamma_j,\alpha_j)$. Since $i\neq
j$ and the pairs listed in~\eqref{eq:ps} are pairwise different this
implies that there is  an arrow~$c$ such that
$(b\gamma_i,b)=(\gamma_jc,c)$. From this equality we deduce that $c=b$,
that $c$ is the first arrow in~$\gamma_i$ and that
$c\gamma_i\in\Gamma_{m+1}$. Hence $\gamma_i$ is a complete cycle, as we
want.

As all the paths $\gamma_1$,~\dots,~$\gamma_r$ are complete cycles, and all
the paths~$\alpha_1$,~\dots,~$\alpha_r$ have length zero,  using the
gentleness of~$(Q,I)$ we see that
  \[
  0 
        = d^m(u) 
        = \sum_{i=1}^r
           \Bigl(
             x_i\cdot(b_i\gamma_i,b_i) -(-1)^mx_i\cdot(\gamma_i a_i,a_i)
           \Bigr),
  \]
with $b_1$,~\dots,~$b_r$ the first arrows of the
paths~$\gamma_1$,~\dots,~$\gamma_r$, and $a_1$,~\dots,~$a_r$ their last
arrows. Since the pairs in~\eqref{eq:ps} are pairwise different, this
implies that there is a unique permutation~$\pi$ of~$\{1,\dots,r\}$ such
that 
  \begin{gather}
  (b_{\pi(i)}\gamma_{\pi(i)},b_{\pi(i)}) = (\gamma_{i}a_{i},a_{i})
        \label{eq:as}
\shortintertext{and}
  x_{\pi(i)} = (-1)^mx_{i}
        \label{eq:bs}
  \end{gather}
for all $i\in\{1,\dots,r\}$. Let $l\coloneqq\min\{i\in\NN_0:\pi^i(1)=1\}$.
The $l$ numbers 
  \[
  1,~\pi(1),~\dots,~\pi^{l-1}(1) 
  \]
are pairwise different and a direct calculation using~\eqref{eq:as}
and~\eqref{eq:bs} shows that 
  \begin{equation}\label{eq:xs}
  \sum_{i=0}^{l-1} x_{\pi^i(1)}\cdot (\gamma_{\pi^i(1)},\alpha_{\pi^i(1)})
  \end{equation}
is a cocycle: as $u$ is irreducible, this sum must be equal to~$u$ and
$l$~equal to~$r$.

From~\eqref{eq:as} we see that $\gamma_{\pi(i)}=\rot(\gamma_i)$ for all
$i\in\{1,\dots,r\}$ and from this that 
  \begin{equation}\label{eq:cs}
  \gamma_{\pi^i(1)} = \rot^i(\gamma_1)
  \end{equation}
for all $i\in\NN_0$. As $r=l$, the paths
$\gamma_1$,~$\rot(\gamma_1)$,~\dots,~$\rot^{r-1}(\gamma_1)$ are pairwise
different and $\rot^r(\gamma_1)=\gamma_1$, so we see that the complete
cycle~$C\coloneqq\gamma_1$ has length~$m$ and period~$r$. On the other
hand, from~\eqref{eq:bs} we see that 
  \begin{equation}\label{eq:ds}
  x_{\pi^i(1)} = (-1)^{im}x_1
  \end{equation}
for all $i\in\NN_0$, so that in particular $0\neq x_1=(-1)^{rm}x_1$, and
therefore either the number~$rm$ is even, and then $m$ is even since $r$
divides~$m$, or the characteristic of the field~$\kk$ is~$2$. Finally,
using~\eqref{eq:cs} and~\eqref{eq:ds} we can rewrite the sum~\eqref{eq:xs},
  \[
  u  = 
        \sum_{i=0}^{l-1}
        x_{\pi^i(1)}\cdot
        (\gamma_{\pi^i(1)},\alpha_{\pi^i(1)})
    =   
        x_1
        \sum_{i=0}^{l-1}
        (-1)^{im}\cdot
        (\rot^i(C),s(\rot^i(C)))
    = x_1 \cdot \cycle{C}.
  \]
With this, we have proved all the claims of the proposition.
\end{proof}

We are left with dealing with the cocycles of rank~$1$.
\begin{proposition}\label{prop:cocycles:rank-1}
Let $m\geq2$. An element~$u$ of~$\Gamma_m\parallel\cB$ is a cocycle and not
a coboundary if and only if one of the following conditions holds:
\begin{enumerate}[label=(\roman*), ref=(\roman*)]

\item\label{it:rank-1:loop} the integer $m$ is even or the characteristic
of~$\kk$ is~$2$, and there is a loop such that $u=(b^m,s(b))$,

\item\label{it:rank-1:complete} the integer~$m$ is odd or the
characteristic of~$\kk$ is~$2$, and there is a complete cycle~$C$
in~$(Q,I)$ with first arrow~$b$ such that $u=(bC,b)$,

\item\label{it:rank-1:maximal} we have $u=(\gamma,\alpha)$ with $\gamma$ a
$\Gamma$-maximal element of~$\Gamma_m$ and $\gamma$ and~$\alpha$ neither
starting nor ending with the same arrow.

\end{enumerate}
\end{proposition}

We remark that these three conditions are exclusive of each other, and that
when condition~\ref{it:rank-1:loop} holds we in fact have that the
path~$b^m$ is a complete cycle of length~$m$ and period~$1$ in~$(Q,I)$ such
that $u=\cycle{b^m}$.

\begin{proof}
Let $u=(\gamma,\alpha)$ be an element of~$\Gamma_m\parallel\cB$ such that
$d^m(u)=0$ and that is not a coboundary. In view of the definition
of~$d^m$, this tells us that 
\begin{enumerate}[label=(\Alph*), ref=(\Alph*)]

\item\label{it:lo:1} either $1-(-1)^m=0$ in~$\kk$, and there are arrows~$a$
and~$b$ in~$Q$ such that 
  \begin{equation}\label{eq:ws}
  (b\gamma,b\alpha)
        = (\gamma a,\alpha a)
        \in \Gamma_{m+1}\parallel\cB,
  \end{equation}

\item\label{it:lo:2} or there is no arrow~$b$ such that
$(b\gamma,b\alpha)\in\Gamma_{m+1}\parallel\cB$ and there is no arrow $a$
such that $(\gamma a,\alpha a)\in\Gamma_{m+1}\parallel\cB$.

\end{enumerate}
We will analyse these two possibilities in three steps.

\textsc{Step 1}. Let us start by supposing that we are in
case~\ref{it:lo:1}. As $1-(-1)^m=0$ in~$\kk$, it is clear that either the
integer~$m$ is even or the characteristic of the field~$\kk$ is 2. If
$\alpha$ has length at least~$2$, as does~$\gamma$, then from the
equality~\eqref{eq:ws} we see that there are paths $\delta\in\Gamma_{m-2}$
and $\zeta\in\cB$ such that $\gamma=b\delta a$ and $\alpha=b\zeta a$, and
this and the gentleness of~$(Q,I)$ imply that $d^m(\delta a,\zeta
a)=u$: this is a contradiction, for we are supposing that $u$ is not a coboundary.
We thus see that $\alpha$ has length at most~$1$. If it has length zero,
then from~\eqref{eq:ws} we see that $a=b$, that $b$ is a loop, that
$\gamma=b^m$ and that $\alpha=s(b)$, so that condition~\ref{it:rank-1:loop}
holds. If instead $\alpha$ has length one, then the equality~\eqref{eq:ws}
implies that $a=\alpha=b$, that $b$ is a loop, and that $\gamma=b^m$: as
$m\geq2$, it follows from this that $C\coloneqq b^{m-1}$ is a complete
cycle starting with the arrow~$b$, and that $u=(bC,b)$. Now, if the
characteristic of the field~$\kk$ is not~$2$, then $m$ is even and we know
from Lemma~\ref{lemma:sum} that
  \[
  d^{m-1}(\tfrac{1}{2}\cycle{b^{m-1}})
        = \tfrac{1}{2}(1-(-1)^{m-1})\cdot(b^m,b)
        = u,
  \]
contradicting our hypothesis: the characteristic of~$\kk$ is thus
necessarily~$2$ and we see that condition~\ref{it:rank-1:complete} holds.

\textsc{Step 2.} Let us next suppose that we are in case~\ref{it:lo:2}, and
that $\gamma$ and~$\alpha$ start with the same arrow, so that there are an
arrow~$b$ and paths $\delta\in\Gamma_{m-1}$ and $\zeta\in\cB$ such that
$\gamma=b\delta$ and $\alpha=b\zeta$. There is then an arrow~$c$ such that
$(\delta c,\zeta c)\in\Gamma_m\parallel\cB$, for otherwise
$d^{m-1}(\delta,\zeta)=u$, contradicting our hypothesis on~$u$. Now $\gamma
c=b\delta c$ is in~$\Gamma_{m+1}$ because $\delta$ has positive length, so
\ref{it:lo:2} implies that $\alpha c=b\zeta c\not\in\cB$: as $b\zeta$ and
$\zeta c$ are in~$\cB$, we see that $\zeta$ has length zero, so that
$\alpha=b$, that $\delta$ is a cycle, and that $bc\in R$. Since $b\delta
c\in\Gamma_{m+1}$, the gentleness of~$(Q,I)$ implies that $b$ is the
first arrow of~$\delta$, and since $b\delta \in\Gamma_m$, the path~$\delta$
is a complete cycle, and we have $u=(b\delta,b)$.

Proceeding symmetrically, we can see that if we are in case~\ref{it:lo:2}
and $\gamma$ and~$\alpha$ this time end with the same arrow~$b$, then there
is a complete cycle~$\epsilon$ ending with~$b$ such that $u=(\epsilon
b,b)$: but then we also have that $b$ is the first arrow of the rotated
complete cycle~$\delta\coloneqq\rot^{-1}(\epsilon)$ and that
$u=(b\delta,b)$.

We thus see that if we are in case~\ref{it:lo:2} and $\gamma$ and~$\alpha$
either start or end with the same arrow, there is a complete cycle~$C$
starting with an arrow~$b$ such that $u=(bC,b)$. Now Lemma~\ref{lemma:sum}
tells us that 
  \[
  d^{m-1}(\cycle{C}) 
        = (1-(-1)^{m-1})\cdot(bC,b) 
        = (1-(-1)^{m-1})\cdot u,
  \]
and since $u$ is not a coboundary we have that either the integer~$m$ is
odd or the characteristic of~$\kk$ is~$2$:
condition~\ref{it:rank-1:complete} thus holds.

\textsc{Step 3.} Finally, we are left with the situation in which
\ref{it:lo:2} holds and the paths~$\gamma$ and~$\alpha$ neither start nor
end with the same arrow: we will show that
condition~\ref{it:rank-1:maximal} holds, and in this situation the
path~$\gamma$ is maximal in~$\Gamma_m$. Suppose that it does not hold, so
that for example there is an arrow~$b$ such that $b\gamma\in\Gamma_{m+1}$.
In view of~\ref{it:lo:2}, we must then have that $b\alpha\not\in\cB$, so
that $\alpha$ has positive length and can be written in the form
$c\bar\alpha$ with $c$ an arrow such that $bc\in R$ and~$\bar\alpha$ a
path: since $b\gamma$ is in~$\Gamma_{m+1}$, the gentleness of~$(Q,I)$
implies that the last arrow of~$\gamma$ has to also be~$c$, and this is
a contradiction. We can reach a similar contradiction if we suppose instead that
there is an arrow $a$ such that $\gamma a\in\Gamma_{m+1}$, of course.

We have shown that if an element~$u$ of~$\Gamma_m\parallel\cB$ is a cocycle
and not a coboundary then one of the conditions~\ref{it:rank-1:loop},
\ref{it:rank-1:complete} or~\ref{it:rank-1:maximal} holds. A
straightforward calculation, on the other hand, shows that if any of these
three conditions holds then $u$ is a cocycle and not a coboundary, thereby
concluding the proof of the proposition. 
\end{proof}

We are an easy step away from completing the description we want of the
Hochschild cohomology of the algebra~$A$, and we will need some notations
to do that in a compact form. We will write

{

\begin{itemize}

\item $\C^\circ(\Gamma)$\indexsymbol{$\C^\circ(\Gamma)$} for the set of all
complete circuits in the presentation~$(Q,I)$, and

\item $\C(\Gamma)$\indexsymbol{$\C(\Gamma)$} for the subset
of~$\C^\circ(\Gamma)$ of those complete circuits~$C$ for which the
following condition holds:
  \[
  \claim[-.8]{either the length of~$C$ is even or the characteristic
  of~$\kk$ is~$2$.}
  \]

\end{itemize}

Moreover, for each positive integer~$m$ we let
$\C_m^\circ(\Gamma)$\indexsymbol{$\C_m^\circ(\Gamma)$} be the subset of 
$\C^\circ(\Gamma)$ of circuits of length~$m$
and
$\C_m(\Gamma)$\indexsymbol{$\C_m(\Gamma)$} be the subset of
$\C(\Gamma)$ of circuits of length~$m$ where we assume $m$ to be even or the characteristic of $\kk$ is 2.

In addition, let $\Crep^\circ(\Gamma)$\indexsymbol{$\Crep^\circ(\Gamma)$} be
a set of representatives of the circuits that belong to~$\C^\circ(\Gamma)$
chosen so that 
  \[
  \claim[.9]{if $C$ and $D$ are elements of~$\Crep^\circ(\Gamma)$ and $CD$
  is a complete cycle in~$(Q,I)$, then $CD$ also belongs
  to~$\Crep^\circ(\Gamma)$.}
  \]
It is possible to choose~$\Crep^\circ(\Gamma)$ like this for essentially
the same reason that made it possible to do it for~$\Crep(\cB)$ in
Section~\ref{sect:low-degrees}. The definition of $\Crep(\Gamma)$\indexsymbol{$\Crep(\Gamma)$} and 
$\Crep_m(\Gamma)$\indexsymbol{$\Crep_m(\Gamma)$}, for each $m\geq1$, are according to the previous ones.

\begin{proposition}\label{prop:hh:m}
Let $m\geq2$. The vector space~$\HH^m(A)$ is freely spanned by the
collection of the following elements:
\begin{itemize}

\item $\cycle{C}$ with $C\in\overline\C^\circ_m(\Gamma)$ with $m$ even or characteristic of $\kk$ equal to 2; 

\item $(bC,b)$ with $C\in\overline\C^\circ_{m-1}(\Gamma)$ and $b$ the first arrow
in~$C$, with $m$ odd or characteristic of $\kk$ equal to 2; 

\item $(\gamma,\alpha)$ in~$\Gamma_m\parallel\cB$ with $\gamma$ a
$\Gamma$-maximal element of~$\Gamma_m$ and $\gamma$ and~$\alpha$ neither
starting nor ending with the same arrow.

\end{itemize}
\end{proposition}

\begin{proof}
Every cocycle in~$\kk(\Gamma_m\parallel\cB)$ is a pairwise disjoint sum of
irreducible cocycles there, and Propositions~\ref{prop:cocycles:rank-r}
and~\ref{prop:cocycles:rank-1}, together with Lemma~\ref{lemma:sum:span},
tell us  that every cocycle is cohomologous to a linear combination of
the following collection of elements:
\begin{enumerate}[label=(\textit{\alph*}), ref=(\textit{\alph*})]

\item\label{it:Z:1} $\cycle{C}$, with $C\in\overline\C_m(\Gamma)$;

\item\label{it:Z:2} 
$(bC,b)$ with $C$ a complete cycle of length $m-1$ and $b$ the first arrow in~$C$; 

\item\label{it:Z:3} $(\gamma,\alpha)$ in~$\Gamma_m\parallel\cB$ with
$\gamma$ a $\Gamma$-maximal element of~$\Gamma_m$ and $\gamma$ and~$\alpha$
neither starting nor ending with the same arrow.

\end{enumerate}
These elements are linearly independent in~$\kk(\Gamma_m\parallel\cB)$, for
they are pairwise disjoint. Let us write $Z$ for the subspace they span. To
prove the claim of the theorem it is clearly enough to show that
  \begin{equation}\label{eq:Z}
  \claim[.85]{the subspace of coboundaries in~$Z$ is spanned by the elements of
  the form $(bC,C)-(-1)^{m-1}\cdot(a\rot(C),a)$, with $C$ in a circuit
  of~$\C_{m-1}(\Gamma)$, and $b$ and $a$ the first arrows of~$C$ and
  of~$\rot(C)$, respectively.}
  \end{equation}
The formula~\eqref{eq:differential} in Proposition~\ref{prop:strametz}
makes it clear that if an element~$(\gamma,\alpha)$
of~$\Gamma_m\parallel\cB$ appears in a coboundary then the path~$\alpha$
has positive length and $\gamma$ and~$\alpha$ either start or end with the
same arrow: it follows immediately from this that a coboundary belonging
to~$Z$ is a linear combination of elements of the form described
in~\ref{it:Z:2} above.

Now suppose that $C$ is a cycle in a circuit belonging to~$\C_{m-1}(\Gamma)$,
that $b$ is the first arrow in~$C$, and that the pair $(bC,b)$ appears in
the coboundary $d^{m-1}(v)$ of some pair~$v=(\delta,\beta)$
in~$\Gamma_{m-1}\parallel\cB$. Clearly, then, we have that
$\beta=s(\delta)$ and that $\delta$ is either $C$ or~$\rot^{-1}(C)$: in the
first case we have that
  \[
  d^{m-1}(v)
        = (bC,b)-(-1)^{m-1}\cdot(Ca,a)
        = (bC,b)-(-1)^{m-1}\cdot(a\rot(C),a)
  \]
with $a$ the last arrow in~$C$ and therefore the first one of~$\rot(C)$,
and in the second case that
  \[
  d^{m-1}(v) = (c\rot^{-1}(C),c) - (-1)^{m-1}\cdot(bC,b)
  \]
with $c$ the first arrow in~$\rot^{-1}(C)$. The claim~\eqref{eq:Z} above
follows from this at once.
\end{proof}

\section{Hochschild cohomology of gentle algebras and some consequences}
\label{sect:corollaries}

First of all, let us summarize the results that we obtained in what
precedes.

\newpage

\begin{theorem}\label{thm:cohomology:basis}
Let $(Q,I)$ be a gentle presentation, let $A\coloneqq\kk Q/I$ be
the algebra it presents, and let $T$ be a spanning tree for the quiver~$Q$.
The Hochschild cohomology $\HH^*(A)$ of~$A$ is the graded
vector space freely generated by the cohomology classes of the following
homogeneous cocycles of the complex $\kk(\Gamma\parallel\cB)$:
\begin{enumerate}[
        label=\textup{($\mathsf{H_{\Roman*}}$)}, 
        ref=\textup{($\mathsf{H_{\Roman*}}$)}
        ]

\item\label{gen:one} the element
  \[
  \one\coloneqq\sum_{i\in Q_0}(e_i,e_i),
  \]

\item\label{gen:B-max} the pairs $(s(\alpha),\alpha)$ with $\alpha$ a
$\cB$-maximal path in~$(Q,I)$,

\item\label{gen:B-sum} the sums
  \[
  \cycle{\alpha} 
        \coloneqq
        \sum_{i=0}^{r-1}
                (s(\rot^i(\alpha)),\rot^i(\alpha))
  \]
with $\alpha\in\Crep(\cB)$, as defined in Propositions~\ref{prop:hh:0},

\item\label{gen:fundamental} the pairs $(c,c)$ with $c$ an arrow in the
complement of the spanning tree~$T$,

\item\label{gen:B-plus} the pairs  $(c, c\alpha)$ with $\alpha \in\Crep(\cB)$
and $c$ the first arrow of~$\alpha$,

\item\label{gen:clean} the pairs $(\gamma,\alpha)$ with $\gamma$ a
$\Gamma$-maximal element of~$\Gamma$ and $\gamma$ and~$\alpha$ neither
beginning nor ending with the same arrow,

\item\label{gen:Gamma-sum}  for each $C\in\Crep^\circ_m(\Gamma)$  with $m$ even or characteristic of $\kk$ equal to 2, the sums, 
 \[
  \cycle{C} 
        \coloneqq  \sum_{i=0}^{r-1} (-1)^{im} \cdot(\rot^i(C),s(\rot^i(C))) = \sum_{i=0}^{r-1} (\rot^i(C),s(\rot^i(C))),
  \]


\item\label{gen:Gamma-plus} the pairs $(bC,b)$ with $C\in\Crep^\circ_{m-1}(\Gamma)$ and
$b$ the first arrow of~$C$ and with $m$ odd or characteristic of $\kk$ equal to 2.

\end{enumerate}
\end{theorem}

\begin{proof}
This is the information of Propositions~\ref{prop:hh:0}, \ref{prop:hh:1}
and~\ref{prop:hh:m}, which deal with the spaces $\HH^0(A)$, $\HH^1(A)$ and
$\HH^m(A)$ with $m\geq2$, respectively.
\end{proof}

When perusing the list in this theorem it is important to keep in mind that
the set $\C(\cB)$ depends only on the quiver~$Q$ and the set of
relations~$R$, but the set $\C(\Gamma)$ depends on those two things
\emph{and} on the characteristic of the ground field~$\kk$.

\bigskip

Just by looking at the cohomology of~$A$ as a graded vector space we obtain
the following two corollaries that characterize the finiteness of the
dimension of~$A$ and of its global dimension.

\begin{corollary}\label{coro:findim}
The following four statements are equivalent:
\begin{tfae}

\item\label{it:findim:1} the algebra $A$ is finite-dimensional,

\item\label{it:findim:2} $\HH^0(A)$ is finite-dimensional,

\item\label{it:findim:3} $\HH^1(A)$ is finite-dimensional,

\item\label{it:findim:4} there is no cocomplete cycle in~$(Q,I)$.

\end{tfae}
When they hold, then in fact $\HH^m(A)$ is finite-dimensional for all
$m\geq0$. 
\end{corollary}

The finite-dimensionality of  $\HH^m(A)$ for some $m\geq2$ is not enough to
conclude that $A$ is finite-dimensional. A simple example of this is the
path algebra on the quiver with one vertex and one arrow that is
infinite-dimensional but has vanishing Hochschild cohomology in all degrees
greater than~$1$.

\begin{proof}
If $A$ is finite-dimensional, then so is $\HH^0(A)$, for it its isomorphic
to the center of~$A$. If $A$ is instead infinite-dimensional, there is a
path~$\gamma$ in~$Q$ of length~$\abs{Q_1}+1$ that necessarily passes
through some arrow~$a$ at least twice: there are then
paths~$\gamma_1$,~$\gamma_2$ and~$\gamma_3$ such that
$\gamma=\gamma_3a\gamma_2a\gamma_1$ and $\gamma_2a$ is a cocomplete cycle.
On the other hand, if there is a cocomplete cycle~$\gamma$ in~$(Q,I)$ and
$c$ is its first arrow, then the classes in~$\HH^1(A)$ of the pairs
$(c,c\delta^k)$ with $k\geq1$ are linearly independent, so that
$\dim\HH^1(A)$ is infinite-dimensional. This shows that
$\eqref{it:findim:3}\implies \eqref{it:findim:4}\implies
\eqref{it:findim:1}\implies \eqref{it:findim:2}$.

Suppose finally that the vector space $\HH^1(A)$ is infinite-dimensional.
Since it has a basis whose elements are the classes of certain elements
of~$\Gamma_1\parallel\cB$ and there are only finitely many arrows, we see
that the set $\cB$ is infinite, so that $A$ is infinite-dimensional. We
have already shown that in that case there is a cocomplete cycle~$\gamma$
in~$(Q,I)$, and then the sums $\cycle{\gamma^k}$ with $k\geq1$ are linearly
independent elements of~$\HH^0(A)$, which is therefore
infinite-dimensional. This shows that
$\eqref{it:findim:2}\implies\eqref{it:findim:3}$.

The equivalence of the four statements is thus proved. As for the last
claim of the corollary, it is obvious that if $A$ is finite-dimensional,
then $\HH^m(A)$ is finite-dimensional for all $m\geq0$, for the
complex~$\kk(\Gamma\parallel\cB)$ is in that case locally finite.
\end{proof}

Using \cite{GreenHuang}, see also \cite{HKK}, and our results above, we obtain the following characterisation of gentle algebras of finite global dimension. }

\begin{corollary}\label{coro:fingldim}
The following three statements are equivalent:
\begin{tfae}

\item\label{it:fingldim:1} the algebra $A$ has finite global dimension,

\item\label{it:fingldim:2} there is an integer~$m_0$ such that $\HH^m(A)=0$
for all $m\geq m_0$,

\item\label{it:fingldim:3} there is no complete cycle in~$(Q,I)$.

\end{tfae}
\end{corollary}

\begin{proof}
Suppose first that the algebra $A$ has infinite global dimension. If we put
$n\coloneqq\abs{Q_1}+1$, there are then two left $A$-modules $M$ and~$N$
and an integer $m\geq n$ such $\Ext_A^m(M,N)\neq0$. If $\cR$ is the
Bardzell resolution that we described in Chapter~\ref{chapter:algebras},
then the complex $\cR\otimes_AM$ is a projective resolution of~$M$ and we
see that $\cR_m=A\otimes_E\kk\Gamma_m\otimes_EA\neq0$ and, in particular,
that there is a path~$\gamma$ in~$\Gamma_m$. This path has length greater
than the number of arrows, so there is some arrow~$a$ that appears twice
in~$\gamma$: there are then paths $\gamma_1$,~$\gamma_2$ and~$\gamma_3$
such that $\gamma=\gamma_3a\gamma_2a\gamma_1$. The path~$a\gamma_2$ is thus
a complete cycle in~$(Q,I)$, and this shows that
$\eqref{it:fingldim:3}\implies\eqref{it:fingldim:1}$.

Suppose next that there is a complete cycle~$\gamma$ in~$(Q,I)$ and let $m$
be its length. For all $k\in\NN$ the path $\gamma^k$ is also a complete
cycle in~$(Q,I)$ and the class of the cocycle~$\cycle{\gamma^{2k}}$ is
non-zero in~$\HH^{2km}(A)$. This shows that
$\eqref{it:fingldim:2}\implies\eqref{it:fingldim:3}$.

Finally, suppose that $A$ has finite global dimension and put
$m_0\coloneqq1+\gldim A$. If~$i$ and~$j$ are two vertices of~$Q$ and $S_i$
and $S_j$ are the corresponding simple $A$-modules, then
$\Ext^{m_0}_A(S_i,S_j)=0$. Now $\cR\otimes_AS_i$ is a projective resolution
of~$S_i$, and the form of the differential of~$\cR$ implies that the
differential of the complex $\Hom_A(\cR\otimes_AS_i,S_j)$ that computes
$\Ext^*_A(S_i,S_j)$ is zero: we therefore have that
 \[
  0 = \Hom_A(\cR_{m_0}\otimes_AS_i,S_j)
    = \Hom_E(\kk\Gamma_{m_0}\otimes_ES_i,S_j).
  \]
This implies that there are no paths in~$\Gamma_{m_0}$ from~$i$ to~$j$, and
tells us that whenever $m\geq m_0$ the set~$\Gamma_{m}$ is empty and
$\HH^m(A)=0$. This shows that
$\eqref{it:fingldim:1}\implies\eqref{it:fingldim:2}$.
\end{proof}

Combining the two Corollaries~\ref{coro:findim} and~\ref{coro:fingldim} we
obtain the following.

\begin{corollary}
The Hochschild cohomology~$\HH^*(A)$ is finite-dimensional if and only if
both the dimension of~$A$ and its global dimension are finite. \qed
\end{corollary}

According to Corollary~\ref{coro:findim}, the Hochschild cohomology of a
gentle algebra~$A$ is locally finite-dimensional exactly when the algebra
itself is finite-dimensional, and in that case, it makes sense to determine
the \emph{growth rate} of that cohomology. It turns out that this has a simple
description --- it does not really grow:

\begin{corollary}
Let $(Q,I)$ be a \fd gentle presentation and let $A\coloneqq\kk Q/I$ be the
algebra it presents. 
\begin{thmlist}

\item The sequence $(\dim\HH^m(A))_{m\geq0}$ is bounded. More precisely,
the number~$N$ of primitive complete cycles in~$(Q,I)$ is finite and we
have that
  \[
  \dim\HH^m(A) \leq 2N
  \]
wherever $m>\abs{Q_1}$. 

\item If there exists an integer $m>\abs{Q_1}$ such that $\HH^m(A)\neq0$,
then there are infinitely many integers with that property.

\end{thmlist}
\end{corollary}

\begin{proof}
The number of $\Gamma$-maximal paths in~$\Gamma$ is finite and they all
have length at most~$\abs{Q_1}$: indeed, if a path is of length greater
than the number of arrows in~$Q$, then it passes through some arrow more
than one time, and if it belongs to~$\Gamma$ is cannot be $\Gamma$-maximal.
It follows from this and Proposition~\ref{prop:hh:m} that whenever
$m>\abs{Q_1}$ the vector space $\HH^m(A)$ is freely spanned by the
following elements of~$\kk(\Gamma_m\parallel\cB)$:
\begin{itemize}

\item the sums $\cycle{C}$ with $C\in\overline\C_m(\Gamma)$; 

\item the pairs $(bC,b)$ with $C\in\overline\C_{m-1}(\Gamma)$ and $b$ the
first arrow in~$C$.

\end{itemize}
We thus see that for such~$m$ we have that
  \begin{equation}\label{eq:dimhh}
  \dim\HH^m(A) = \abs{\C_m(\Gamma)} + \abs{\C_{m-1}(\Gamma)}.
  \end{equation}

There are finitely many primitive complete cycles
in~$(Q,I)$ --- for two of them have an arrow in common exactly when they
are conjugate, since the presentation~$(Q,I)$ is \fd gentle  --- and if $N$ is
their number then for all $n\in\NN$ the cardinal of the set~$\C_n(\Gamma)$
is at most~$N$, as every element of~$\C_n(\Gamma)$ is the conjugacy class
of a factor of the $n$th power of a primitive complete cycle. Together
with~\eqref{eq:dimhh} this proves the inequality that appears in the
first part of the corollary.

On the other hand, if there is an integer $m$ such that $m>\abs{Q_1}$ and
$\HH^m(A)\neq0$, then our observations imply that one of the
sets~$\C_m(\Gamma)$ or~$\C_{m-1}(\Gamma)$ is non-empty, and therefore for
all $k\in\NN$ one of~$\C_{km}(\Gamma)$ or~$\C_{k(m-1)}(\Gamma)$ is
non-empty, and there are infinitely many Hochschild cohomology spaces which
are non-zero.
\end{proof}

The number~$\abs{Q_1}$ in the second part of this corollary is the smallest
one that makes that claim true. For example, if $N\geq2$, the quiver~$Q$ is
an oriented cycle of length~$N$ and the ideal~$I$ is generated by all paths
of length~$2$ in~$Q$ except one, then the presentation~$(Q,I)$ is gentle,
there are no complete cycles in~$(Q,I)$ and therefore only for finitely
many integers~$m$ we have $\HH^m(A)\neq0$, and $\HH^{\abs{Q_1}}(A)$ is of
dimension~$1$ and, in particular, non-zero.



\chapter{Hochschild homology of quadratic monomial algebras}
\label{chapter:homology}

In this chapter we fix a quadratic monomial presentation $(Q,I)$,  and
compute the Hochschild homology~$\HH_*(A)$ of the algebra $A\coloneqq\kk
Q/I$ that it presents. As in the previous chapters, our main interest is in
gentle algebras, but working more generally with a quadratic monomial
presentation does not require any extra effort, so we will do so. 

The Hochschild homology of~$A$ is, since we are working over a field,
canonically isomorphic to $\Tor^{A^e}_*(A,A)$, and we will realize it as
the homology of the complex~$A\otimes_{A^e}\cR$, with 
  \[
  \cR 
    = A\otimes_E\kk\Gamma\otimes_E A
  \]
the Barzdell resolution that we described in
Chapter~\ref{chapter:algebras}.

If $\alpha$ is a path in~$Q$ of positive length, then there are arrows
$\lfact1{\alpha}$ and $\rfact2{\alpha}$, the last and the first arrow
of~$\alpha$, respectively, and paths $\lfact2{\alpha}$
and~$\rfact1{\alpha}$, possibly of length zero, such that
  \[
  \alpha 
    = \lfact1{\alpha}\lfact2{\alpha} 
    = \rfact1{\alpha}\rfact2{\alpha}.
  \]
Two paths~$\alpha$ and~$\beta$ in the quiver~$Q$ are
\newterm{antiparallel}\index{Antiparallel paths} if
$s(\alpha)=t(\beta)$ and $t(\alpha)=s(\beta)$, and in that case we can
consider the concatenations~$\alpha\beta$ and~$\beta\alpha$. If $X$
and~$Y$ are sets of paths in~$Q$, then we denote by $X\odot Y$ the set
of all pairs $(\alpha,\beta)$ in~$X\times Y$ with $\alpha$ and~$\beta$
antiparallel, and by $\kk(X\odot Y)$ the vector space it freely spans.
For each $m\geq0$ there is a map
  \[
  \Psi_m:
    \kk(\cB\odot\Gamma_m) 
    \to 
    A\otimes_{A^e}(A\otimes_E\kk\Gamma_m\otimes_E A)
  \]
such that $\Phi_m(\alpha,\gamma)=\alpha\otimes(1\otimes\gamma\otimes
1)$ for all choices of $(\alpha,\gamma)$ in~$\cB\odot\Gamma_m$, and it
is an isomorphism of vector spaces. There is a unique way to turn the
graded vector space~$\kk(\cB\odot\Gamma)$ into a complex in such a way
that the collection of maps~$(\Psi_m)_{m\geq0}$ is an isomorphism of
complexes $\kk(\cB\odot\Gamma)\to A\otimes_{A^e}\cR$: for each
$m\geq1$ the differential 
  \[ 
  d_m:
    \kk(\cB\odot\Gamma_m) 
    \to 
    \kk(\cB\odot\Gamma_{m-1})
  \]
is such that
  \begin{equation} \label{eq:hom:dm:2}
  d_m(\alpha,\gamma) 
    = (\alpha\lfact1{\gamma},\lfact2{\gamma})
      + (-1)^m \cdot(\rfact2{\gamma}\alpha,\rfact1{\gamma})
  \end{equation}
for all pairs~$(\alpha,\gamma)$ in $\cB\odot\Gamma_m$. We compute the
Hochschild homology~$\HH_*(A)$ of~$A$ by identifying it with the
homology of the chain complex~$\kk(\cB\odot\Gamma)$.

\bigskip

As in Chapter~\ref{chapter:cohomology} we let~$\C$ be the set of all
circuits in~$Q$, and put
  \begin{equation} \label{eq:cprime}
  \C' \coloneqq \C \cup \{\{e\} : e\in Q_0 \}.
  \end{equation}
If a pair $(\alpha,\gamma)$ belongs to $\cB\odot\Gamma$, then either
both of its components have length zero, in which case the path
$\alpha\gamma$ has length zero, or the path~$\alpha\gamma$ is a cycle:
in both cases, the concatenation~$\alpha\gamma$ belongs to a unique
element of~$\C'$. It follows from this that if we define, for each
$C\in\C'$, 
  \[
  (\cB\odot\Gamma)_C
    \coloneqq \{(\alpha,\gamma)\in\cB\odot\Gamma:\alpha\gamma\in C\},
  \]
then we have a partition
  \(
  \cB\odot\Gamma
    = \bigsqcup_{C\in\C'}(\cB\odot\Gamma)_C
  \). 
Moreover, if $C\in\C'$ and we write $\kk(\cB\otimes\Gamma)_C$ for the span
of~$(\cB\odot\Gamma)_C$ in the complex~$\kk(\cB\otimes\Gamma)$, it follows
at once from the formula~\eqref{eq:hom:dm:2} for the differential of the
latter that $\kk(\cB\odot\Gamma)_C$ is in fact a subcomplex, and we
therefore have a direct sum decomposition of complexes
  \begin{equation}\label{eq:hom:desc}
  \kk(\cB\odot\Gamma) 
    = \bigoplus_{C\in\C'}\kk(\cB\odot\Gamma)_C.
  \end{equation}
This observation reduces the problem of computing the homology of the
complex $\kk(\cB\odot\Gamma)$ to that of computing the homology of the
each of the smaller complexes $\kk(\cB\odot\Gamma)_C$, one for each
$C\in\C'$, which is what now do.

\bigskip

We deal first with the direct summands appearing
in~\eqref{eq:hom:desc} that correspond to elements of~$\C'$ that are
complete or cocomplete circuits. For cocomplete ones, we have the
following result:

\begin{lemma}\label{lemma:homology:cocomplete}
Let $C$ be a cocomplete circuit in~$(Q,I)$, let $r$ be its period, and
let $\bar C$ be a cycle in~$C$.
\begin{thmlist}

\item The vector space~$H_0(\kk(\cB\odot\Gamma)_C)$ is one-dimensional
and spanned by the homology class of~$(\bar C,s(\bar C))$, which is
independent of the choice of~$\bar C$ in~$C$.

\item The vector space~$H_1(\kk(\cB\odot\Gamma)_C)$ is
one-dimensional, spanned by the class of
  \[
  \hcycle{\bar C} 
    \coloneqq \sum_{i=0}^{r-1} 
                (\rfact1{\rot^i(\bar C)}, \rfact2{\rot^i(\bar C)})
    \in \kk(\cB\odot\Gamma_1)_C,
  \]
which is again independent of the choice of~$\bar C$ in~$C$.

\item For all $m\geq2$ we have that $H_m(\kk(\cB\odot\Gamma)_C)=0$.

\end{thmlist}
\end{lemma}

\begin{proof}
Let us write $\bar C_i\coloneqq\rot^i(\bar C)$ for each
$i\in\{0,\dots,r-1\}$. Since the circuit~$C$ is cocomplete, it is
clear that $(\cB\odot\Gamma_m)_C=\emptyset$ for all $m\geq2$, and that 
  \begin{gather*}
  (\cB\odot\Gamma_0)_C
    = \{
        (\bar C_0,s(\bar C_0)),
        \dots,
        (\bar C_{r-1},s(\bar C_{r-1}))
      \}
\shortintertext{and}
  (\cB\odot\Gamma_1)_C
    = \{
        (\rfact1{(\bar C_0)},\rfact2{(\bar C_0)}),
        \dots,
        (\rfact1{(\bar C_{r-1})},\rfact2{(\bar C_{r-1})})
      \},
  \end{gather*}
both sets having cardinality exactly~$r$. The differential in the
complex~$\kk(\cB\odot\Gamma)_C$ is such that
  \[
  d(\rfact1{(\bar C_i)},\rfact2{(\bar C_i)})
    = (\bar C_i,s(\bar C_i)) - (\bar C_{i+1},s(\bar C_{i+1}))
  \]
for each $i\in\{0,\dots,r-1\}$, with indices taken modulo~$r$. 
The claims of the lemma follow easily from these observations.
\end{proof}

The corresponding result for complete cycles is similar but more
complicated, since it involves the characteristic of the ground field:

\begin{lemma}\label{lemma:homology:complete}
Let $C$ be a complete circuit in~$(Q,I)$, let~$l$ and~$r$ be its length
and its period, respectively, and let $\bar C$ be a cycle in~$C$.
\begin{thmlist}

\item If $(-1)^{(l+1)r}=1$ in~$\kk$, then 
\begin{itemize}

\item $H_m(\kk(\cB\odot\Gamma)_C)=0$ for all integers $m$ different
from~$l$ and~$l-1$, and 

\item the vector spaces $H_l(\kk(\cB\odot\Gamma)_C)$ and
$H_{l-1}(\kk(\cB\odot\Gamma)_C)$ are both of dimension~$1$, and they are
spanned by the $l$-cycle 
  \begin{equation}\label{eq:l-cycle}
  \hskip4em 
  \hcycle{\bar C} 
    \coloneqq
      \sum_{i=0}^{r-1} (-1)^{(l+1)i}\cdot \bigl(s(\rot^i(\bar C)), \rot^i(\bar C))
      \in \kk(\cB\odot\Gamma_l)_C
  \end{equation}
and the $(l-1)$-cycle 
  \[
  (\lfact1{\bar C},\lfact2{\bar C}) 
    \in \kk(\cB\odot\Gamma_{l-1})_C,
  \]
respectively. These generators do depend on the choice of~$\bar C$
in~$C$. Namely, we have that 
  \[
  \hcycle{\rot(\bar C)}
    = (-1)^{l+1}\cdot\hcycle{\bar C}
  \]
and that the $(l-1)$-cycles
  \[
  (\lfact1{\rot(\bar C)},\lfact2{\rot(\bar C)}) 
  \quad\text{and}\quad
  (-1)^{l+1}\cdot (\lfact1{\bar C},\lfact2{\bar C})
  \]
are homologous.

\end{itemize}

\item If $(-1)^{(l+1)r}\neq1$ in~$\kk$, then we have
$H_m(\kk(\cB\odot\Gamma)_C)=0$ for all $m\in\ZZ$.

\end{thmlist}
\end{lemma}

\begin{proof}
If $(\alpha,\gamma)$ is an element of~$(\cB\odot\Gamma)_C$, then the
length of~$\alpha$ is either zero or one, for it is a factor of a
cycle belonging to the circuit~$C$, and $C$~is complete. This tells us
that the set $(\cB\odot\Gamma_m)_C$ is empty unless~$m=l$ or $m=l-1$.
On the other hand, if we write $\bar C_i\coloneqq\rot^i(\bar C)$ for
each $i\in\{0,\dots,r-1\}$, we have that
  \begin{gather*}
  (\cB\odot\Gamma_l)_C 
    = \bigl\{
      (s(\bar C_0),\bar C_0),
      \dots,
      (s(\bar C_{r-1}),\bar C_{r-1})
      \bigr\},
      \\
  (\cB\odot\Gamma_{l-1})_C 
    = \bigl\{
      (\lfact1{(\bar C_0)},\lfact2{(\bar C_0)}),
      \dots,
      (\lfact1{(\bar C_{r-1})},\lfact2{(\bar C_{r-1})})
      \bigr\},
  \end{gather*}
that both sets have cardinal exactly~$r$, and that
  \[
  d_l\bigl((s(\bar C_i),\bar C_i)\bigr)
    = (\lfact1{(\bar C_i)},\lfact2{(\bar C_i)})
      + (-1)^l \cdot(\lfact1{(\bar C_{i+1})},\lfact2{(\bar C_{i+1})})
  \]
for each $i\in\{0,\dots,r-1\}$, with indices taken modulo~$r$.

An $l$-chain~$z$ in the complex~$\kk(\cB\odot\Gamma)_C$ can be written
in the form 
  \[
  \sum_{i=0}^{r-1}\lambda_i\cdot(s(\bar C_i),\bar C_i),
  \]
with uniquely determined coefficients
$\lambda_0$,~\dots,~$\lambda_{r-1}\in\kk$, and it is a cycle in the
complex exactly when
  \[
  0 
    = d_l(z) 
    = \sum_{i=0}^{r-1}
        \bigl(\lambda_i+(-1)^l\lambda_{i-1}\bigr)
        \cdot (\lfact1{(\bar C_i)},\lfact2{(\bar C_i)}).
  \]
This happens precisely when $\lambda_{i+1}=(-1)^{l+1}\lambda_i$ for
all $i\in\{0,\dots,r-1\}$ or, equivalently, when
$\lambda_i=(-1)^{(l+1)i}\lambda_0$ for all $i\in\{1,\dots,r-1\}$ and
$((-1)^{(l+1)r}-1)\lambda_0=0$. It follows from this that when
$(-1)^{(l+1)r}\neq1$ in~$\kk$ the differential
  \[
  d_l:\kk(\cB\odot\Gamma_l)_C\to \kk(\cB\odot\Gamma_{l-1})_C
  \]
is injective, so that it is an isomorphism, and the vector spaces
$H_l(\kk(\cB\odot\Gamma)_C)$ and $H_{l-1}(\kk(\cB\odot\Gamma)_C)$ are
both zero, and that if instead $(-1)^{(l+1)r}=1$ in~$\kk$, then those
two vector spaces are both one-dimensional. In this last case the
first one is generated by the element~\eqref{eq:l-cycle} described in
the statement of the lemma and, in view of the form of the
differential~$d_l$ of our complex, the class
in~$H_{l-1}(\kk(\cB\odot\Gamma)_C)$ of $(\lfact1{(\bar
C_i)},\lfact2{(\bar C_i)})$ is homologous to $(-1)^{l+1}\cdot
(\lfact1{(\bar C_{i+1})},\lfact2{(\bar C_{i+1})})$, so that this
vector space is spanned by the class of~$(\lfact1{(\bar
C_0)},\lfact2{(\bar C_0)})$. 
\end{proof}

The next lemma deals with the circuits in~$\C'$ that are neither
complete nor cocomplete, which are the ones we have not yet
considered.

\begin{lemma}\label{lemma:homology:neither}
If $C$ is a circuit  in~$\C'$ that is neither complete nor cocomplete,
then the complex~$\kk(\cB\odot\Gamma)_C$ is exact.
\end{lemma}

\begin{proof}
Let $l$ be the length of~$C$. It is clear that the set
$(\cB\odot\Gamma_m)_C$ is empty if $m>l$, so that
$H_m(\kk(\cB\odot\Gamma)_C)=0$ for such~$m$. To compute the homology
of the complex~$\kk(\cB\odot\Gamma)_C$ for lower degrees we consider
first a few special cases.
\begin{itemize}

\item Suppose first that $(\cB\odot\Gamma_l)_C$ is not empty, so that
there is path~$\gamma\in\Gamma_l\cap C$ such that the
pair~$(s(\gamma),\gamma)$ is in~$(\cB\odot\Gamma_l)_C$. As the
circuit~$C$ is not complete, this is in fact the only element
of~$(\cB\odot\Gamma_l)_C$. On the other hand, as the circuit~$C$ is
not cocomplete, we have that $l>1$, so that the path~$\gamma$ has
length at least~$2$. Let $a$,~$b$ and~$\delta$ be the arrows and the
path, respectively, such that $\gamma=b\delta a$. That the circuit~$C$
is neither complete nor cocomplete implies that
  \begin{gather*}
  (\cB\odot\Gamma_l)_C = \{(s(\gamma),\gamma)\}, 
  \\
  (\cB\odot\Gamma_{l-1})_C = \{(a,b\delta),\,(b,\delta a)\}, 
  \\
  (\cB\odot\Gamma_{l-2})_C = \{(ab,\delta)\},
  \end{gather*}
and that $(\cB\odot\Gamma_{m})_C = \emptyset$ for all $m<l-2$, and in
the complex~$\kk(\cB\odot\Gamma)_C$ the differential is such that 
  \begin{gather*}
  d_l(s(\gamma),\gamma) = (b,\delta a) +(-1)^l\cdot(a,b\delta), 
  \\
  d_{l-1}(a,b\delta) = (ab,\delta), 
  \\
  d_{l-1}(b,\delta a) = (-1)^{l-1}\cdot(ab,\delta).
  \end{gather*}
We can immediately see from this that the
complex~$\kk(\cB\odot\Gamma)_C$ is exact in this case.

\item Next, we suppose that $(\cB\odot\Gamma_l)_C$ is empty, so that
in particular we have that $l>1$. We further suppose that
$(\cB\odot\Gamma_{l-1})$ is not empty, so that there is an element
$(a,\gamma)$ in~$(\cB\odot\Gamma_{l-1})_C$ in which the first
component~$a$ is an arrow and the second one~$\gamma$ has positive
length. As $(\cB\odot\Gamma_l)_C=\emptyset$, we have that the
paths~$a\gamma$ and~$\gamma a$ are not in~$\Gamma$, and as $C$ is not
cocomplete, we have that $\gamma$ has length at least~$2$.

Let $f$,~$g$ and~$\delta$ be the arrows and the path, respectively,
such that $\gamma=g\delta f$. As $(\cB\odot\Gamma_l)=\emptyset$, the
paths~$ag$ and~$fa$ are not in~$I$. Since the circuit~$C$ is neither
complete nor cocomplete, we have that
 \begin{gather*}
  (\cB\odot\Gamma_{l-1})_C = \{(a,g\delta f)\}, \\
  (\cB\odot\Gamma_{l-2})_C = \{(ag,\delta f),\,(fa,g\delta)\}, \\
  (\cB\odot\Gamma_{l-3})_C = \{(fag,\delta)\}, 
  \end{gather*}
and that $(\cB\odot\Gamma_m)_C=\emptyset$ if $m<l-3$. In the
complex~$\kk(\cB\odot\Gamma)_C$ we have that
  \begin{gather*}
  d_{l-1}(a,g\delta f) =(ag,\delta f)+(-1)^{l-1}(fa,g\delta), \\
  d_{l-2}(ag,\delta f) = (-1)^{l-2}\cdot(fag,\delta), \\
  d_{l-2}(fa,g\delta) = (fag,\delta),
  \end{gather*}
and therefore that the the complex~$\kk(\cB\odot\Gamma)_C$ is exact.

\item Suppose now that the set $(\cB\odot\Gamma_0)_C$ is not empty. As
the circuit~$C$ is not cocomplete, that set has exactly one element,
which is of the form $(\alpha,s(\alpha))$, with $\alpha$ in~$C\cap\cB$
and, since $C$ is also not complete, $\alpha$ of length at
least~$2$. If $a$,~$b$ and~$\beta$ are the arrows and the path,
respectively, such that $\alpha=b\beta a$, then $ab\in I$,
  \begin{gather*}
  (\cB\odot\Gamma_2)_C = \{(\beta,ab)\}, \\
  (\cB\odot\Gamma_1)_C = \{(b\beta,a),(\beta a,b)\}, \\
  (\cB\odot\Gamma_0)_C = \{(b\beta a,s(a))\}, 
  \end{gather*}
and $(\cB\odot\Gamma_m)_C=\emptyset$ for all integers $m>2$. The
differential in the complex $\kk(\cB\odot\Gamma)_C$ is such that
  \begin{gather*}
  d_2(\beta,ab) = (\beta a,b) + (b\beta,a), \\
  d_1(b\beta,a) = (b\beta a,s(a)), \\
  d_1(\beta a,b) = -(b\beta a,s(a)),
  \end{gather*}
and therefore it is exact.

\item Finally, suppose that  $(\cB\odot\Gamma_0)_C$ is empty and that
$(\cB\odot\Gamma_1)_C$ is not, so that there is an element $(\alpha,a)$ in
this last set. As $(\cB\odot\Gamma_0)_C=\emptyset$, the paths $\alpha a$
and $a\alpha$ are not in~$\cB$ and therefore, since $C$ is not cocomplete,
the path~$\alpha$ has length at least~$2$. If $b$,~$c$ and~$\beta$ are the
arrows and the path, respectively, such that $\alpha=c\beta b$, we have that
  \begin{gather*}
  (\cB\odot\Gamma_3)_C=\{(\beta,bac)\}, \\
  (\cB\odot\Gamma_2)_C=\{(\beta b,ac),\,(c\beta,ba)\}, \\
  (\cB\odot\Gamma_1)_C=\{(c\beta b,a)\}, 
  \end{gather*}
that $(\cB\odot\Gamma_m)_C=\emptyset$ if $m=0$ or $m>3$, and that 
in the complex~$\kk(\cB\odot\Gamma)_C$ 
  \begin{gather*}
  d_3(\beta,bac)=(\beta b,ac)-(c\beta,ba), \\
  d_2(c\beta,ba)=(c\beta b,a), \\
  d_2(\beta b,ac) = -(c\beta b,a).
  \end{gather*}
This implies that complex is exact.
\end{itemize}

We are left with considering the `generic'  situation, in which the
four sets $(\cB\odot\Gamma_0)_C$, $(\cB\odot\Gamma_1)_C$,
$(\cB\odot\Gamma_l)_C$ and $(\cB\odot\Gamma_{l-1})_C$ are empty. To
show that the complex~$\kk(\cB\odot\Gamma)_C$ is exact it will be
enough to fix an integer~$m$ such that $2\leq m\leq l-2$ and show that
$H_m(\kk(\cB\odot\Gamma)_C)=0$.

We fix a cycle~$\bar C$ in the circuit~$C$, consider the primitive
cycle~$D$ and the positive integer~$t$ such that $\bar C=D^t$, and let
$r$ be the length of~$D$, which is the period of~$C$. For each
$i\in\{0,\dots,r-1\}$ there exist uniquely determined paths $\alpha_i$
and~$\gamma_i$ in~$Q$ such that $\rot^i(\bar C)=\alpha_i\gamma_i$ and
$\abs{\gamma_i}=m$, and if we let $\I$ be the set of
indices~$i\in\{0,\dots,r-1\}$ such that $\alpha_i\in\cB$ and
$\gamma_i\in\Gamma_m$, then 
  \[
  (\cB\odot\Gamma_m)_C = \{ (\alpha_i,\gamma_i):i\in\I \}.
  \]
We view the indices as taken modulo~$r$ throughout.

Let $z$ be an $m$-cycle in the complex~$\kk(\cB\odot\Gamma)_C$, and
let $c:\I\to\kk$ be the unique function such that
$z=\sum_{i\in\I}c(i)\cdot(\alpha_i,\gamma_i)$. In order to complete
the proof of the lemma, we will show that $z$ is a boundary. With this
in mind it is clear that we may assume --- by replacing it by a
homologous cycle if needed --- that $z$~satisfies the following
condition:
  \begin{equation}\label{eq:bd}
  \claim{for all $i\in\I$ we have $c(i)=0$ if the pair
  $(\alpha_i,\gamma_i)$ is a boundary.}
  \end{equation}

We start by checking that for each integer~$i$ in the
set~$\{0,\dots,r-1\}$ we have that
  \begin{equation}\label{eq:obs:1}
  i,\,i+1\in\I \implies c(i+1) = (-1)^{m+1}c(i).
  \end{equation}
Let $i$ be an element of~$\I$ such that $i+1$ is also in~$\I$. As $2\leq
m\leq l-2$, the four paths~$\alpha_i$, $\alpha_{i+1}$, $\gamma_i$
and~$\gamma_{i+1}$ all have length at least two, and they are related as
follows:
  \begin{equation}\label{eq:two}
  \begin{aligned}
  \alpha_{i+1} &= \lfact2{(\alpha_i)}\lfact1{(\gamma_i)}, \\
  \gamma_{i+1} &= \lfact2{(\gamma_i)}\lfact1{(\alpha_i)}.
  \end{aligned}
  \qquad\qquad
  \begin{tikzpicture}[
        baseline={(0,-0.15)}, scale=0.7, font=\tiny,
        very thick, -{Straight Barb[width=1mm, length=1mm]}
        ]
    \def\rInner{0.5+0.25}
    \def\rMid{0.75+0.25}
    \def\rOuter{1+0.25}
    \draw[red] (0+3:\rInner) arc(0+3:160-3:\rInner);
    \draw[blue] (160+3:\rInner) arc(160+3:360-3:\rInner);
    \draw[red] (-40+3:\rOuter) arc(-40+3:120-3:\rOuter);
    \draw[blue] (120+3:\rOuter) arc(120+3:320-3:\rOuter);
    \draw (0+3:\rMid) arc(0+3:120-3:\rMid);
    \draw (120+3:\rMid) arc(120+3:160-3:\rMid);
    \draw (160+3:\rMid) arc(160+3:320-3:\rMid);
    \draw (320+3:\rMid) arc(320+3:360-3:\rMid);
    \node[anchor=center] at (160/2:\rInner-0.4) {$\gamma_i$};
    \node[anchor=center] at ({(-40+120)/2}:\rOuter+0.4) {$\gamma_{i+1}$};
    \node[anchor=center] at ({(160+360)/2}:\rInner-0.4) {$\alpha_i$};
    \node[anchor=center] at ({(120+320)/2}:\rOuter+0.4) {$\alpha_{i+1}$};
  \end{tikzpicture}
  \end{equation}
As $\gamma_i\in\Gamma_m$, the path~$\lfact2{(\gamma_i)}$ is
in~$\Gamma_{m-1}$. On the other hand, since
$\alpha_i=\lfact1{(\alpha_i)}\lfact2{(\alpha_i)}$ and
$\alpha_{i+1}=\lfact2{(\alpha_i)}\lfact1{(\gamma_i)}$ are both
in~$\cB$ and have length at least~$2$, the path
$\alpha_i\lfact1{(\gamma_i)}$ is also in~$\cB$. We thus see that the
pair $(\alpha_i\lfact1{(\gamma_i)},\lfact2{(\gamma_i)})$ is an element
of~$(\cB\odot\Gamma_{m-1})_C$ that can also be written in the form
$(\rfact2{(\gamma_{i+1})}\alpha_{i+1},\rfact1{(\gamma_{i+1})})$. The
coefficient of that element in~$d_m(z)$ is $c(i)+(-1)^mc(i+1)$, and
therefore we have that $c(i+1)=(-1)^{m+1}c(i)$, as we wanted.

The second observation that we need is that 
  \begin{equation}\label{eq:obs:2}
  i\not\in\I,i+1\in\I \implies c(i+1)=0
  \end{equation}
holds whenever $i\in\{0,\dots,r-1\}$. To prove it, let us fix
$i\in\{0,\dots,r-1\}$ such that $i$~is not in~$\I$ but $i+1$ is. As
$i\not\in\I$, we have that $\alpha_i\not\in\cB$ or that
$\gamma_i\not\in\Gamma$. We consider the two possibilities.
\begin{itemize}

\item Suppose first that $\alpha_i\not\in\cB$. In view of the first
equality of the two in~\eqref{eq:two}, the path
$\lfact1{(\alpha_{i})}\lfact1{(\alpha_{i+1})}$ is then in~$I$, so that
$(\lfact2{(\alpha_{i+1})},\gamma_{i+1}\lfact1{(\alpha_{i+1})})$ is an
element of~$(\cB\odot\Gamma_{m+1})_C$. As the right factor of
length two in the path
$\lfact2{(\alpha_{i+1})}\lfact1{(\gamma_{i+1})}$ coincides with the
left factor of length two in~$\gamma_i$, which belongs to the
ideal~$I$, we see that
  \[
  d(\lfact2{(\alpha_{i+1})},\gamma_{i+1}\lfact1{(\alpha_{i+1})})
    = (-1)^{m+1}\cdot(\alpha_{i+1},\gamma_{i+1}),
  \]
and that $(\alpha_{i+1},\gamma_{i+1})$ is a coboundary: according to
our assumption~\eqref{eq:bd} we then have that $c(i+1)=0$.
        
\item Suppose now that $\alpha_i\in\cB$. The
path~$\lfact1{(\alpha_i)}\lfact1{(\alpha_{i+1})}$, which coincides
with $\rfact2{(\gamma_{i+1})}\lfact1{(\alpha_{i+1})}$, is therefore
in~$\cB$: it follows from this that
$\rfact2{(\gamma_{i+1})}\alpha_{i+1}$ is also in~$\cB$ and that the
pair $(\rfact2{(\gamma_{i+1})}\alpha_{i+1},\rfact1{(\gamma_{i+1})})$
belongs to $(\cB\odot\Gamma_{m-1})$. This pair appears in the boundary
of~$(\alpha_{i+1},\gamma_{i+1})$ with coefficient $(-1)^{m+1}$, and
not in the boundary of $(\alpha_j,\gamma_j)$ for any
$j\in\I\setminus\{i+1\}$ --- precisely because $i\not\in \I$. We can
conclude from all this that the coefficient with which  the pair
$(\rfact2{(\gamma_{i+1})}\alpha_{i+1},\rfact1{(\gamma_{i+1}))}$
of~$(\cB\odot\Gamma_{m-1})_C$ appears in $d(z)$ is exactly
$(-1)^{m+1}c(i+1)$, and that we have $c(i+1)=0$ also in this case.

\end{itemize}
Now, as the circuit~$C$ is neither complete nor cocomplete the set~$\I$ is
\emph{properly} contained in~$\{0,\dots,r-1\}$. This, together with our two
observations~\eqref{eq:obs:1} and~\eqref{eq:obs:2} implies that the
function~$c$ is identically zero.
\end{proof}

If we now put everything together, we obtain a description of the
Hochschild homology of the algebra~$A$ presented by~$(Q,I)$, which, as
we said above, we identify with the homology of the
complex~$\kk(\cB\odot\Gamma)$.


\begin{theorem}\label{thm:basishomology}
Let $(Q,I)$ be a quadratic monomial presentation, and let~$A$ be the
algebra $\kk Q/I$ that it presents.
\begin{thmlist}

\item The vector space~$\HH_0(A)$ is freely spanned by the homology
classes of
\begin{itemize}

\item the pairs $(e,e)$, one for each vertex~$e$ of~$Q$, 

\item the pairs $(\bar C,s(\bar C))$, one for each cocomplete
circuit~$C$ of~$Q$, and

\item the pairs $(a,s(a))$, one for each loop~$a$ in~$Q$ such that
$a^2\in I$.

\end{itemize}

\item The vector space~$\HH_1(A)$ is freely spanned by the homology
classes of
\begin{itemize}

\item the sums
  \[
   \hcycle{\bar C}
        \coloneqq \sum_{i=0}^{r-1}(\rfact1{\rot^i(\bar C)},
                                   \rfact2{\rot^i(\bar C)})
        \in\kk(\cB\odot\Gamma_1),
  \]
one for each cocomplete circuit~$C$ in~$Q$, with $r$ the period
of~$C$,

\item the pairs $\hcycle{a}\coloneqq(s(a),a)$ with $a$ a loop in~$Q$
such that $a^2\in I$,

\item the pairs $(\lfact1{\bar C},\lfact2{\bar C})$, one for each
complete ciruit~$C$ of~$Q$ of length~$2$ and period~$r$ are such that
$(-1)^{r}=1$ in~$\kk$.

\end{itemize}

\item If $m\geq2$, the vector space~$\HH_m(A)$ is freely spanned by
the homology classes of
\begin{itemize}

\item the sums 
  \[
  \hcycle{\bar C}  \coloneqq
    \sum_{i=0}^{r-1}
      (-1)^{(m +1)i}\cdot
      \bigl(s(\rot^i(\bar C)), \rot^i(\bar C)\bigr)   
      \in \kk(\cB\odot\Gamma_m),
  \]
one for each complete circuit~$C$ whose length~$m$ and period~$r$ are
such that $(-1)^{(m+1)r}=1$ in~$\kk$, and

\item the pairs $(\lfact1{\bar C},\lfact2{\bar C})$, one for each
complete circuit~$C$ of length $m+1$ and period~$r$ such that
$(-1)^{mr}=1$ in~$\kk$.

\end{itemize}
\end{thmlist}
\end{theorem}

\begin{proof}
The complex~$\kk(\cB\odot\Gamma)$ decomposes as a direct sum
$\bigoplus_{C\in\C'}\kk(\cB\odot\Gamma)_C$. The homology of the
complex~$\kk(\cB\odot\Gamma)_C$ is described by
Lemmas~\ref{lemma:homology:complete}, \ref{lemma:homology:cocomplete}
and~\ref{lemma:homology:neither} when $C$ is a circuit that is
complete, or cocomplete, or neither, respectively. On the other hand,
if $C$ is an element of~$\C'$ of the form~$\{e\}$ with $e$ a vertex
of~$Q$, then it is obvious that the complex~$\kk(\cB\odot\Gamma)_C$
has homology of dimension one, concentrated in degree~$0$, and spanned
by the class of the pair~$(e,e)$. The theorem follows immediately from
all~this.
\end{proof}

In this theorem we have described the graded vector space~$\HH_*(A)$
degree by degree. It is convenient for many purposes to have at hand a
`transposed' description, organized instead by the parameters that
determine the cycles, as follows:


\newpage 
\begin{corollary}\label{coro:homology:basis:tr}\pushQED{\qed}
Let $(Q,I)$ be a quadratic monomial presentation and let $A$ be the
algebra~$\kk Q/I$. The graded vector space~$\HH_*(A)$ is freely
spanned by the homology classes of the following elements of the
complex~$\kk(\cB\odot\Gamma)$:
	\begin{itemize}
		\item For each vertex~$e$ in~$Q$, the $0$-cycle \[(e,e).\]

		\item For each complete circuit~$C$ in~$(Q,I)$ whose length~$m$ and
		period~$r$ are such that $(-1)^{(m+1)r}=1$ in~$\kk$, the $m$-cycle
	 		\[\hcycle{\bar C} \coloneqq
		  		\sum_{i=0}^{r-1}
	      		(-1)^{(m+1)i}\cdot
	      		\bigl(s(\rot^i(\bar C)), \rot^i(\bar C)),\]
		and the $(m-1)$-cycle
	  		\[(\lfact1{\bar C},\lfact2{\bar C}).\]

		\item For each cocomplete circuit~$C$ in~$(Q,I)$ of length~$r$, the
		$1$-cycle
		 	\[\hcycle{\bar C} \coloneqq 
 				\sum_{i=0}^{r-1}
 		    	(\rfact1{\rot^i(\bar C)}, \rfact2{\rot^i(\bar C)}) \]
		and the $0$-cycle
 	 	\[(\bar C,s(\bar C)).\] 
	\end{itemize}
\end{corollary}

\section{Some consequences}





We can extract some qualitative information about our algebras from the
Hochschild homology that we computed in the previous section.
We start by looking at what the finite-dimensionality of that
homology spaces means:

\begin{proposition}
Let $(Q,I)$ be a quadratic monomial presentation, and let $A\coloneqq\kk
Q/I$ be the corresponding quadratic monomial algebra. 
\begin{thmlist}

\item The following statements are equivalent:
\begin{tfae}

\item The algebra~$A$ is finite-dimensional.

\item $\HH_0(A)$ is a finite-dimensional vector space.

\item $\HH_1(A)$ is a finite-dimensional vector space.

\end{tfae}

\item For every $m\in\NN_0$ such that $m\geq2$ we have that
$\dim\HH_m(A)<\infty$.

\end{thmlist}
\end{proposition}

\begin{proof}
\thmitem{1} If the algebra~$A$ is infinite-dimensional, then there exists a
cocomplete circuit~$C$ in~$(Q,I)$, and in that case if $\bar C$ is an
element of~$C$ then the homology classes in $\HH_1(A)$ of the $1$-cycles
$\hcycle{\bar C^k}$ are linearly independent, so that
$\dim\HH_1(A)=\infty$, and the homology classes in~$\HH_0(A)$ of the
$0$-cycles~$(\bar C^k,s(\bar C^k)$ are linearly independent, so that
$\dim\HH_0(A)=\infty$. This shows that the statements~\tfaeitem{2}
and~\tfaeitem{3} both imply the statement~\tfaeitem{1}. 

Let us suppose now that the algebra $A$ is finite-dimensional. In that case
there are no cocomplete cycles in~$(Q,I)$: since there are finitely many
vertices and loops in~$Q$, we see from then first part of
Theorem~\ref{thm:basishomology} that $\dim\HH_0(A)<\infty$. Similarly,
since there are no cocomplete cycles, and the number of loops in~$A$ and
the number of complete circuits of length~$2$ in~$(Q,I)$ are both finite
the second part of that theorem tells us that $\dim\HH_1(A)<\infty$. We
thus see that the statement~\tfaeitem{1} implies the other two.

\thmitem{2} Let $m$ be an integer such that $m\geq2$.
According to the third part of Theorem~\ref{thm:basishomology},
$\HH_m(A)$ is spanned by classes that correspond to some complete circuits
of length~$m$ or~$m+1$, and the number of such circuits is finite: this
implies at once that $\dim\HH_m(A)<\infty$.
\end{proof}

The Hochschild homology spaces of a finite-dimensional algebra are all
finite-dimensional, and it make sense to ask what can be say about about
their dimensions. First of all, we can describe easily the algebras for
which the Hochschild homology spaces are eventually all zero:

\begin{proposition}
Let $(Q,I)$ be a quadratic monomial presentation. If the quotient algebra
$A\coloneqq\kk Q/I$  is finite-dimensional, then
the following statements are equivalent:
\begin{tfae}

\item The algebra~$A$ has finite global dimension.

\item $\HH_m(A)=0$ for all $m\geq1$.

\item There exists $m_0\in\NN_0$ such that
$\HH_m(A)=0$ for all $m\geq m_0$.

\item There exists $m_0\in\NN_0$ such that the space
$\bigoplus_{m\geq m_0}\HH_m(A)$ is finite-dimensional.

\end{tfae}
\end{proposition}

\begin{proof}
Let us suppose that $A$ is finite-dimensional.
If it also has finite global dimension, then there are no
complete circuits in~$(Q,I)$, and it follows from the second and third part
of Theorem~\ref{thm:basishomology} that $\HH_m(A)=0$ whenever $m\geq1$.
This shows that the implication~\implication{1}{2} holds, and the
implications~\implication{2}{3} and~\implication{3}{4} are obvious.

Suppose now that \tfaeitem{1} does not hold, so that $A$ has infinite
global dimension and there is in~$(Q,I)$ a primitive complete circuit~$C$.
Let $r$ be the length of~$C$, and let $\bar C$ be an element of~$C$.
According to the third part of Theorem~\ref{thm:basishomology}, the
homology class of the cycle $\hcycle{\bar C^{2k+1}}$ in~$\HH_{(2k+1)r}(A)$
is non-zero for all $k\in\NN$, and therefore the statement~\tfaeitem{4}
does not hold. We thus see that the implication \implication{4}{1} is also
true.
\end{proof}

When the
Hochschild homology spaces are not all eventually zero, we can look at the
rate of growth, if any, of their dimensions.

\begin{question}
We would like to prove that given a $(Q,I)$ quadratic monomial presentation such that quotient algebra
$A\coloneqq\kk Q/I$  is finite-dimensional, then
\begin{itemize}

\item either there exists an integer~$d$ such that $\dim\HH_m(A)\leq d$ for
all $m\in\NN_0$, 

\item or the sequence $(\dim\HH_m(A))_{m\geq0}$ growths exponentially,
\end{itemize}
and that if the presentation is gentle, then the first possibility occurs.
\end{question}

Let us suppose that $A$ is finite-dimensional and first that no two
different primitive complete circuits share an arrow. The number of
primitive complete circuits is then finite: let  $c$ be their number and
let $l$ be the minimum of their lengths. Since every complete cycle is, in
a unique way, a power of a primitive complete cycle, for every number~$L$
the number of complete circuits of length at most~$L$ is at most $cL/l$.
Using this and Theorem~\ref{thm:basishomology} we can see that there are
integers~$d$ and~$e$ such that
$\sum_{m=0}^n\dim\HH_m(A)\leq dn+e$ for all $n\in\NN_0$, and therefore
$\dim\HH_m(A)\leq d$ for all $m\in\NN_0$. We are thus in the first case
described in the proposition.

Let us suppose now thar $A$ is finite-dimensional and that there exist two
primitive complete circuits that share an arrow. There are thus primitive
complete cycles $C$ and $D$ in~$(Q,I)$ that start with the same arrow~$a$,
so that there are paths $C'$ and $D'$ such that $C=C'a$ and $D=D'a$. For
every choice of $E_1$,~\dots,~$E_k$ in the set~$\{C',D'\}$ the path
$E_1aE_2a\cdots E_ka$ is a complete cycle in~$(Q,I)$. 
We still need to ensure first that these are different circuits,  and that they satisfy the condition on lengths and periods that
appears in the theorem.

\section{The Connes boundary map}

In this section we will compute the the map 
  \begin{equation} \label{eq:B}
  B:\HH_n(A)\to\HH_{n+1}(A)
  \end{equation}
induced on the Hochschild homology of our quadratic monomial
algebra~$A$ by the Connes boundary map. We will use this information
in the following section to compute the cyclic homology of~$A$. We
start by briefly recalling the relevant definitions and context needed
for this calculation from \cite[Section 2.1.7]{Loday:cyclic}.

\bigskip

In this chapter we have realized the Hochschild homology~$\HH_*(A)$
of~$A$ as the homology of the complex $A\otimes_{A^e}\cR$ obtained
from the Barzdell projective resolution~$\cR$ of~$A$ as an
$A$-bimodule, but in order to compute the Connes boundary map we will
have to consider a different realization of that Hochschild homology,
constructed from a different projective resolution. For each
non-negative integer~$p$ we consider the $A$-bimodule
  \[ 
  {\mathbf B}_pA\coloneqq A\otimes_EA^{\otimes_Ep}\otimes_EA.
  \]
It is the component of degree~$p$ of a chain complex~${\mathbf B}A$, the
\newterm{bar resolution} of~$A$ relative to~$E$, whose differential
$d_p:{\mathbf B}_pA\to {\mathbf B}_{p-1}A$ is such that
  \begin{equation} \label{eq:d-hoch}
  d_p(a_0\otimes\cdots\otimes a_{p+1})
    = \sum_{i=0}^p(-1)^i\cdot
       a_0\otimes\cdots\otimes a_ia_{i+1}\otimes\cdots\otimes a_{p+1}
  \end{equation}
for all $p\in\NN_0$ and all choices of~$a_0$,~\dots,~$a_{p+1}$ in~$A$.
The complex~${\mathbf B}A$ is a projective resolution of~$A$ as an $A$-bimodule,
with augmentation $\epsilon:{\mathbf B}_0A=A\otimes_EA\to A$ the unique map of
$A$-bimodules such that $\epsilon(1\otimes1)=1$. Of course, the
homology of the complex~$A\otimes_{A^e}{\mathbf B}A$ is therefore canonically
isomorphic to the Hochschild homology $\HH_*(A)$ of~$A$ and, in
particular, to our concrete realization of it in terms of the Barzdell
resolution. Explicitly, this means that if $F:\cR\to {\mathbf B}A$ is any
morphism of complexes of $A$-bimodules compatible with the
augmentations over~$A$ of~$\cR$ and of~${\mathbf B}A$, then the morphism
$\id_A\otimes F:A\otimes_{A^e}\cR\to A\otimes_{A^e}{\mathbf B}A$ induces an
isomorphism in homology, 
  \begin{equation} \label{eq:araba}
  H_*(A\otimes_{A^e}\cR)\to H_*(A\otimes_{A^e}{\mathbf B}A),
  \end{equation}
that is the canonical one between the two realizations of the
Hochschild homology of~$A$ as the homology of~$A\otimes_{A^e}\cR$ and
as the homology of~$A\otimes_{A^e}{\mathbf B}A$.

The complex~$A\otimes_{A^e}{\mathbf B}A$ is often presented in a different way.
The \newterm{Hochschild complex} $C(A,A)$ is the chain complex that
has for each $p\in\NN_0$ component of degree~$p$ given by the vector
space
  \[
  C_p(A,A)\coloneqq A\otimes_EA^{\otimes_Ep}
  \]
and for each $p\in\NN$ differential $d_p:C_p(A,A)\to C_{p-1}(A,A)$
such that
  \begin{multline*}
  d_p(a_0\otimes\cdots\otimes a_p) \\
        = \sum_{i=0}^{p-1}(-1)^i\cdot a_0\otimes\cdots\otimes
              a_ia_{i+1}\otimes\cdots\otimes a_p
          + (-1)^p\cdot a_pa_0\otimes a_1\otimes\cdots\otimes a_{p-1}
  \end{multline*}
whenever $a_0$,~\dots,~$a_p$ in~$A$. There is an isomorphism of
complexes 
  \[
  \rho:A\otimes_{A^e}{\mathbf B}A\to C(A,A)
  \]
that for each $p\in\NN_0$ has component
  \(
  \rho_p:
        A\otimes_{A^e}(A\otimes_E A^{\otimes_Ep}\otimes_E A) 
        \to
        A^{\otimes_E(p+1)}
  \)
such that
  \[
  \rho_p(a\otimes(a_0\otimes a_1\otimes\cdots\otimes a_p\otimes a_{p+1}))
  = a_{p+1}aa_0\otimes a_1\otimes\cdots\otimes a_p
  \]
whenever $a$,~$a_0$,~\dots,~$a_p$ are in~$A$, and it induces on
homology an isomorphism
  \begin{equation} \label{eq:bacaa}
  H_*(A\otimes_{A^e}{\mathbf B}A) \to H_*(C(A,A)).
  \end{equation}
This gives us a third realization of the Hochschild homology of~$A$,
now as the homology of the Hochschild complex~$C(A,A)$.

The interest of this third realization is that it is on it that the
Connes boundary map is defined. The \newterm{Connes boundary map} is
the morphism of complexes 
  \[
  B:C(A,A)\to C(A,A)[1]
  \]
that for each choice $p\in\NN_0$ and $a_0$,~\dots,~$a_p\in A$ has
  \begin{multline} \label{eq:B-p}
  B_p(a_0\otimes\cdots\otimes a_p) 
    = \sum_{i=0}^{p}
        \Bigl(
          (-1)^{ni}
          1\otimes a_i\otimes\cdots\otimes a_p
          \otimes a_0\otimes\cdots\otimes a_{i-1}
          \\
          - (-1)^{ni}
          a_{i-1}\otimes 1\otimes a_i\otimes\cdots\otimes a_p
          \otimes a_0\otimes\cdots\otimes a_{i-2}
        \Bigr).
  \end{multline}
In particular, in low degrees we have for all $a_0$,~$a_1\in A$ that
  \begin{gather}
  B_0(a_0) = 1\otimes a_0+a_0\otimes 1, \label{eq:B-0}\\
  B_1(a_0\otimes a_1) = 1\otimes a_0\otimes a_1
                      - 1\otimes a_1\otimes a_0
                      + a_0\otimes 1\otimes a_1
                      - a_1\otimes 1\otimes a_0.
                      \label{eq:B-1}
  \end{gather}
A key fact is that the Connes boundary map~$B$ is a boundary: we have
$B\circ B=0$. It induces a map $H(C(A,A))\to H(C(A,A)[1])$ and the
map~\eqref{eq:B} that we want to compute is the result of conjugating
the latter by the composition of the isomorphisms~\eqref{eq:araba}
and~\eqref{eq:bacaa}, which is a map
  \[
  B : H_*(A\otimes_{A^e}\cR) \to H_*(A\otimes_{A^e}\cR)[1].
  \]
To do this, we will need an explicit comparison map $F:\cR\to {\mathbf B}A$ from
the Barzdell resolution to the bar resolution in order to compute the
isomorphism~\eqref{eq:araba} and an explicit comparison map in the
other direction, $G:{\mathbf B}A\to\cR$, to compute the inverse isomorphism. The
first one poses no problem:

\begin{lemma}
There is a morphism of complexes of $A$-bimodules $F:\cR\to {\mathbf B}A$
compatible with the augmentations of the resolutions~$\cR$ and~${\mathbf B}A$
of~$A$ such that for each $m\in\NN_0$ the component 
  \[
  F_m:
    \cR_m = A\otimes_E\kk\Gamma_m\otimes_EA
    \to 
    {\mathbf B}_mA = A\otimes_EA^{\otimes_E m}\otimes_EA
  \]
has
  \[
  F_m(1\otimes\gamma\otimes 1)
    = 1\otimes c_m\otimes c_{m-1}\otimes\cdots\otimes c_1\otimes 1
  \]
whenever $\gamma=c_m\cdots c_1\in\Gamma_m$.
\end{lemma}

The morphism~$F$ described here is injective, and can be viewed as the
inclusion of the bimodule Koszul complex relative to~$E$ for the
algebra~$A$, which is certainly a Koszul quadratic algebra, into the
bar resolution --- we refer to the book~\cite{PP} of P.\,Polishchuk
and L.\,Positselski for more information about this point of view.

\begin{proof}
This can be proved by an easy direct calculation.
\end{proof}

Comparison morphisms ${\mathbf B}A\to\cR$ in the other direction are more
complicated to describe, since they depend inevitably on the
combinatorics of the set of quadratic monomials that define the
algebra~$A$. We will use the morphism constructed by L.\,Román and
M.J.\,Redondo in~\cite{RR:comparison} and which they describe in
Section~3.2 of that paper. In fact, we will luckily not need to know
the map in full detail: the following proposition describes what we do
need.

\begin{proposition}\label{prop:comp:G}
There exists a morphism $G:{\mathbf B}A\to\cR$ of complexes of $A$-bimodules
satisfying the following properties:
\begin{itemize}

\item It is homogeneous with respect to the grading in~${\mathbf B}A$ and~$\cR$
given by the length of paths.

\item Its components of degree~$0$ and~$1$ are such that
  \begin{gather*}
  G_0(1\otimes 1)
    = \sum_{i\in Q_0}
        1\otimes e_1\otimes 1
\shortintertext{and}
  G_1(1\otimes \alpha\otimes 1)
    = \sum_{i=1}^m
        a_m\cdots a_{i+1}
        \otimes a_i
        \otimes a_{i-1}\cdots a_1
  \end{gather*}
whenever $\alpha=a_m\cdots a_1$ is an element of~$\cB$.

\item If $\gamma=c_m\cdots c_1$ is the factorization as a product of arrows
of path in~$\Gamma$, then 
  \[
  G_m(1\otimes c_m\otimes\cdots\otimes c_1\otimes 1)
    = 1\otimes\gamma\otimes1.
  \]

\item If $\gamma$ is a path in~$\cB$ and
$\gamma = \gamma_m\cdots\gamma_1$ is a factorization of~$\gamma$ as a
product of paths with $m\geq2$, then 
  \[
  G_m(1\otimes\gamma_1\otimes\cdots\otimes\gamma_1\otimes1) = 0.
  \]

\end{itemize}
\end{proposition}

\medskip

With this information we can compute directly the action of the Connes
boundary map on each of the elements of the basis of the Hochschild
homology of a quadratic monomial algebra that we described in
Theorem~\ref{thm:basishomology} and in
Corollary~\ref{coro:homology:basis:tr}

\begin{proposition}\label{prop:B}
Let $(Q,I)$ be a quadratic monomial presentation and let $A$ be the
algebra~$\kk Q/I$. 
\begin{itemize}

\item For each vertex~$e$ in~$Q$ we have that
  \(
  B(e,e) = 0
  \).

\item For each complete circuit~$C$ in~$(Q,I)$ whose length~$m$ and
period~$r$ are such that $(-1)^{(m+1)r}=1$ in~$\kk$ we have
  \begin{equation} \label{eq:bb:1}
  B\hcycle{\bar C} = 0,
  \qquad
  B(\lfact1{\bar C},\lfact2{\bar C}) = \frac{m}{r}\hcycle{\bar C}.
  \end{equation}

\item For each cocomplete circuit~$C$ in~$(Q,I)$ of length~$m$ and
period~$r$ we have
  \begin{equation} \label{eq:bb:2}
  B\hcycle{\bar C} = 0,
  \qquad
  B(\bar C,s(\bar C)) = \frac{m}{r}\hcycle{\bar C}.
  \end{equation}

\end{itemize}
\end{proposition}

\begin{proof}
The map~$B$ that we want to compute is induced in each
degree~$m$ of the homology by the composition  
  \begin{equation} \label{eq:bbb}
  \begin{tikzcd}
  \kk(\cB\odot\Gamma_m) \arrow[r, equal]
    &[-1.5em] A\otimes_{A^e}\cR_m \arrow[r, "\id_A\otimes F_m"]
    &[2.5em] A\otimes_{A^e}{\mathbf B}_mA \arrow[r, "\rho_m"]
    & C_m(A,A) \arrow[d, "B_m"]
    \\
  \kk(\cB\odot\Gamma_{m+1})
    & A\otimes_{A^e}\cR_{m+1} \arrow[l, equal]
    & A\otimes_{A^e}{\mathbf B}_{m+1}A \arrow[l, "\id_A\otimes G_{m+1}"]
    & C_{m+1}(A,A) \arrow[l, "\rho_{m+1}^{-1}"]
  \end{tikzcd}
  \end{equation}
Let us first take $m=0$, and consider a vertex~$e\in Q_0$. Chasing the image 
of the $0$-cycle~$(e,e)$ along this composition gives:
  \[
  \begin{tikzcd}
  (e,e) \arrow[r, mapsto]
    & e\otimes(1\otimes e\otimes 1) \arrow[r, mapsto]
    & e\otimes(1\otimes e\otimes 1) \arrow[r, mapsto]
    & e \arrow[d, mapsto]
    \\
  0
    & 0 \arrow[l, mapsto]
    & \begin{gathered}
      1\otimes(1\otimes e\otimes 1) \\
      {}+e\otimes(1\otimes 1\otimes 1)
      \end{gathered} \arrow[l, mapsto]
    & 1\otimes e+e\otimes 1 \arrow[l, mapsto]
  \end{tikzcd}
  \]
The key point here is that the map~$G$ vanishes on the
elements~$1\otimes e\otimes 1$ and $1\otimes 1\otimes 1$ of~${\mathbf B}_1A$:
indeed, $G$ is homogeneous with respect to length, these elements have
total length~$0$, and the homogeneous component
of~$A\otimes_{A^e}\cR_1$ of length~$0$ is~$0$. This proves the first
claim of the proposition.

Let now $C$ be a complete circuit in~$(Q,I)$, let $m$ and~$r$ be its
length and its period, and let us suppose that $(-1)^{(m+1)r}=1$
in~$\kk$, so that we have, respectively, in~$\HH_m(A)$ and~$\HH_{m-1}(A)$ the
homology classes of the cycles
  \[
  \hcycle{\bar C} \coloneqq
        \sum_{i=0}^{r-1}
              (-1)^{(m+1)i}\cdot
              \bigl(s(\rot^i(\bar C)), \rot^i(\bar C)),
  \qquad
  \qquad
  (\lfact1{\bar C},\lfact2{\bar C}).
  \]
The cycle~$\hcycle{\bar C}$ is weight-homogeneous of weight~$m$ and
all the maps in the diagram above preserve weight: the image
of~$\hcycle{\bar C}$ by the composition is zero because it is a
weight-homogeneous element of~$\kk(\cB\odot\Gamma_{m+1})$ of
weight~$m$, and thus $B\hcycle{\bar C}=0$.

Let us suppose that $\bar C=c_m\cdots c_1$ is the factorization
of~$\bar C$ as a product of arrows, so that $\lfact1{\bar C}=c_m$ and
$\lfact2{\bar C}=c_{m-1}\cdots c_1$. The image of~$(\lfact1{\bar
C},\lfact2{\bar C})$ by the map $\kk(\cB\odot\Gamma_{m-1})\to
C_{m-1}(A,A)$ in the diagram above is the elementary tensor
$c_m\otimes\cdots\otimes c_1$, and the image of this by the
map~$\rho_m^{-1}\circ B_{m-1}$ in the diagram is therefore
{ 
  \medmuskip=1mu\relax
  \begin{multline*}
          \sum_{i=0}^{m-1}
             \Biggl(
               (-1)^{(m-1)i}
               1\otimes(1\otimes c_{m-i}\otimes\cdots\otimes c_1
               \otimes c_m\otimes\cdots\otimes c_{m-i+1}\otimes1)
               \\
               - (-1)^{(m-1)(i-1)}
               c_{m-i+1}\otimes(1\otimes 1\otimes c_{m-i}
               \otimes\cdots\otimes c_1
               \otimes c_m\otimes\cdots\otimes c_{m-i+2}\otimes1)
             \Biggr).
  \end{multline*}
}
For each $i\in\{0,\dots,m-1\}$ the tensor
  \[
  1\otimes 1\otimes c_{m-i}
               \otimes\cdots\otimes c_1
               \otimes c_m\otimes\cdots\otimes c_{m-i+2}\otimes1
  \]
in~${\mathbf B}_mA$ is weight-homogeneous of weight~$m-1$, and therefore its
image by~$G_m$ in~$\cR_m$ is zero, while that of
  \[
  1\otimes c_{m-i}\otimes\cdots\otimes c_1
               \otimes c_m\otimes\cdots\otimes c_{m-i+1}\otimes1
  \]
is 
  \[
  1\otimes c_{m-i}\cdots c_1c_m\cdots c_{m-i+1}\otimes1
  = 1\otimes \rot^i(\bar C)\otimes 1
  \]
precisely because of the third property of the map~$G$ described in
Proposition~\ref{prop:comp:G}. It follows from this that the image
under the Connes map of the cycle~$(\lfact1{\bar C},\lfact2{\bar C})$ is
  \[
  \sum_{i=0}^{m-1}
    (-1)^{(m-1)i}\cdot (s(\rot^i\bar C),\rot^i(\bar C)),
  \]
which is exactly~$\hcycle{\bar C}$. This proves the second equality
in~\eqref{eq:bb:1}.

Finally, let $C$ be a cocomplete circuit in~$(Q,I)$ of period~$r$, to
which correspond a $0$-cycle and a $1$-cycle,
  \begin{equation} \label{eq:hcb}
  (\bar C,s(\bar C)),
  \qquad
  \qquad
  \hcycle{\bar C} 
        \coloneqq \sum_{i=0}^{r-1}(\rfact1{\rot^i(\bar C)},
                                   \rfact2{\rot^i(\bar C)}).
  \end{equation}
Let $\bar C=c_m\cdots c_1$ be the factorization of~$\bar C$ as a
product of arrows. The image of the cycle~$(\bar C,s(C))$ under the
composition $\kk(\cB\odot\Gamma_0)\to C_0(A,A)$ in the
diagram~\eqref{eq:bbb} is $\bar C$, and according to~\eqref{eq:B-0}
this is mapped by~$\rho_1^{-1}\circ B_0$ to 
  \[
  s(\bar C)\otimes(1\otimes \bar C\otimes1)
  + \bar C\otimes(1\otimes s(\bar C)\otimes1).
  \] 
The morphism $G_1$ maps $1\otimes s(\bar C)\otimes 1$, a weight-homogeneous
element of weight zero, to zero, because there are no other elements
of that weight in~$\cR_1$, so $\id_A\otimes G_1$ vanishes on~$\bar
C\otimes(1\otimes s(\bar C)\otimes1)$. On the other hand, the
description of~$G_1$ in Proposition~\ref{prop:comp:G} tells us that
  \[
  G_1(1\otimes\bar C\otimes 1)
        = \sum_{i=1}^m
          c_m\cdots c_{i+1}\otimes c_i\otimes c_{i-1}\cdots c_1.
  \]
It follows from this at once that the image of $\bar C,s(\bar C))$
under the composition of the morphisms in the diagram~\eqref{eq:bbb} is
  \begin{align*}
  \MoveEqLeft
  \sum_{i=1}^m(c_{i-1}\cdots c_1c_m\cdots c_{i+1},c_i) \\
    & = \sum_{i=1}^m(\rfact1{\rot^i(\bar C)},\rfact2{\rot^i(\bar C)})
      = \frac{m}{r}
        \sum_{i=0}^{r-1}(\rfact1{\rot^i(\bar C)},\rfact2{\rot^i(\bar C)})
      = \frac{m}{r}\hcycle{\bar C}.
  \end{align*}
This proves the second equality in~\eqref{eq:bb:2}.

If $i\in\{0,\dots,r-1\}$, then the image of $(\rfact1{\rot^i(\bar
C)},\rfact2{\rot^i(\bar C)})\in\kk(\cB\odot\Gamma_1)$ by the
composition that ends in~$C_1(A,A)$ in the diagram~\eqref{eq:bbb} is
$\rfact1{\rot^i(\bar C)}\otimes \rfact2{\rot^i(\bar C)}$, which in
turn is mapped by~$\rho_2^{-1 }\circ B_1$ to
  \begin{multline*}
  1\otimes
        (1\otimes\rfact1{\rot^i(\bar C)}
        \otimes\rfact2{\rot^i(\bar C)}\otimes1)
  - 1\otimes
        (1\otimes\rfact2{\rot^i(\bar C)}
        \otimes \rfact1{\rot^i(\bar C)}\otimes 1) \\
  + \rfact1{\rot^i(\bar C)}
        \otimes(1\otimes1\otimes \rfact2{\rot^i(\bar C)}\otimes1)
  - \rfact2{\rot^i(\bar C)}\otimes
        (1\otimes1\otimes \rfact1{\rot^i(\bar C)}\otimes1).
  \end{multline*}
The last property of the morphism~$G$ described in
Proposition~\ref{prop:comp:G} implies immediately that $\id_A\otimes G_2$
vanishes on these elements, and therefore that the image under~$B$ of the
cycle~$\hcycle{\bar C}$ in~\eqref{eq:hcb} is zero. This completes the
proof of the proposition.
\end{proof}

\section{The cyclic homology of quadratic monomial algebras}

As promised, we will use now the Connes boundary map that we
described in the previous section to compute the cyclic homology of our quadratic
monomial algebras. We use for this the Connes spectral sequence that
goes from Hochschild homology to cyclic homology, and start by
recalling what we need about it.

\bigskip

There is a first quadrant homologically indexed double complex
$\BC(A,A)$ of vector spaces with
  \[
  \BC_{p,q}(A,A) = \begin{cases*}
                  C_{q-p}(A,A) & if $0\leq p\leq q$; \\
                  0 & in any other case
                  \end{cases*}
  \]
with
\begin{itemize}

\item vertical differentials
  \(
  d^v_{p,q}:\BC_{p,q}(A,A)\to \BC_{p,q-1}(A,A)
  \)
given by the differential~$d_{q-p}:C_{q-p}(A,A)\to C_{q-p-1}(A,A)$ of the
Hochschild complex $C(A,A)$ and 

\item vertical differentials
  \(
  d^h_{p,q}:\BC_{p,q}(A,A)\to \BC_{p-1,q}(A,A)
  \)
given by the Connes boundary map $B:C_{q-p}(A,A)\to C_{q-p+1}(A,A)$.

\end{itemize}
We have drawn the bottom left corner of this complex in
Figure~\ref{fig:cyclic}. The homology of the total complex
of~$\BC(A,A)$ is canonically isomorphic, according to \cite[Theorem
2.1.8]{Loday:cyclic}, to the cyclic homology~$\HC_*(A)$ of the
algebra~$A$. The standard filtration by columns of this double
complex, on the other hand, gives rise to a spectral sequence ---
usually referred to as the \newterm{Connes spectral sequence} --- that
converges canonically to~$\HC_*(A)$ and has first page such that
  \[
  E^1_{p,q} 
    = \begin{cases*}
      \HH_{q-p}(A) & if $0\leq p\leq q$; \\
      0 & in any other case.
      \end{cases*}
  \]
Moreover, the differential on this page is precisely the map induced by the
Connes boundary map on Hochschild homology,
  \[
  \begin{tikzcd}
  E^1_{p,q} \arrow[r, "d^1_{p,q}"] \arrow[d, equal]
    & E^1_{p-1,q} \arrow[d, equal]
    \\
  \HH_{q-p}(A) \arrow[r, "B"]
    & \HH_{q-p+1}(A)
  \end{tikzcd}
  \]
All this is Theorem~1.9 in the paper~\cite{LQ} by D.\,Quillen and
\mbox{J.-L.}\,Loday. We will use this spectral sequence to compute the
cyclic homology of~$A$.

\begin{figure}
  \centering
  \begin{tikzcd}[
          every matrix/.append style={name=m},
          execute at end picture={
            \begin{scope}[commutative diagrams/.cd, every arrow, every label]
              \foreach \b [remember=\b as \a (initially 1)] in {2,3,4,5,6}
                \foreach \x in {1,2,3,4,5}
                  {
                  \ifnum\numexpr 7-\x>\a\relax
                    \draw[->] (m-\a-\x) -- (m-\b-\x);
                  \fi
                  }
              \foreach \b [remember=\b as \a (initially 1)] in {2,3,4,5,6}
                \foreach \x in {2,3,4,5}
                  {
                  \ifnum\numexpr 7-\x>\a\relax
                    \draw[<-] (m-\x-\a) -- (m-\x-\b);
                  \fi
                  }
            \end{scope}
            }
          ]
    \vdots & \vdots & \vdots & \vdots & \vdots \\
    A^{\otimes5} & A^{\otimes4} & A^{\otimes3} & A^{\otimes2} & A^{\otimes1} \\
    A^{\otimes4} & A^{\otimes3} & A^{\otimes2} & A^{\otimes1} & {} \\
    A^{\otimes3} & A^{\otimes2} & A^{\otimes1} & {}   & {} \\
    A^{\otimes2} & A^{\otimes1} & {}   & {}   & {} \\
    A^{\otimes1} & {}   & {}   & {}   & {}  
  \end{tikzcd}
\caption{The cyclic complex}
\label{fig:cyclic}
\end{figure}

\bigskip

We start by computing the second page of the spectral sequence. In
view of the obvious translational symmetry of the spectral sequence,
this amounts to the determination of the homology of the complex
  \[
  \begin{tikzcd}
  \HH_0(A) \arrow[r, "B_0"]
    & \HH_1(A) \arrow[r, "B_1"]
    & \HH_2(A) \arrow[r, "B_2"]
    & \HH_3(A) \arrow[r, "B_3"]
    & \cdots
  \end{tikzcd}
  \]
and of  the cokernels of its differentials. Indeed, if we write
$\H_\dR^p(A)$ for the homology of this cochain complex at~$\HH_p(A)$,
then the second page of the Connes spectral sequence corresponding
to~$A$ has
  \[
  E^2_{p,q} = 
    \begin{cases*}
    \H_\dR^{q-p}(A) & if $0<p\leq q$; \\
    \coker B_{q-1} & if $0=p\leq q$; \\ 
    0 & in any other case.
    \end{cases*}
  \]
At this point in our work the calculation of~$\H_\dR^*(A)$ is easy:

\begin{proposition}
Let $(Q,I)$ be a quadratic monomial presentation, and let $A$ the
algebra~$\kk Q/I$ that it presents. The graded vector
space~$\H_\dR^*(A)$ is freely spanned by the homology classes of the
following elements of~$\kk(\cB\odot\Gamma)$:
\begin{itemize}

\item for each vertex~$e$ in~$Q$, the $0$-cycle $(e,e)$,

\item for each complete circuit~$C$ in~$(Q,I)$ whose length~$m$ and
period~$r$ are such that $(-1)^{(m+1)r}=1$ and $m/r=0$ in~$\kk$, the
$m$-cycle $\hcycle{\bar C}$ and the $(m-1)$-cycle $(\lfact1{\bar
C},\lfact2{\bar C})$,

\item for each cocomplete circuit~$C$ in~$(Q,I)$ of length~$m$ and
period~$r$ such that $m/r=0$ in~$\kk$, the $1$-cycle $\hcycle{\bar C}$
and the $0$-cycle $(\bar C,s(\bar C))$.

\end{itemize}
In particular, if the characteristic of the field~$\kk$ is zero, then
we have $\H_\dR^0(A)\cong E$ and $\H_\dR^p(A)=0$ for all positive
integers~$p$.
\end{proposition}

\begin{proof}
This is immediate given the information provided by
Corollary~\ref{coro:homology:basis:tr} about~$\HH_*(A)$ and
Proposition~\ref{prop:B} about the Connes boundary map~$B$.
\end{proof}

We have a similar description for the cokernel of the Connes boundary
map:

\begin{proposition}
Let $(Q,I)$ be a quadratic monomial presentation, and let $A$ the
algebra~$\kk Q/I$ that it presents. The cokernel of the map
$B:\HH_*(A)\to\HH_*(A)$ is freely spanned by the homology classes of
the following elements of~$\kk(\cB\odot\Gamma)$:
\begin{itemize}

\item for each vertex~$e$ in~$Q$, the $0$-cycle $(e,e)$,

\item for each complete circuit~$C$ in~$(Q,I)$ whose length~$m$ and
period~$r$ are such that $(-1)^{(m+1)r}=1$ in~$\kk$, the $(m-1)$-cycle
  \(
  (\lfact1{\bar C},\lfact2{\bar C}),
  \)
and, if additionally $m/r=0$ in~$\kk$, the $m$-cycle
  \(
  \hcycle{\bar C}
  \),

\item for each cocomplete circuit~$C$ in~$(Q,I)$ of length~$m$ and
period~$r$, the $0$-cycle
  \(
  (\bar C,s(\bar C))
  \)
and, if additionally $m/r=0$ in~$\kk$, the $1$-cycle
  \(
  \hcycle{\bar C} 
  \).

\end{itemize}
\end{proposition}

\begin{proof}
Again, this follows at once from
Corollary~\ref{coro:homology:basis:tr} and Proposition~\ref{prop:B}.
\end{proof}

These last two propositions taken together describe the second page of
the spectral sequence. The key fact that allows us to finish our 
calculation of cyclic homology is that this spectral sequence
degenerates at that point, as we shall presently see.

\begin{theorem}\label{thm:Connesspectral}
The Connes spectral sequence 
  \[
  E^1_{p,q} \cong\HH_{q-p}(A) \underset{p}{\Longrightarrow}
        \HC_*(A)
  \]
degenerates on its second page, and we therefore have isomorphisms
  \[
  \HC_m(A) = \coker B_m\oplus\bigoplus_{i\geq0}\H_\dR^{m-2i}(A).
  \]
\end{theorem}

It is remarkable that this description of the cyclic homology of the
algebra~$A$ is entirely similar to that of the cyclic homology of a
smooth commutative algebra that is essentially of finite type and
defined over a field of characteristic zero. The calculation of the
cyclic homology of these algebras was done by Loday and Quillen
in~\cite{LQ}, and the results appear as Proposition~2.3.7 and
Theorem~3.4.12 in Loday's book~\cite{Loday:cyclic} --- of course,
smooth algebras are non-singular objects, while our quadratic monomial
algebras are very far from that in general. The cyclic homology of
quadratic monomial algebras has been determined by Emil Sk\"{o}ldberg
\cite{Skoldberg:cyclic}, by essentially the same method that we have
used here.

\begin{proof}
Just as in~\eqref{eq:cprime} at the beginning of this chapter we let
$\C$ be the set of all circuits in the quiver~$Q$, and put $\C'
\coloneqq \C \cup \{\{e\} : e\in Q_0 \}$. There is a $\C'$-grading on
the Hochschild complex $C(A,A)$: if $C$ is an element of~$\C'$, then
the $C$-component $C(A,A)_C$ of~$C(A,A)$ is the span of all elementary
tensors of the form
$\gamma_0\otimes\gamma_1\otimes\cdots\otimes\gamma_p$ with $p\in\NN_0$
and $\gamma_1$,~$\gamma_1$,~\dots,~$\gamma_p$ paths in~$\cB$ such that
the product $\gamma_1\cdots\gamma_p$ belongs to the circuit~$C$. It is
clear from the formula~\eqref{eq:d-hoch} that $C(A,A)_C$ is indeed a
subcomplex of~$C(A,A)$ for each $C\in\C'$. Moreover, the
formula~\eqref{eq:B-p} for the Connes boundary map shows that
$B:C(A,A)\to C(A,A)[1]$ is homogeneous with respect to this
$\C'$-grading, so that the entire cyclic complex~$\BC(A,A)$ acquires
in this way a $\C'$-grading. Of course, this grading induces one on
the Connes spectral sequence and, in particular, all  the differentials of
the spectral sequence preserve it.

In proving Theorem~\ref{thm:basishomology} we computed $\HH_*(A)$ as
the homology of the complex~$\kk(\cB\odot\Gamma)$, and to do that we
used the fact that this complex also has a $\C'$-grading
$\kk(\cB\odot\Gamma)=\bigoplus_{C\in\C'}\kk(\cB\odot\Gamma)_C$. Using
the definitions of the morphisms~$F$ and~$\rho$ we can see immediately
that the composition
  \[
  \begin{tikzcd}
  \kk(\cB\odot\Gamma) \arrow[r, equal]
    &[-1.5em] A\otimes_{A^e}\cR \arrow[r, "\id_A\otimes F"]
    &[2.5em] A\otimes_{A^e}BA \arrow[r, "\rho"]
    & C(A,A) 
  \end{tikzcd}
  \]
is $\C'$-homogeneous, and since it is a quasi-isomorphism, this allows
us to describe the $\C'$-homogeneous components of the homology
of~$C(A,A)$ --- remember that Lemmas~\ref{lemma:homology:complete},
\ref{lemma:homology:cocomplete} and~\ref{lemma:homology:neither}
describe the homology of the $\C'$-homogeneous components
of~$\kk(\cB\odot\Gamma)$. Essentially the information is contained in
the statement of Theorem~\ref{thm:basishomology}. What interests us about
this now is the following: for each $C\in\C'$:
\begin{itemize}

\item we have $\HH_0(A)_C\neq0$ only if $C$ is either of length~$0$, or a
complete circuit in~$(Q,I)$ of length~$1$, or a cocomplete circuit;

\item we have $\HH_1(A)_C\neq0$ only if $C$ is either a complete
circuit in~$(Q,I)$ of length~$1$ or~$2$, or a cocomplete circuit; 

\item when $m\geq2$, we have $\HH_m(A)_C\neq0$ only if $C$ is a
complete cycle in~$(Q,I)$ whose length is either~$m$ or~$m+1$.

\end{itemize}
Since the spectral sequence is one of~$\C'$-graded vector spaces and
we can compute its page~$E^1$ in terms of $\HH_*(A)$ even as a
$\C'$-graded vector space, we can deduce that for all $r$,~$p$,~$q$
with $r\geq2$ and $0\leq p\leq q$ and all $C\in\C'$ we have that
$(E^r_{p,q})_C\neq0$ only if one of the following three conditions
holds:
\begin{itemize}

\item either $p=q$ and $C$ has length~$0$, 

\item or $0\leq q-p\leq 1$ and $C$ is cocomplete, 

\item or $0\leq q-p$ and $C$ is complete of length $q-p$ or~$q-p+1$,

\end{itemize}
simply because $(E^r_{p,q})_C$ is an iterated subquotient
of~$\HH_{q-p}(A)_C$. As a consequence of this we have that whenever
$r\geq2$ and $0\leq p\leq q$ the differential
  \[
  d^r_{p,q}:(E^r_{p,q})_C\to(E^r_{p-r,q+r-1})_C
  \]
is zero. Indeed:
\begin{itemize}

\item suppose first that $C$ is a complete circuit of length~$m$. If
the domain of the map is non-zero, then $m$ is either~$q-p$
or~$q-p+1$, and if the codomain of the map is non-zero, then $m$ is
either $q-p+2r-1$ or~$q-p+2r$: these two conditions cannot hold
simultaneously, so one of the two spaces is zero;

\item suppose next that $C$ is a cocomplete circuit. If the domain of
the map is non-zero, then $q-p$ is~$0$ or~$1$, and if the codomain is
non-zero, then $q-p+2r-1$ is~$0$ or~$1$. Again, these two conditions
are not compatible, so one of the two spaces is zero;

\item suppose finally that $C$ has length zero. If the domain of the
map is non-zero, then $p=q$, and if the codomain of the map is
non-zero, then $p-r=q+r-1$, and we see once again that one of the two
must vanish.

\end{itemize}
We thus see that the spectral sequence indeed degenerates on its
second page, and the theorem follows at once.
\end{proof}

\chapter{The cup and cap products}
\label{chapter:cup}

In this section,  we fix a 
gentle presentation~$(Q,I)$
and the corresponding algebra $A\coloneqq\kk Q/I$,  with the intention
of making explicit the associative algebra structure on the Hochschild
cohomology~$\HH^*(A)$ of~$A$ given by the cup product, and the action
of~$\HH^*(A)$ on the homology~$\HH_*(A)$ given by the cap product.

\section{The cup product}
\label{sect:cup:calculation}

There are several ways to compute the cup product on~$\HH^*(A)$. As we
computed the cohomology itself using the Bardzell resolution~$\cR$ of~$A$,
we want to do the calculation of the cup product without involving another resolution in order  to avoid having to use comparison maps. The strategy is as
follows. The complex~$\cR\otimes_A\cR$ is a projective resolution of~$A$ as
an $A$-bimodule, with augmentation~$\eta:\cR\otimes_A\cR\to A$ given by the
composition
  \[
  \begin{tikzcd}
  \cR\otimes_A\cR \arrow[r, "\epsilon\otimes\epsilon"]
    & A\otimes_AA \arrow[r, "\mu"]
    & A
  \end{tikzcd}
  \]
with $\epsilon:\cR\to A$ the augmentation of~$\cR$ over~$A$ and $\mu$ the
canonical isomorphism induced by the multiplication of~$A$. There is
therefore a morphism $\Delta:\cR\to\cR\otimes_A\cR$ of complexes of
$A$-bimodules making the diagram
  \[
  \begin{tikzcd}
  \cR \arrow[r, "\Delta"] \arrow[d, swap, "\epsilon"]
    & \cR\otimes_A\cR \arrow[d, "\eta"]
    \\
  A \arrow[r, equal]
    & A
  \end{tikzcd}
  \]
commute. If now $p$,~$q\geq0$ and $\phi:\cR_p\to A$ and $\psi:\cR_q\to A$
are a $p$- and a $q$-cochain in the complex~$\Hom_{A^e}(\cR,A)$, then the
cup product $\phi\smile\psi:\cR_{p+q}\to A$ is a $(p+q)$-cochain in the same
complex which is the composition
  \[
  \begin{tikzcd}
  \cR_{p+q} \arrow[r, "\Delta"]
    & \cR_p\otimes_A\cR_q \arrow[r, "\phi\otimes\psi"]
    & A\otimes_AA \arrow[r, "\mu"]
    & A.
  \end{tikzcd}
  \]
One can check that in this way we turn $\Hom_{A^e}(\cR,A)$ into a
differential graded algebra, and that the algebra structure on its
cohomology, which is canonically isomorphic to
$\HH^*(A)=\Ext_{A^e}^*(A,A)$, is the one given by the Yoneda product.

To carry this out, we will use the morphism $\Delta:\cR\to\cR\otimes_A\cR$
of complexes of $A$-bimodules given, for $m\geq0$ and  $\gamma=c_m\cdots c_1\in\Gamma_m$, by
  \[
  \Delta(1\otimes\gamma\otimes 1)
        = \sum_{i=0}^m
          (1\otimes c_m\cdots c_{i+1}\otimes1)
          \otimes
          (1\otimes c_{i}\cdots c_1\otimes1).
  \]
Under our
standard identifications $\Hom_{A^e}(\cR_m,A)=\kk(\Gamma_m\parallel\cB)$ of
Chapter~\ref{chapter:cohomology}, the resulting multiplication in the
complex $\kk(\Gamma\parallel\cB)$ is as follows: if $m$,~$n\geq0$ and
$(\gamma,\alpha)\in\Gamma_m\parallel\cB$ and
$(\delta,\beta)\in\Gamma_n\parallel\cB$, then
  \[
  (\gamma,\alpha)\smile(\gamma',\alpha')
        = \begin{cases*}
          (\gamma\gamma',\alpha\alpha')
                & if $\gamma\gamma'\in\Gamma_{m+n}$ and $\alpha\alpha'\in\cB$;
                \\
          0
                & in any other case.
          \end{cases*}
  \]
Using this simple formula, our description of the cocycles and coboundaries
in the complex~$\kk(\Gamma\parallel\cB)$, and a non-negligible dose of
determination, we can compute the cup product of any two of the elements in
the basis of~$\HH^*(A)$ that we described in Theorem~\ref{thm:cohomology:basis} in
Chapter~\ref{sect:corollaries}. We record the results in 
Table~\ref{tbl:cup} and use the rest of this section to calculate its entries. 

We start with an observation that singles out an exceptional situation.

\begin{lemma}\label{lemma:exception}
Let $\gamma$ be a $\Gamma$-complete path in~$\Gamma$ and $\alpha$ a
$\cB$-maximal cycle in~$\cB$, and suppose that $s(\gamma)=s(\alpha)$. If
one of~$\gamma$ or~$\alpha$ has length~$1$, then the quiver has exactly one
vertex and one arrow. 
\end{lemma}

\begin{proof}
The maximality of~$\alpha$ implies that $\alpha a$, $a\alpha \in I$ for
every arrow $a$. The gentleness of $(Q,I)$ implies that the first and
last arrow of~$\gamma$ are the first and last arrow of~$\alpha$. Since
one of the two paths has length one, we have $\alpha=\gamma\in
Q_1$ and $\alpha$ is the only arrow in~$Q$.
\end{proof}

\begin{table}
  \centering
  \setlength\extrarowheight{2pt}
  \setlength{\tabcolsep}{7pt}
  \setlength{\aboverulesep}{0pt}
  \setlength{\belowrulesep}{0pt}
  \colorlet{empty}{gray!40}
  \colorlet{max}{red!10}
  \colorlet{glue}{blue!15}
  \begin{tabular}{c@{\hskip1.3em}*{7}{c}}
  $\smile$
    & $(s(\alpha),\alpha)$
    & $\cycle{\alpha}$
    & $(c,c\delta)$
    & $(c,c)$
    & $(\gamma,\alpha)$
    & $\cycle{C}$
    & $(bC,b)$
    \\ \toprule
  $(s(\alpha),\alpha)$
    & \cellcolor{max} $0_{\text{max}}$
    & \cellcolor{max} $0_{\text{max}}$
    & \cellcolor{max} $0_{\text{max}}$
    & \cellcolor{max} $0_{\text{max}}$
    & 0 \eqref{it:cup:cos}
    & 0 \eqref{it:cup:sec}
    & \cellcolor{max} $0_{\text{max}}$
    \\ \midrule
 $\cycle{\alpha}$
    &  \cellcolor{empty}
    & \eqref{it:cup:cycle0:a}
    & \eqref{it:cup:cycle0:a}
    & \eqref{it:cup:cycle0:b}
    & \cellcolor{max} $0_{\text{max}}$
    & $0$ \eqref{it:cup:b}
    & $0$ \eqref{it:cup:b}
    \\ \midrule
 $(c,c\delta)$
   &  \cellcolor{empty}
   &  \cellcolor{empty}
   & \cellcolor{glue} $0_{\text{glue}}$
   & \cellcolor{glue} $0_{\text{glue}}$
   & \cellcolor{max} $0_{\text{max}}$
   & $0$  \eqref{it:cup:C:c}
   & \cellcolor{glue} $0_{\text{glue}}$
   \\ \midrule
  $(c,c)$
    & \cellcolor{empty}
    &  \cellcolor{empty}
    &  \cellcolor{empty}
    & \cellcolor{glue} $0_{\text{glue}}$
    & \cellcolor{max} $0_{\text{max}}$
    & \eqref{it:cup:C:d}
    & \cellcolor{glue} $0_{\text{glue}}$
    \\ \midrule
  $(\gamma,\alpha)$
    & \cellcolor{empty}
    &  \cellcolor{empty}
    &  \cellcolor{empty}
    & \cellcolor{empty}
    & \cellcolor{max} $0_{\text{max}}$
    & \cellcolor{max} $0_{\text{max}}$
    & \cellcolor{max} $0_{\text{max}}$
    \\ \midrule
  $\cycle{C}$
    & \cellcolor{empty}
    &  \cellcolor{empty}
    &  \cellcolor{empty}
    & \cellcolor{empty}
    & \cellcolor{empty}
    & \eqref{it:cup:C:e}
    & \eqref{it:cup:C:e}
    \\ \midrule
  $(bC,b)$
    & \cellcolor{empty}
    &  \cellcolor{empty}
    &  \cellcolor{empty}
    & \cellcolor{empty}
    & \cellcolor{empty}
    & \cellcolor{empty}
    & \cellcolor{glue} $0_{\text{glue}}$
    \\ \bottomrule
  \end{tabular}
\bigskip
\caption{The cup product of elements of our  basis of~$\HH^*(A)$, for $Q$
not a quiver with one vertex and one arrow.}
\label{tbl:cup}
\end{table}

Let us suppose that the quiver is not a loop and go through the entries of
Table~\ref{tbl:cup}. We will deal with that exceptional case later.
\begin{enumerate}[label=(\roman*), ref=\roman*]

\item The element~$\one$ is clearly the unit for the cup product.

\item The entries in the table marked $0_{\text{max}}$ correspond to
products of basis elements which vanish because one of the two factors is
the class of an element of~$\Gamma\parallel\cB$ that has either its first
component $\Gamma$-maximal or its second component $\cB$-maximal, and the
other factor is a linear combination of elements of~$\Gamma\parallel\cB$
all of which have in that same position paths of positive length. Then the
product is  zero.

\item The entries marked $0_{\text{glue}}$ correspond to products in which
the first factor is an element of~$\Gamma\parallel\cB$ whose two components
have positive length and start in the same arrow, and in which the second
factor is an element of~$\Gamma\parallel\cB$ whose components have positive
length and end in the same arrow.

\item  Let $\alpha$ be a cocomplete cycle in $\Crep(\cB)$ of period $r$.
\begin{enumerate}[label=(\alph*), ref=\theenumi.\alph*]

\item\label{it:cup:cycle0:b} Let $c$ be an arrow in~$Q_1\setminus T$. If
$c$ is one of the arrows in the cycle~$\alpha$, then there is exactly one
$i\in\{0,\dots,r-1\}$ such that $c$ is the first arrow in~$\rot^i(\alpha)$,
and then
  \[
  \cycle{\alpha} \smile (c,c) = (c,c\rot^i(\alpha))=(c,c\delta),
  \]
with $\delta=\rot^i(\alpha)$. If the cycle~$\alpha$ does not go through the
arrow~$c$, we have that 
  \[
  \cycle{\alpha} \smile (c,c) = 0.
  \]
        
\item\label{it:cup:cycle0:a} If $\delta$ is another element of
$\Crep(\cB)$, then either $\alpha$ and $\delta$ are powers of the same
primitive cycle, then $\alpha\delta$ is in $\Crep(\cB)$  and 
  \begin{align*}
  \cycle{\alpha}\smile\cycle{\delta} &= \cycle{\alpha\delta}, 
  &
  \cycle{\alpha}\smile (c,c\delta) &= (c,c\alpha\delta),
\intertext{with~$c$ the first arrow in~$\delta$, or they are not and}
  \cycle{\alpha}\smile\cycle{\delta} &= 0, 
  &
  \cycle{\alpha}\smile (c,c\delta) &= 0
  \end{align*}	

\end{enumerate}

\item\label{it:cup:cos} Let $\alpha$ be a $\cB$-maximal path and let
$(\gamma,\beta)\in\Gamma\parallel\cB$ be such that $\gamma$ is
$\Gamma$-maximal and $\gamma$ and~$\beta$ neither begin nor end with the
same arrow. If $s(\alpha)\neq t(\gamma)$, then 
\begin{equation}\label{eq:ss}
(s(\alpha),\alpha)\smile(\gamma,\beta) = 0.
\end{equation}
Suppose now that $s(\alpha)=t(\gamma)$, and let $a$ be the first arrow
in~$\alpha$ and $b$ the last one in~$\gamma$. If the length of~$\beta$
is~$0$, then because $\alpha$ is~$\cB$-maximal the path $ab$ has to be
in~$R$, and because $\gamma$ is $\Gamma$-maximal that path cannot be in~$R$:
this is a contradiction, and we see that $\beta$ necessarily has positive length.
Thus \eqref{eq:ss} also holds because $\alpha$ is~$\cB$-maximal.

\item Let $C\in\Crep(\Gamma)$ and let $m$ be the length and $r$ the
period of~$C$.
\begin{enumerate}[label=(\alph*), ref=\theenumi.\alph*]

\item\label{it:cup:sec} Let $\alpha$ be a $\cB$-maximal cycle. If
$i\in\{0,\dots,r-1\}$, then either we have that $s(\rot^i(C))\neq
s(\alpha)$ and 
  \[
  (s(\alpha),\alpha)\smile(\rot^i(C),s(\rot^i(C))) = 0,
  \]
or $s(\rot^i(C))=s(\alpha)$. Suppose that we are in this last case. If one
of $C$ or~$\alpha$ had length~$1$, then Lemma~\ref{lemma:exception} would
tell us that the quiver~$Q$ is the one we have excluded: both paths
therefore have length at least~$2$.
Let $a$ and $b$ be the first and last arrows
in~$\alpha$, and let $f$ and~$g$ be the first and last arrows
in~$\rot^i(C)$. Since $\alpha$ is $\cB$-maximal, we must have $ag\in R$,
and since $\rot^i(C)$ is a $\Gamma$-complete cycle, that $a=f$. Similarly,
the maximality of~$\alpha$ implies that $fb\in R$ and then $b=g$. 
As both $\alpha$ and~$C$ have length at least~$2$, there are
paths $\delta$ and~$\zeta$ such that $\alpha=b\zeta a$ and
$\rot^i(C)=b\delta a$, and
  \[
  (s(\alpha),\alpha)\smile(\rot^i(C),s(\rot^i(C))) 
        = (b\delta a,b\zeta a) \equiv 0,
  \]
as the last pair is the coboundary of~$(\delta a,\zeta a)$.
We thus conclude that, apart from the exceptional case, we have
  \[
  (s(\alpha),\alpha)\smile\cycle{C} \equiv 0.
  \]

\item\label{it:cup:b} Let $\alpha$ be a cocomplete cycle in $\Crep(\cB)$
and let $r'$ be its period.
If $s(\rot^i(C))\neq s(\rot^j(\alpha))$ for all $i$ and~$j$,
then we clearly have that 
  \begin{align*}
  \cycle{\alpha}\smile \cycle{C} &=0,
  & 
  \cycle{\alpha}\smile(bC,b)&= 0,
  \end{align*}
with~$b$ the first arrow in~$C$. Let us suppose that, on the contrary, there are
integers $i\in \{0,...,r-1\}$ and $j\in \{0,...,r'-1\}$ such that
  \(
  s(\rot^i(C))= s(\rot^j(\alpha)).
  \)
As the algebra is gentle, the only possibility is that~$\rot^i(C)$ and $\rot^j(\alpha)$ start with the same arrow~$a$. Then, 
taking into account that either $m$ is even or that characteristic of $\kk$ is 2, we have 
 \hspace{6em} \begin{flalign*}
  \cycle{\alpha}\smile \cycle{C} 
  &  =(-1)^{mi} \cdot(\rot^i(C), \rot^j(\alpha)) \\
     & \;\;\;\;\; +(-1)^{m(i+1)}\cdot(\rot^{i+1}(C), \rot^{j+1}(\alpha))\\
   & = (\rot^i(C), \rot^j(\alpha))  + (\rot^{i+1}(C), \rot^{j+1}(\alpha))\\ 
\shortintertext{and}\\
 \;\;\; \;\; \cycle{\alpha}\smile (bC,b)
&  \equiv (-1)^{m(i+1)}\cdot(a\rot^{i+1}(C), a\rot^{j+1}(\alpha))\\
 & = (a\rot^{i+1}(C), a\rot^{j+1}(\alpha))
  \end{flalign*}
are  coboundaries, so that the classes~$\cycle{\alpha}\smile \cycle{C}$
and~$\cycle{\alpha}\smile (bC,b)$ are zero in cohomology.

\item\label{it:cup:C:c}  Let $\delta$ be a cocomplete cycle in $\Crep(\cB)$
and $c$ its first arrow. Using  \eqref{it:cup:cycle0:b} and \eqref{it:cup:b}
we see that
  \[
  (c,c\delta)\smile \cycle{C}= (c,c)\smile \cycle{\delta} \smile \cycle{C}= 0.
  \]

\item\label{it:cup:C:d} Let now $c$ be an arrow in~$Q_1\setminus T$. If $c$
is one of the arrows in the cycle~$C$, then there is exactly one
$i\in\{0,\dots,r-1\}$ such that $c$ is the first arrow in~$\rot^i(C)$, and
then
  \[
  (c,c) \smile \cycle{C} = (-1)^{mi} \cdot(c\rot^i(C),c) \equiv (bC,b) ,
  \]
with $b$ the first arrow of~$C$, because since $C\in\Crep(\Gamma)$ either
$m$ is even or the characteristic of~$\kk$ is~$2$. On the other hand, if
the cycle~$C$ does not go through the arrow~$c$, we clearly have that
  \[
  (c,c) \smile \cycle{C} = 0.
  \]

\item\label{it:cup:C:e} If $D$ is another element of~$\Crep(\Gamma)$
then either $C$ and $D$ are powers of the same primitive cycle, so
that $CD\in\Crep(\Gamma)$ and 
  \begin{align*}
  \cycle{C}\smile\cycle{D} &= \cycle{CD}, 
  &
  (bC,b)\smile\cycle{D} &= (b,bCD),
\intertext{with~$b$ the first arrow in~$C$, or they are not and}
  \cycle{C}\smile\cycle{D} &= 0, 
  &
  (bC,b)\smile\cycle{D} &= 0.
  \end{align*}

\end{enumerate}
\end{enumerate}

Having gone through all the entries in Table~\ref{tbl:cup}, we can make the
following useful observation: a product of elements of our basis
of~$\HH^*(A)$ is either zero or an element of that
basis.

\begin{remark}\label{rem:strict-commutativity}
In general, if $\Lambda$ is an arbitrary algebra and the characteristic of
the ground field~$\kk$ is not~$2$, then with respect to the cup product
$\HH^*(\Lambda)$ is a \emph{strictly} graded-commutative algebra ---~it is
graded-commutative and, moreover, the square of an homogeneous element of
odd order is zero~--- but if the characteristic is~$2$ then
$\HH^*(\Lambda)$ may be only graded-commutative but not strictly so.  Our
calculations above show that we have an example of such a not strictly graded-commutative algebra in characteristic 2 whenever there is a $\Gamma$-complete cycle~$C$
in the presentation~$(Q,I)$ of odd length, then $\HH^*(A)$ is not strictly
graded-commutative: the square of~$\cycle{C}$, a non-zero class of
odd degree, is $\cycle{C^2}$, and this is not zero.
\end{remark}

We now consider the cases which
we excluded above.

\begin{remark}\label{rem:exception}
Let us suppose that the quiver $Q$ has exactly one vertex and one
arrow~$a$, so in particular the spanning tree~$T$ is empty.
If the presentation~$(Q,I)$ is \fd gentle, we have~$a^2\in R$, the
unique $\cB$-maximal path is~$a$, there are no $\Gamma$-maximal paths.
Depending on the characteristic of the ground field, we have two cases:
\begin{itemize}

\item If the characteristic of~$\kk$ is not~$2$, then
$\Crep(\Gamma)=\{a^{2l}:l\geq1\}$,
the vector space~$\HH^0(A)$ is freely spanned by~$\one$
and~$(s(a),a)$, and for each $k\geq1$ the vector space $\HH^k(A)$ is
freely spanned by~$(a^k,a)$ if $k$ is odd, and by $\cycle{a^k}$ if
$k$ is even. For all integers~$m$,~$n\geq1$ we have that
  \begin{align*}
  & (s(a),a)\smile(s(a),a) = 0,
  & 
  & (s(a),a)\smile \cycle{a^{2m}} = (a^{2m},a)
  \\
  & (s(a),a)\smile(a^{2n-1},a) =0 
  &
  & \cycle{a^{2m}} \smile \cycle{a^{2n}} 
              = \cycle{a^{2(m+n)}}
  \\
  & \cycle{a^{2m}}\smile(a^{2n-1},a) 
                = (a^{2(m+n)-1},a)
  &
  & (a^{2m-1},a) \smile (a^{2n-1},a) = 0.
  \end{align*}

\item If the characteristic of~$\kk$ is~$2$, then
$\Crep(\Gamma)=\{a^{l}:l\geq1\}$, the vector space~$\HH^0(A)$ is freely
spanned by~$\one$ and~$(s(a),a)$, and for each $m\geq1$ the vector
space~$\HH^m(A)$ by~$\cycle{a^m}$ and by $(a^m,a)$. For all integers
$m$,~$n\geq1$ we have that
  \begin{align*}
  &  (s(a),a)\smile(s(a),a) = 0, 
  && (s(a),a)\smile\cycle{a^m} = (a^m,a), \\
  &  (s(a),a)\smile(a^m,a) = 0, 
  && \cycle{a^m}\smile\cycle{a^n} = \cycle{a^{m+n}}, \\
  &  \cycle{a^m}\smile(a^n,a) = (a^{m+n},a), 
  && (a^m,a) \smile (a^n,a) = 0.
  \end{align*}
\end{itemize}
On the other hand, if the presentation $(Q,I)$ is  gentle and not \fd
gentle, then there are no relations,  there are no $\cB$-maximal paths and
the unique $\Gamma$-maximal path is~$a$. Then
$\Crep(\cB)=\{a^{l}:l\geq1\}$, 
\begin{itemize}

\item the vector space~$\HH^0(A)$ is freely
spanned by~$\one$ and the elements~$\cycle{a^k}$, one for each $k>0$;

\item the vector space~$\HH^1(A)$ is freely spanned by the pairs~$(a,a^k)$, one for each
$k\geq0$; 

\item for all $m>1$ we have $\HH^m(A)=0$. 

\end{itemize}
Now it is clear that
\begin{align*}
	& \cycle{\alpha^m}\smile  \cycle{\alpha^n}
                = \cycle{\alpha^{m+n}},
	&&  \cycle{\alpha^m}\smile (a,s(a))=(a,a^{m}),\\ 
	&  \cycle{\alpha^m}\smile (a,a^n)=(a,a^{m+n}),
	&& (a,a^n)\smile (a,a^m)= 0
\end{align*}
for all integers $m$,~$n\geq1$. 
\end{remark}

\section{A presentation for the cohomology algebra}
\label{sect:cup:presentation}

Our next task is to exhibit a presentation of the
algebra~$\HH^*(A)$. We let $\Crepprim(\Gamma)$ be the set of
those elements of~$\Crep(\Gamma)$ that are not proper powers of another
element of~$\Crep(\Gamma)$.  The elements of~$\Crepprim(\Gamma)$ are the
primitive elements of~$\Crep(\Gamma)$ when the characteristic of~$\kk$
is~$2$. But when the characteristic of $\kk$ is not~$2$, the elements of $\Crepprim(\Gamma)$ are the primitive
elements of~$\Crep(\Gamma)$ of even length together with the squares of the
primitive $\Gamma$-complete cycles of odd length  in~$(Q,I)$.
Similarly, we let $\Crepprim(\cB)$ denote the set of elements
of~$\Crep(\cB)$ that are not proper powers of another element
of~$\Crep(\cB)$ --- which is empty if the presentation~$(Q,I)$ is \fd gentle.

\bigskip

We start with a simple consequence of the  gentleness of our
presentation.

\begin{lemma}
The sets $\Crepprim(\Gamma)$ and $\Crepprim(\cB)$ are finite,
no two elements in the same set have an arrow in common, and 
  \begin{gather*}
  \Crep(\Gamma)=\{C^k:C\in\Crepprim(\Gamma),k\geq1\} \\
\shortintertext{and}
  \Crep(\cB)=\{\alpha^k:\alpha\in\Crepprim(\cB),k\geq1\}.
  \end{gather*}
\end{lemma}

\begin{proof}
If the set~$\Crepprim(\Gamma)$ were infinite, then two of its elements
would start with the same arrow, and this is impossible because the
presentation~$(Q,I)$ is gentle. Similarly, if $C$ and~$D$ are two
distinct elements of~$\Crepprim(\Gamma)$ which share an arrow, then the
gentleness of the presentation implies that $C$ and~$D$ are
powers of conjugate primitive $\Gamma$-complete cycles, and this is a contradiction,
for $C$ and~$D$ are primitive and not conjugate. This proves the first two
claims of the lemma for $\Crepprim(\Gamma)$, and the third one is
immediate. By symmetry we obtain the result for $\Crepprim(\cB)$.
\end{proof}

Next, we use our calculation of the products of pairs of elements of the
basis of~$\HH^*(A)$ to exhibit a generating set for that cohomology as an
algebra:

\begin{proposition}\label{prop:G}
The set~$\mathscr{G}$ of cohomology
classes of the following cocycles of~$\kk(\Gamma\parallel\cB)$
is a generating set for the algebra~$\HH^*(A)$:
\begin{itemize}

\item The pairs $(s(\alpha),\alpha)$ with $\alpha$ a $\cB$-maximal cycle
in~$(Q,I)$.

\item The sums $\cycle{\alpha}$ with $\alpha\in\Crepprim(\cB)$.

\item The pairs $(c,c)$ with $c$ an arrow in the complement of the spanning
tree~$T$.

\item The pairs $(\gamma,\alpha)$ with $\gamma$ a $\Gamma$-maximal element
of~$\Gamma$ and $\gamma$ and~$\alpha$ neither beginning or ending with the
same arrow.

\item The sums $\cycle{C}$ with $C\in\Crepprim(\Gamma)$.
\end{itemize}
\end{proposition}

\begin{proof}
To show that the set~$\mathscr{G}$ generates the algebra~$\HH^*(A)$ it is
enough to show that each element in the basis described at the beginning of
Section~\ref{sect:corollaries} is generated by it. This is obvious
for the basis elements in~\ref{gen:one}, \ref{gen:B-max},
\ref{gen:fundamental}, and \ref{gen:clean}. 
If $C$ is an element of~$\Crep(\Gamma)$, then there is a
$D\in\Crepprim(\Gamma)$ and an integer~$k\geq1$ such that $C=D^k$, so that
$\cycle{C}=\cycle{D}^{\smile k}$. On the other hand, since~$C$ is an
oriented cycle in the quiver~$Q$, there is an arrow~$c\in Q_1\setminus T$
that appears in~$C$, and then $(c,c)\smile\cycle{D^k}=(bC,b)$, with $b$ the
first arrow in~$D$. Similarly, we can obtain the elements
in~\ref{gen:B-sum} and~\ref{gen:B-plus} from cocomplete cyles~$\alpha\in
\Crepprim(\cB)$ and arrows~$c\in Q_1\setminus T$.
\end{proof}

\begin{remark}
The set~$\mathscr{G}$ of Proposition~\ref{prop:G} generates~$\HH^*(A)$
minimally except when  
  \begin{equation}\label{eq:exception}
  \claim[0.75]{the quiver $Q$ has one vertex and one loop $a$, and either
  $a^2\notin I$, or $a^2\in I$ and the characteristic of the ground field
  is~$2$.}
  \end{equation}
This can be checked by inspection: when we are not in the situation
of~\eqref{eq:exception},
none of the elements in~$\mathscr{G}$ is a linear combination of products
of others. In the exceptional case~\eqref{eq:exception}, on the other hand, 
we have that $(s(a),a)\smile (a,s(a))=(a,a)$, so that the generator~$(a,a)$
listed in the proposition is not really needed.
\end{remark}

Finally, the last step is to write down a sufficient set of relations that 
present the cohomology algebra. In doing that we will make use of the
following observation:

\begin{lemma}\label{lemma:kernel}\pushQED{\qed}
Suppose that $f:\Lambda\to\Omega$ is a morphism of algebras, and that $B$
and~$B'$ are bases of~$\Lambda$ and~$\Omega$, respectively. If $f$ maps
each element of~$B$ to scalar multiple of an element of~$B'$, then the
kernel $I\coloneqq\ker f$ is the linear span of its subset 
  \begin{equation}\label{eq:kernel}
  \{x-ty:x,y\in B,\,x\neq y,\,t\in\kk\} \cap I. 
  \end{equation}
\end{lemma}

\begin{proof}
Suppose that the map~$f$ satisfies the condition in the lemma and, to reach
a contradiction, that there is an $n\in\NN$ such that we can find pairwise
different elements $b_1$,~\dots,~$b_n$ of~$B$ and non-zero scalars
$\lambda_1$,~\dots,~$\lambda_n\in\kk$ so that the linear
combination $z\coloneqq\lambda_1b_1+\cdots+\lambda_nb_n$ is in the kernel
of~$f$ and not in the span of the set~\eqref{eq:kernel}. Without loss of
generality, we can assume moreover that $n$ is minimal with respect to that
property. In view of the form of the set~\eqref{eq:kernel} we then have
that $n\geq3$. On the other hand, the hypothesis on~$f$ and the minimality
of~$n$ imply that there is an element $b'\in B'$ and non-zero scalars
$\mu_1$,~\dots,~$\mu_n$ such that $f(b_i)=\mu_ib'$ for all
$i\in\{1,\dots,n\}$. As
  \[
  0=f(z)=\left(\sum_{i=1}^n\lambda_i\mu_i\right)b'
  \]
and $b'\neq0$, we have that
  \[
  z = \sum_{i=1}^n\lambda_ib_i 
        - \mu_1^{-1}\left(\sum_{i=1}^n\lambda_i\mu_i\right)b_1
    = \sum_{i=2}^n\lambda_i(b_i-\mu_1^{-1}\mu_ib_1),
  \]
and this is a contradiction, since for each $i\in\{2,\dots,n\}$ the element
$b_i-\mu_1^{-1}\mu_ib_1$ belongs to the set~\eqref{eq:kernel}. This proves
the lemma.
\end{proof}

\begin{theorem}\label{thm:algebra}
Let $(Q,I)$ be a gentle presentation and suppose that either the
quiver~$Q$ is not one with one vertex and one arrow whose square is in~$I$
or that the characteristic of~$\kk$ is not~$2$, and let $A\coloneqq\kk Q/I$
be the algebra it presents. The cohomology algebra $\HH^*(A)$ is the
quotient of the free
graded-commutative algebra generated by the set~$\G$ of
Proposition~\ref{prop:G} by the ideal generated by the following elements:
\begin{itemize}

\item $u\smile v$, one for  each choice of~$u$ and~$v$ in~$\G$ except those
in which 
\begin{itemize}

\item $u=v=\cycle{C}$ for some~$C\in\Crepprim(\Gamma)$,

\item $u=v=\cycle{\alpha}$ for some~$\alpha\in\Crepprim(\cB)$, 

\item $u$ and~$v$ are, in some order, $\cycle{C}$ and $(c,c)$ with
$C\in\Crepprim(\Gamma)$ and $c$ an arrow in~$Q_1\setminus T$ through which
$C$~passes.
 
\item $u$ and~$v$ are, in some order, $\cycle{\alpha}$ and $(c,c)$ with
$\alpha\in\Crepprim(\cB)$ and $c$ an arrow in~$Q_1\setminus T$ through
which $\alpha$~passes.

\end{itemize}

\item $(c,c)\smile\cycle{C}-(d,d)\smile\cycle{C}$, one for each choice of two
\emph{distinct} arrows~$c$ and~$d$ in~$Q_1\setminus T$ and of a complete
cycle~$C$ in~$\Crepprim(\Gamma)$ that passes both through~$c$ and through~$d$.

\item $(c,c)\smile\cycle{\alpha}-(d,d)\smile\cycle{\alpha}$, one for each
choice of two \emph{distinct} arrows~$c$ and~$d$ in~$Q_1\setminus T$ and of
a cocomplete cycle~$\alpha$ in~$\Crepprim(\cB)$ that passes both through~$c$
and through~$d$.

\end{itemize}
\end{theorem}

Of course, the elements of~$\G$ have each a cohomological degree, and this is important
here in order to determine the commutation relations implicit in this
presentation. 

\begin{proof}%
\def\H{\mathscr{H}}%
\def\M{\mathscr{M}}%
\def\I{\mathscr{I}}%
If the quiver~$Q$ has one vertex and one arrow $a$, then the hypothesis of
the theorem implies that either that  the characteristic of~$\kk$ is not~$2$
or~$I=0$, and we can check in both cases the claim of the theorem by hand
using Remark~\ref{rem:exception}. We will therefore assume in what remains
of the proof that the quiver is not of that form. 

Let $\H$ be the free graded-commutative algebra generated by the set~$\G$.
We fix an arbitrary total order~$\preceq$ on the set~$\G$ such that the
elements of the form~$(c,c)$, with~$c\in Q_1\setminus T$, are smaller than
all others, and write $\M$ for the set of elements of~$\H$ obtained as
products of zero or more elements of the set~$\G$ in which the factors are
non-decreasing with respect to the order~$\preceq$ and in which, if the
characteristic of~$\kk$ is not~$2$, no element of odd degree appears more
than once. This set~$\M$ is a basis for the algebra~$\H$; we will refer to
its elements as \newterm{monomials}.

As the algebra~$\HH^*(A)$ is graded-commutative, there is a unique morphism
of graded algebras $\pi:\H\to\HH^*(A)$ mapping each element
of~$\G$ to itself, and it is surjective because the set~$\G$
generates~$\HH^*(A)$ as an algebra. Let $\I$ be its kernel.

The set~$\G$ is contained in our basis of~$\HH^*(A)$, and our calculations
show that this basis has the property that the product of any two of its
elements is either zero or  an element of that
basis. This tells us that the image under~$\pi$ of an element of the
basis~$\M$ of~$\H$ is either zero or  an
element of our basis of~$\HH^*(A)$, and therefore the map~$\pi$ falls under
the hypothesis of Lemma~\ref{lemma:kernel} and the
ideal~$\I$ is the span of its elements that are linear combinations of at
most two elements of~$\M$.
\begin{itemize}

\item Let $w$ be an element of~$\M$, so that there are an integer~$n\geq0$
and elements $u_1$,~\dots,~$u_n$ of~$\G$ such that $u_1\preceq\cdots\preceq
u_n$ and $w=u_1\cdots u_n$. Suppose that $w$
is not divisible by any of the quadratic monomials described in the first
bullet point of the theorem. If $n<2$ then clearly $\pi(w)\neq0$. Suppose that
instead $n\geq2$. It is easy to see that there is then a
cycle~$D\in\Crepprim(\Gamma)\cup\Crepprim(\cB)$ such that all the factors
in~$w$ except at most one are equal to~$\cycle{D}$, and that if not all of
them are equal to that then there is an arrow~$c\in Q_1\setminus T$ through
which the cycle~$D$ passes such that the remaining factor is equal
to~$(c,c)$. The image of~$w$ under~$\pi$ is thus either $\cycle{D}^{\smile
n}$ or $(c,c)\smile\cycle{D}^{\smile(n-1)}$, which we know to be non-zero.
As all the quadratic monomials listed in the theorem are certainly in~$\I$,
we can conclude with all this that the elements of~$\M$ that belong to~$\I$
are precisely those divisible by those quadratic monomials.

\item Next, let $z$ be an element of the ideal~$\I$ that is a linear
combination of two different elements~$x$ and~$y$ of~$\M$ and such that
neither of those two monomials is itself in~$\I$. Those two monomials have
the same image under the map~$\pi$ but are different: according to our
discussion in the previous point this is only possible if there is a cycle
$D\in\Crepprim(\Gamma)\cup\Crepprim(\cB)$ and two different arrows~$c$
and~$d$ that appear both in~$D$ such that the monomials~$x$ and~$y$ are
equal  to~$(c,c)\cdot\cycle{D}^n$ and
$(d,d)\cdot\cycle{D}^n$ for some positive integer~$n$. The difference
$z'\coloneqq(c,c)\cdot\cycle{D}^n-(d,d)\cdot\cycle{D}^n$ is in the
ideal~$\I$: as neither~$x$ nor~$y$ are in that ideal, we see that $z$ is a
scalar multiple of~$z'$.  
We see that $z'$ is divisible
by~$(c,c)\cdot\cycle{D}-(d,d)\cdot\cycle{D}$, and this is one of the
elements listed in the second and third bullet points of the theorem.

\end{itemize}
Putting everything together, we can easily conclude that the elements given in
the theorem indeed generate the kernel~$\I$ of the map~$\pi$, and this
proves the claim of the theorem.
\end{proof}

The presentation for the algebra~$\HH^*(A)$ given by this theorem is a
monomial quadratic presentation exactly when we can pick the spanning
tree~$T$ so that the following condition is satisfied:
 \begin{equation}\tag{$\star$}\label{eq:star}
 \claim{every element of~$\Crepprim(\Gamma)$ and $\Crepprim(\cB)$ passes
 through exactly one element of~$Q_1\setminus T$.}
 \end{equation} 
In this case, the second and third bullet points of the statement of
the theorem do not give any elements. However, this cannot be done in general. For
example, assuming the characteristic of the ground field is not~$2$,
in the gentle presentation of Figure~\vref{fig:example:no-star},
the complement of every spanning tree has four arrows, there are two
circuits in~$\Crepprim(\Gamma)$, conjugated to the
cycles~$(cba)^2$ and~$(fed)^2$, and these two circuits partition the set of
arrows of the quiver, so that whatever the choice of the spanning tree~$T$
there is an element on~$\Cprim(\Gamma)$ that involves two arrows from the
complement of~$T$. Worse, in this example there is one element
in~$\Crepprim(\cB)$, the circuit of the cycle $fbdcea$, and obviously it
passes through all four arrows of the complement of every spanning tree
in~$Q$.

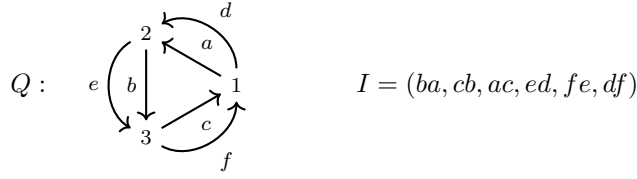
\begin{figure}
  \[
  Q:\quad
  \begin{tikzpicture}[auto, thick, font=\footnotesize,
                      scale=0.8, baseline=(1.base)]
  \node (1) at (0:1) {$1$};
  \node (2) at (120:1) {$2$};
  \node (3) at (240:1) {$3$};
  
  \draw [->](1)  to node[swap] {$a$} (2);
  \draw [->](2)  to node[swap] {$b$} (3);
  \draw [->](3)  to node[swap] {$c$} (1);
  \draw [->](1)  to[out=0+90, in=120-90]   node[swap] {$d$} (2);
  \draw [->](2)  to[out=120+90, in=240-90] node[swap] {$e$} (3);
  \draw [->](3)  to[out=240+90, in=0-90]   node[swap] {$f$} (1);
  
  \end{tikzpicture}
  \qquad\qquad
  I = (ba,cb,ac,ed,fe,df)
  \]
\caption{A presentation in which we cannot find a spanning tree satisfying
condition~\eqref{eq:star} of page~\pageref{eq:star}.}
\label{fig:example:no-star}
\end{figure}

\begin{remark}
In Theorem~\ref{thm:algebra} there is exactly one case excluded by the
hypothesis on the presentation: that in which the quiver~$Q$ has one vertex
and one arrow~$a$, the square of~$a$ is in~$I$, and the characteristic
of~$\kk$ is equal to~$2$. Using the information given by
Remark~\ref{rem:exception} we can see that in that situation the cohomology
algebra~$\HH^*(A)$ is freely generated as a graded-commutative algebra by
the classes of the elements~$(s(a),a)$ and $(a,s(a))$ of degrees~$0$
and~$1$, subject only to the relation $(s(a),a)\smile(s(a),a)=0$. 
\end{remark}

\section{The cap product}
\label{sect:cup:cap}

The cap product turns the Hochschild homology $\HH_*(A)$ of an algebra into
a right module over the Hochschild cohomology algebra~$\HH^*(A)$. It is
usually constructed in terms of the bar resolution of the algebra, as in
S.\,Witherspoon's book \cite[\S1.5]{Witherspoon},  but in fact it can be
computed using any bimodule projective resolution --- this is indicated 
in the book of H. Cartan and S. Eilenberg \cite[Chapter~XI, Exercise 2]{C-E} and spelt out, with a different
choice of signs, in~\cite{Armenta}. We will recall the
details for the case of our algebras.

We fix a  gentle presentation~$(Q,I)$, write $A\coloneqq\kk Q/I$,
and let $\epsilon:\cR\to A$ be the Bardzell projective resolution of~$A$ as
an~$A$-bimodule that we described in Chapter~\ref{chapter:algebras}. In
Section~\ref{sect:cup:calculation} we exhibited a morphism of
complexes~$\Delta:\cR\to\cR\otimes_A\cR$ over the identity of~$A$. If
$M$ and~$N$ are two $A$-bimodules, then the Hochschild homology $\H_*(A,M)$ with
coefficients in~$M$ and the Hochschild cohomology~$\H^*(A,N)$ with values
in~$N$ are the homology and the cohomology of the complexes
$M\otimes_{A^e}\cR$ and $\Hom_{A^e}(\cR,M)$, respectively, and the cap
product
  \begin{equation} \label{eq:frown}
  \mathord\frown:\H_*(A,M)\otimes\H^*(A,N)\to\H_*(A,M\otimes_AN)
  \end{equation}
is induced by the composition of maps of complexes depicted in
Figure~\vref{fig:cap} --- which we will simply write also~$\frown$. The map
  \[
  \ev:\cR\otimes_A\cR\otimes\Hom_{A^e}(\cR,N)\to N\otimes_A\cR
  \]
appearing there is such that
  \[
  \ev(x\otimes y\otimes f) = (-1)^{qr+pr}f(x)\otimes y
  \]
whenever $x\in\cR_p$, $y\in\cR_q$ and $f\in\Hom_{A^e}(\cR_r,N)$. Because everything is homogeneous, 
we have $f(x)=0$ if $p\neq r$.

\begin{figure}[t]
  \begin{tikzcd}
  M\otimes_{A^e}\cR\otimes\Hom_{A^e}(\cR,N)
        \arrow[d, "\id\otimes\Delta\otimes\id"]
        \\
  M\otimes_{A^e}(\cR\otimes_A\cR)\otimes\Hom_{A^e}(\cR,N)
        \arrow[d, "\id\otimes\ev"]
        \\
  M\otimes_{A^e}(N\otimes_A\cR)
        \arrow[d, equal]
        \\
  (M\otimes_AN)\otimes_{A^e}\cR
  \end{tikzcd}
\caption{The morphism of complexes that induces the cap product.}
\label{fig:cap}
\end{figure}

Specializing $M$ and $N$ to~$A$, and replacing the complexes
$\Hom_{A^e}(\cR,A)$ and $A\otimes_{A^e} \cR$ by the isomorphic
complexes~$\kk(\Gamma\parallel\cB)$ and~$\kk(\cB\odot\Gamma)$ of
Chapters~\ref{chapter:cohomology} and~\ref{chapter:homology}, with which we
computed~$\HH^*(A)$ and~$\HH_*(A)$, respectively, this becomes the map
  \[
  \mathord\frown:
  \kk(\cB\odot\Gamma)\otimes\kk(\Gamma\parallel\cB)
  \to\kk(\cB\odot\Gamma)
  \]
such that whenever $(a,x)\in\cB\odot\Gamma_p$ and
$(y,b)\in\Gamma_q\parallel\cB$ we have
  \begin{equation}\label{eq:frown:pairs}
  (a,x)\frown(y,b)
    = \begin{cases*}
      (-1)^{(p+q)q}\cdot(ab,z)
        & if $x$ factorizes as $yz$; \\
      0 & in any other case.
      \end{cases*}
  \end{equation}

To determine the whole cap product~$\frown$ of~\eqref{eq:frown} it is
enough to compute the map $(\place)\frown g:\HH_*(A)\to\HH_*(A)$ for each
element $g$ of the generating set~$\mathscr{G}$ of the cohomology
algebra~$\HH^*(A)$ that we described in Proposition~\ref{prop:G}, because
it makes~$\HH_*(A)$ into a right $\HH^*(A)$-module. We do this in the next
few lemmas.

\bigskip

We start by considering cap products between the `acyclic' part
of $\HH^*(A)$ and $\HH_*(A)$. Dealing with $\HH^0(A)$ is very easy: 

\begin{lemma}\label{lemma:cap:1}
Suppose that the quiver~$Q$ is not one with one vertex and one arrow.
If~$\alpha$ is a $\cB$-maximal cycle in~$(Q,I)$, then
$\HH_*(A)\frown(s(\alpha),\alpha)=0$.
\end{lemma}

We will see that the conclusion does not hold if $Q$ is just a loop.

\begin{proof}
Note that $\alpha$ has length at least~$2$: if it had length~$1$, then its
$\cB$-maximality, the gentleness of~$(Q,I)$ and the connectedness
of~$Q$ would imply that $Q$ has one vertex and one arrow, contradicting
the hypothesis on the quiver.

Let $(\beta,\delta)\in \cB\odot\Gamma$ and suppose that $(\beta,\delta)
\frown (s(\alpha),\alpha)\neq0$. From \eqref{eq:frown:pairs} we see that
$\delta=s(\alpha)\delta$ and that $\beta\alpha\in\cB$: as $\alpha$ is
$\cB$-maximal, the path~$\beta$ has length~$0$. We thus have
$(\beta,\delta) \frown (s(\alpha),\alpha)=(\alpha,\delta)$. Let $C$ be the
circuit that contains~$\alpha\delta$. 
\begin{itemize}

\item If $\delta$ has length zero, then $\alpha\delta=\alpha$
and since $\alpha$ is a $\cB$-maximal cycle belonging to~$\cB$ of length at
least~$2$, we see that the circuit~$C$ is neither complete not cocomplete.

\item If $\delta$ has positive length, then the path $\alpha\delta$ is not
in~$\cB$ because $\alpha$ is $\cB$-maximal, and is not in~$\Gamma$ because
$\alpha$ has length at least~$2$. It follows that also in this case the
circuit~$C$ is neither complete not cocomplete.

\end{itemize}
We thus see that $\kk(\cB\odot\Gamma)\frown(s(\alpha),\alpha)$ is contained
in the direct sum $\bigoplus_D\kk(\cB\odot\Gamma)_D$, with $D$ running over
the circuits that are neither complete not cocomplete, and we know this
complex is exact from Lemma~\ref{lemma:homology:neither}. The
claim of the lemma follows from this.
\end{proof}

Next we consider the cap products between the `acyclic' part of
$\HH^*(A)$ and $\HH_*(A)$, for positive cohomological degrees.

\begin{lemma}\label{lemma:cap:2}
Suppose that $Q$ is not a quiver with one vertex and one arrow. If
$(\gamma,\alpha)$ is an element of~$\Gamma\parallel\cB$ such that $\gamma$ a
$\Gamma$-maximal element of~$\Gamma$ of positive length and $\gamma$
and~$\alpha$ neither begin nor end with the same arrow, then
  \[
  \HH_*(A)\frown(\gamma,\alpha)=0.
  \]
\end{lemma}

As in the previous lemma, the hypothesis on the quiver here
is necessary for the conclusion to hold.

\begin{proof}
We claim that
  \[
  \claim{if there is a pair~$(\beta,\delta)$ in $\cB\odot\Gamma$ such that
  $(\beta,\delta)\frown(\gamma,\alpha)$ is not zero, then
  $s(\gamma)=t(\gamma)=s(\alpha)=t(\alpha)$, the unique such pair is
  $(s(\gamma),\gamma)$, and that cap product is equal
  to~$(\alpha,s(\alpha))$.}
  \]
Indeed, let $(\beta,\delta)$ be an element of~$\cB\odot\Gamma$ such that
$(\beta,\delta)\frown(\gamma,\alpha)\neq0$. There is then a path~$\eta$
such that $\delta=\gamma\eta$ and $\beta\alpha\in\cB$. As $\gamma$ is
$\Gamma$-maximal, $\eta$ has length zero and $\delta=\gamma$. Similarly, as
the path~$\alpha$ is $\cB$-maximal, $\beta$ also has length zero.
Now it is clear from our formula~\eqref{eq:frown:pairs} that
$(\beta,\delta)\frown(\gamma,\alpha)=(\alpha,s(\alpha))$ which proves the claim.

Suppose that $\alpha$ has positive length. 
As the path $\alpha$ is $\cB$-maximal, it is not a cocomplete cycle.
It follows then from Theorem~\ref{thm:basishomology} that the only way
for the homology class of the $0$-cycle $(\alpha,s(\alpha))$ to be non-zero
is that $\alpha$ be a loop with $\alpha^2\in I$. Since $\alpha$ is
$\cB$-maximal and the quiver is not just a loop, this cannot occur.

Suppose now that $\alpha$ has length zero. The path~$\gamma$ is then a
cycle and, since it has positive length and is $\Gamma$-maximal, it is not
a complete cycle. The pair $(s(\gamma),\gamma)$ does not appear with
non-zero coefficient in any of the elements of the basis of~$\HH_*(A)$
given in Theorem~\ref{thm:basishomology}, and this implies at once that
$\HH_*(A)\frown(\gamma,\alpha)=0$ also in this case.
\end{proof}

The next set of cap products we consider is that of those
between the `diagonal' part of $\HH^{1}(A)$ and $\HH_*(A)$.

\begin{lemma}\label{lemma:cap:3}
Let $c$ be an arrow in the complement of the spanning tree~$T$.
\begin{thmlist}

\item If $C$ is a cocomplete cycle in~$(Q,I)$ in which the arrow~$c$ 
appears, then 
  \[
  \hcycle{\bar C}\frown(c,c) = (\bar C,s(\bar C)).
  \]
  where $\bar C$ is the rotation of $C$ such that $s(\bar C) = t(c)$.

\item If $C$ is a complete cycle in~$(Q,I)$ whose length~$m$ and period~$r$
are such that $(-1)^{(m+1)r}=1$ in~$\kk$ and in which the arrow~$c$
appears, then
  \[
  \hcycle{\bar C}\frown(c,c) = (-1)^{m+1}\cdot(\lfact1{\bar C},\lfact2{\bar C}) = (-1)^{m+1}\cdot(c,\lfact2{\bar C}).
  \]
where $\bar C$ is the rotation of $C$ such that $t(\bar C) = t(c)$.

\item If $u$ is any of the elements of the basis of~$\HH_*(A)$ described in
Theorem~\ref{thm:basishomology} that is not of the form~$\hcycle{\bar C}$
for some complete or cocomplete cycle~$C$ in~$(Q,I)$, then 
  \[
  u\frown(c,c) = 0.
  \]

\end{thmlist}
\end{lemma}

\begin{proof}
If $(\alpha,\gamma)$ is an element of~$\cB\odot\Gamma_m$, then
  \[
  (\alpha,\gamma)\frown(c,c) 
        = \begin{cases*}
          (-1)^{m+1}\cdot(\alpha c,\eta)
                & if $\gamma$ factorizes as $c\eta$; \\
          0 & if not.
          \end{cases*}
  \]
Let $C$ be a cocomplete cycle in~$(Q,I)$, let $r$ be 
its period, so that
  \[
  \hcycle{C} 
        = \sum_{i=0}^{r-1}(\rfact1{\rot^i( C)},
                                   \rfact2{\rot^i( C)})
        \in\kk(\cB\odot\Gamma_1).
  \]
If the arrow $c$ does not appear in~$C$ then~$\hcycle{
C}\frown(c,c)=0$. Let us suppose, then, that it does appear. In that
case because of the  gentleness of~$(Q,I)$,
there is exactly one $j\in\{0,\dots,r-1\}$ such that $\rfact2{\rot^j(
C)}$ is~$c$, and therefore
  \begin{align*}
  \hcycle{ C} \frown (c,c)
       &= (\rfact1{\rot^j( C)}, \rfact2{\rot^j( C)}) \frown (c,c) 
        = (\rot^j( C),s(\rot^j( C)))
  \end{align*}
and we have shown in Lemma~\ref{lemma:homology:cocomplete} that this
$0$-cycle is homologous to~$( C,s( C))$.

Let now $C$ be a complete cycle in~$(Q,I)$, let $m$ and~$r$ be its
length and its period, respectively, and let us suppose that
$(-1)^{(m+1)r}=1$ in~$\kk$, so that to~$C$ corresponds the $m$-cycle in homology
  \[
  \hcycle{ C} =
  \sum_{i=0}^{r-1}(-1)^{(m+1)i}\cdot
        \bigl(s(\rot^i( C)), \rot^i(C)).
  \]
As in the previous case, if the arrow~$c$ does not appear in the
cycle~$ C$, we clearly have that $\hcycle{C}\frown(c,c)=0$. Let us
suppose it does. As~the
presentation~$(Q,I)$ is gentle, there is a unique
$j\in\{0,\dots,r-1\}$ such that the last arrow of~$\rot^j(C)$ is~$c$,
and then 
  \[
  \hcycle{ C}\frown(c,c)
        = (-1)^{(m+1)j+(m+1)}\cdot(\lfact1{\rot^j( C)},
                              \lfact2{\rot^j(C)}).
  \]
by Lemma~\ref{lemma:homology:complete}. Furthermore, we have 
  \[
  (-1)^{(m+1)j+(m+1)+j(m+1)}\cdot(\lfact1{ C}, \lfact2{ C})
   =
  (-1)^{m+1}\cdot(\lfact1{ C}, \lfact2{ C}).
  \]

The remaining elements of the basis of~$\HH_*(A)$ given in
Theorem~\ref{thm:basishomology} are either of degree zero, and their cap
product with~$(c,c)$ vanishes trivially, or of the form
$(\lfact1{\bar C},\lfact2{\bar C})$ with $C$ a complete circuit of length
at least~$2$, and for them we have $(\lfact1{\bar C},\lfact2{\bar
C})\frown(c,c)=0$.
\end{proof}

Now we have to deal with the part of $\HH^*(A)$ that
corresponds to cocomplete cycles. We start with the cycles in $\HH^0(A)$,
which, we recall, only exist when the algebra~$A$ is infinite dimensional.

\begin{lemma}\label{lemma:cap:4}
Let $\alpha$ be an element of~$\Crepprim(\cB)$. 
\begin{thmlist}

\item The only elements~$u$ of the basis of~$\HH_0(A)$ described in
Theorem~\ref{thm:basishomology} such that $u\frown\cycle{\alpha}\neq0$
in~$\HH_0(A)$ are 
\begin{itemize}

\item those of the form $(e,e)$ with $e$ a vertex of~$Q$ that is equal to
$s(\rot^j(\alpha))$ for some $j\in\NN_0$, for which we have
  \[
  (e,e)\frown\cycle{\alpha} = (\alpha,s(\alpha)),
  \]

\item and those of the form $(\alpha^k,s(\alpha^k))$ with $k\in\NN$, for
which we have that
  \[
  (\alpha^k,s(\alpha^k))\frown\cycle{\alpha} = (\alpha^{k+1},s(\alpha^{k+1})).
  \]

\end{itemize}

\item The only elements~$u$ of the basis of~$\HH_1(A)$ described in that
theorem such that $u\frown\cycle{\alpha}\neq0$ in~$\HH_1(A)$ are those of
the form $\hcycle{\alpha^k}$ with $k\in\NN$, and we have that
  \[
  \hcycle{\alpha^k}\frown\cycle{\alpha} = \hcycle{\alpha^{k+1}}.
  \]

\item If $m\geq2$, then $\HH_m(A)\frown\cycle{\alpha} = 0$.

\end{thmlist}
\end{lemma}

\begin{proof}
As $\alpha$ is primitive, its length and its period coincide and are
positive, 
and if their value is~$r$ then
  \[
  \cycle{\alpha}
        \coloneqq
        \sum_{i=0}^{r-1}(s(\rot^i(\alpha)),\rot^i(\alpha)).
  \]
Let $u$ be one of the cycles in homology listed in Theorem~\ref{thm:basishomology} and
let us suppose that the class of~$u\frown\cycle{\alpha}$ is not zero
in~$\HH_*(A)$. There is then a pair $(\beta,\delta)$ in~$\cB\odot\Gamma$
that appears in~$u$ with non-zero coefficient and such that $(\beta,\delta)
\frown\cycle{\alpha}\neq0$. This implies that there exists an index
$j\in\{0,\dots,r-1\}$ such that
  \(
  (\beta,\delta) \frown(s(\rot^j(\alpha)),\rot^j(\alpha))\neq0
  \), 
and this in turn implies that $\delta=s(\rot^j(\alpha))\delta$, that
$\beta\rot^j(\alpha)\in\cB$, and thus that
  \[
  (\beta,\delta)\frown(s(\rot^j(\alpha)),\rot^j(\alpha))
        = (\beta\rot^j(\alpha),\delta).
  \]
As the class of the cycle $u\frown\cycle{\alpha}$ is not zero
in~$\HH_*(A)$, we can suppose that we chose the pair~$(\beta,\delta)$ above
so that the circuit that contains $\beta\rot^j(\alpha)\delta$ is either
complete or cocomplete --- this follows from
Lemma~\ref{lemma:homology:neither}. 

Let us suppose first that 
  \begin{equation}\label{eq:zap:0}
  \claim[-0.8]{either $\alpha$ has length at least $2$ or $\beta$ has positive length}
  \end{equation}
and show that in that situation
  \begin{equation}\label{eq:zap:a}
  \claim[-0.8]{the circuit that contains the cycle $\beta\rot^j(\alpha)\delta$
  is cocomplete.}
  \end{equation}
If $\alpha$ has length at least two, this is obvious. If instead
$\alpha$ has length~$1$ and $\beta$ has positive length, then we have that
$r=1$, that $j=0$ and that $\beta\rot^j(\alpha)\delta$ is in fact
$\beta\alpha\delta$: as both~$\alpha$ and~$\beta$ have positive length and
$\beta\alpha=\beta\rot^j(\alpha)\in\cB$, the statement \eqref{eq:zap:a} is
also true.

We thus see that under the hypothesis~\eqref{eq:zap:0} the cycle
$u\frown\cycle{\alpha}$, when written according to the decomposition 
  \[
  \kk(\cB\odot\Gamma) = \bigoplus_{C\in\C'}\kk(\cB\odot\Gamma)_C
  \]
of~\eqref{eq:hom:desc} in Chapter~\ref{chapter:homology}, only has components
that are non-zero in those direct summands that correspond to elements~$C$
of~$\C'$ that are either cocomplete or neither complete nor cocomplete which may occur if $\beta$ has length 0 and $\alpha$ has length one.
Since the cycle $u\frown\cycle{\alpha}$ is also non-homologous to zero, it
follows from Lemmas~\ref{lemma:homology:cocomplete}
and~\ref{lemma:homology:neither} that it has degree~$0$ or~$1$ in homology.
Since~$\cycle{\alpha}$  has degree~$0$ in~$\HH^0(A)$, we thus see that $u$
itself has degree~$0$ or~$1$ in homology.

Now we have three cases to consider.
\begin{itemize}

\item Suppose that the length of~$\delta$ is~$0$ and that $\beta$ has
positive length. According to~\eqref{eq:zap:a} the cycle
$\beta\rot^j(\alpha)$ is cocomplete, so gentleness implies that
$\beta$ is a power of~$\rot^j(\alpha)$: there is an $l\in\NN$ such that
$\beta=\rot^j(\alpha^l)$. According to the list of
Theorem~\ref{thm:basishomology} we have that $u$ is homologous --- since a
pair such as $(\beta,\delta)$ cannot appear in any other element in that
list --- to the class of $(\alpha^l,s(\alpha))$, and using this we can
compute immediately that $u\frown\cycle{\alpha}$ is homologous to the class
of $(\alpha^{l+1},s(\alpha))$.

\item Suppose next that the lengths of~$\delta$ and of~$\beta$ are both
zero, so that in fact $u=(e,e)$ with $e$ the vertex~$s(\delta)$. In this
case we have that
  \[
  u\frown\cycle{\alpha}=(\rot^j(\alpha),s(\rot^j(\alpha)),
  \]
and this last pair is homologous to~$(\alpha,s(\alpha))$.

\item Suppose finally that the length of~$\delta$ is~$1$. The 
gentleness of the presentation implies at once that the
cycle~$\beta\rot^j(\alpha)\delta$, which is cocomplete by~\eqref{eq:zap:a},
is a power of~$\rot^{j+1}(\alpha)$, since $\alpha$ is a cocomplete and
primitive cycle in the quiver. It follows from this that there is a
positive integer~$k$ such that $\delta\beta=\rot^j(\alpha^k)$. The
circuit~$C$ that contains~$\delta\beta$ is thus the one that
contains~$\alpha^k$, and in view of Theorem~\ref{thm:basishomology} we have
that $u$ is necessarily the cocycle
  \[
  \hcycle{\bar C} 
        \coloneqq \sum_{i=0}^{r-1}(\rfact1{\rot^i(\bar C)},
                                   \rfact2{\rot^i(\bar C)})
        \in\kk(\cB\odot\Gamma_1)
  \]
and that if $D$ is the circuit that contains~$\alpha^{k+1}$ then
  \[
  u\frown\cycle{\alpha} = \hcycle{\bar D}.
  \]

\end{itemize}

We have now to consider the possibility in which the
hypothesis~\eqref{eq:zap:0} does not hold, so that $\alpha$ has length~$1$
and that $\beta$ has length~$0$. The period~$r$ of~$\alpha$ is of
course~$1$, and we simply have that $\cycle{\alpha}=(s(\alpha),\alpha)$.
As the pair $(\beta,\delta)$ appears in~$u$, and $u$ is one of the cycles
listed in Theorem~\ref{thm:basishomology}, inspecting the list given there
tells us that one of the following possibilities occurs:
\begin{itemize}

\item There is a vertex~$e$ in~$Q$ such that $u=(e,e)$.

\item There is a loop~$a$ in~$Q$ such that $a^2\not\in I$ such that
$u=(s(a),a)$.

\item There is a complete cyle~$C$ in~$(Q,I)$ whose period~$r$ and
length~$m$ are such that $(-1)^{(m+1)r}=1$ in~$\kk$, and $u$ is the cycle
  \[
  \hcycle{\bar C} =
  \sum_{i=0}^{r-1}(-1)^{(l+1)i}\cdot
        \bigl(s(\rot^i(\bar C)), \rot^i(\bar C)).
  \]

\end{itemize}

In the first case we have that 
  \[
  u\frown\cycle{\alpha}
        = (e,e)\frown(s(\alpha),\alpha)
        = \begin{cases*}
          (\alpha,s(\alpha)) & if $e=s(\alpha)$; \\
          0 & if not.
          \end{cases*}
  \]
In case $e=s(\alpha)$ the class of the resulting cycle $(\alpha,s(\alpha))$ is
non-zero in~$\HH_0(A)$ and belongs to the basis of that space with which we
are working --- as $\alpha$ has length~$1$, we  choose a representative for the
circuit that contains the cycle~$\alpha$ as $\alpha$  itself.

\smallskip

In the second case, we have
  \[
  u\frown\cycle{\alpha}
        = (s(a),a)\frown(s(\alpha),\alpha)
        = \begin{cases*}
          (\alpha,a) & if $s(a)=s(\alpha)$; \\
          0 & if not.
          \end{cases*}
  \]
If $\alpha$ and~$a$ are two different loops, then $(\alpha,a)$ is a
coboundary, and if they coincide then $(\alpha,a)$ is the basis element
$\hcycle{\overline{a^2}}$ of~$\HH_1(A)$.

Finally, suppose that we are in the third case. If $i\in\{0,\dots,r-1\}$ is
such that $s(\rot^i(\bar C))=s(\alpha)$, then gentleness implies that
$C$ has length at least~$2$, for $C$ is complete and~$\alpha$ cocomplete,
and then that $\alpha$ is both the first and the last arrow in~$\bar C$:
this is a contradiction. As
  \[
  (s(\rot^i(\bar C)),\rot^i(\bar C))\frown(s(\alpha),\alpha)
    = \begin{cases*}
      (\alpha,\rot^i(\bar C)) & if $s(\rot^i(\bar C))=s(\alpha)$; \\
      0 & if not;
      \end{cases*}
  \]
we see that $u\frown(s(\alpha),\alpha)=0$ even before passing to homology.
\end{proof}

The last step in our calculation of the cap product is  dealing with products between the classes in~$\HH^*(A)$ corresponding to
complete cycles and $\HH_*(A)$: as usual, it is when working with complete
cycles that most of the action occurs. Our final three lemmas do this.

\begin{lemma}\label{lemma:cap:5}
Let $C$ be an element of~$\Cprim(\Gamma)$ and let $D$ be a complete circuit
in~$(Q,I)$ whose length~$n$ and period~$t$ are such that $(-1)^{(n+1)t}=1$
in~$\kk$. 
\begin{thmlist}

\item If the circuits $C$ and $D$ are not powers of the same primitive
complete circuit, then $\hcycle{\bar D}\frown\cycle{\bar C}=0$.

\item If instead there exists a primitive circuit~$E$ and non-negative
integers~$w$ and~$k$ such that $\bar C=\bar E^w$ and $\bar D=\bar E^k$,
then
  \[
  \hcycle{\bar D}\frown\cycle{\bar C}
    = \begin{dcases*}
      0 
        & if $k<w$; 
        \\
      \sum_{i=0}^{r-1}(-1)^i\cdot(s(\rot^i(\bar E)),s(\rot^i(\bar E)))
        & if $k=w$; 
        \\
      \hcycle{\bar E^{k-w}}
        & if $k>w$.
      \end{dcases*}
  \]

\end{thmlist}
\end{lemma}

\begin{proof}
Let $m$ and~$r$ be the length and the period of the circuit~$C$; since $C$
is in~$\Cprim(\Gamma)$, either $m$ is even or the characteristic of~$\kk$
is~$2$. There is a complete primitive circuit~$E$ of length~$r$ and an
integer~$w$ such that $\bar C=\bar E^w$, and $w=2$ if $r$ is odd and the
characteristic of~$\kk$ is not~$2$, and $w=1$ in any other case.
The $m$-cocycle $\cycle{\bar C}$ is the sum
  \[
  \cycle{\bar C} = \sum_{i=0}^{r-1}
        (-1)^{im}\cdot(\rot^i(\bar C),s(\rot^i(\bar C)))
        \in \kk(\Gamma_{m}\parallel\cB).
  \]
Associated to the complete circuit~$D$ is the $n$-cycle 
  \[
  \hcycle{\bar D} = 
    \sum_{i=0}^{t-1}(-1)^{(n+1)i}\cdot
          \bigl(s(\rot^i(\bar D)), \rot^i(\bar D))
  \]
in the complex~$\kk(\cB\odot\Gamma)$, whose homology class is one of the
elements of the basis of~$\HH_n(A)$ described in
Theorem~\ref{thm:basishomology}. According to 
formula~\eqref{eq:frown:pairs} for the cap product, if
$i\in\{0,\dots,r-1\}$ and $j\in\{0,\dots,s-1\}$, then
  \begin{multline*}
  \hskip4em
  (s(\rot^j(\bar D)),\rot^j(\bar D))
  \frown(\rot^i(\bar C),s(\rot^i(\bar C))) \\
    = \begin{cases*}
       (-1)^{(n+m)m}\cdot(s(\rot^i(\bar C),\eta))
        & if $\rot^j(\bar D)$ factors as $\rot^i(\bar C)\eta$; \\
      0 & in any other case.
      \end{cases*}
  \end{multline*}
If $\bar D$ is not a power of the primitive cycle~$\bar E$, then for no
choice of~$i$ and~$j$ the path~$\rot^j(\bar D)$ factors as $\rot^i(\bar
C)\eta$, and this implies at once that $\hcycle{\bar D}\frown\cycle{\bar
C}=0$.

Let us suppose that instead there is an integer~$k\in\NN$ such that $\bar
D=\bar E^k$. The periods~$r$ and~$t$ of~$C$ and~$D$ are then equal, and
their definition implies that whenever $i$,~$j\in\{0,\dots,r-1\}$ we have
that
  \[
  \claim[.75]{$\rot^j(\bar D)$ factors as $\rot^i(\bar C)\eta$ if and only
  if $j=i$ and $k\geq w$, and when that is the case we have that either
  $k>w$ and $\eta=\rot^i(\bar E^{k-w})$, or $k=w$ and $\eta=s(\rot^i(\bar
  E))$.}
  \]
It follows from this that
{\small
  \[
  \hcycle{\bar D}\frown\cycle{\bar C}
        = \begin{dcases*}
          0 
                & if $k<w$; 
                \\
          \sum_{i=0}^{r-1}
          (-1)^{im +(n+1)i+(n+m)m}\cdot
          (s(\rot^i(\bar E)),s(\rot^i(\bar E))),
                & if $k=w$; 
                \\
          \sum_{i=0}^{r-1}
          (-1)^{im+(n+1)i+(n+m)m}\cdot
          (s(\rot^i(\bar E^{k-w})),\rot^i(\bar E^{k-w})),
                & if $k>w$.
          \end{dcases*}
  \]
  }
Since $m=wr$ and $n=kr$, one can check at once that for all
$i\in\{0,\dots,r-1\}$ the equality 
  \begin{align*}  (-1)^{(n+1)i+(n+m)m} 
  &=(-1)^{im+(n+1)i+(n+m)m}  \\
  &= (-1)^{((k-w)r+1)i+(n+m)m} \\
  &= (-1)^{((k-w)r+1)i} 
  \end{align*}
holds in~$\kk$, because either $m$ is even or the characteristic of the
field~$\kk$ is two. 

We are left with considering two cases.
If $k=w$, then what we have is that
  \begin{equation}\label{eq:sigh:1}
  \hcycle{\bar D}\frown\cycle{\bar C}
        = \sum_{i=0}^{r-1}
          (-1)^{i}\cdot
          (s(\rot^i(\bar E)),s(\rot^i(\bar E))),
  \end{equation}
the alternating sum of the vertices through which the cycle~$\bar E$ passes
where the homology class of each  is an element of the basis
of~$\HH_0(A)$ given by Theorem~\ref{thm:basishomology}.
If instead $k>w$, then the cap product we are plodding for is
  \begin{equation}\label{eq:sigh:2}
  \hcycle{\bar D}\frown\cycle{\bar C}
        = \sum_{i=0}^{r-1}
          (-1)^{((k-w)r+1)i+(n+m)m}\cdot
          (s(\rot^i(\bar E^{k-w})),\rot^i(\bar E^{k-w})).
  \end{equation}
The complete cycle $\bar E^{k-w}$ that appears here has length~$(k-w)r$
and period~$r$, and in~$\kk$ we have that 
  \[
  (-1)^{((k-w)r+1)r} = (-1)^{(n+1)t-mr} = 1,
  \]
so that the sum~\eqref{eq:sigh:2} above is precisely the $(k-w)r$-cycle
$\hcycle{\bar E^{k-w}}$ from the list of
Theorem~\ref{thm:basishomology}.
With this we have proved all the claims of the lemma.
\end{proof}

\begin{lemma}\label{lemma:cap:6}
Let $C$ be an element of~$\Cprim(\Gamma)$, let $D$ be a complete circuit in
the gentle presentation~$(Q,I)$ whose length~$n+1$ and period~$t$ are such
that $(-1)^{nt}=1$ in~$\kk$, and let
  \[
  (\lfact1{\bar D},\lfact2{\bar D})
  \]
be the $n$-cycle in the complex~$\kk(\cB\odot\Gamma)$ corresponding to~$D$
in the list of Theorem~\ref{thm:basishomology}.
\begin{thmlist}

\item If the circuits $C$ and $D$ are not powers of the same primitive
complete circuit, then 
  \(
  (\lfact1{\bar D},\lfact2{\bar D}) \frown\cycle{\bar C}=0
  \).

\item If instead there exists a primitive circuit~$E$ and non-negative
integers~$w$ and~$k$ such that $\bar C=\bar E^w$ and $\bar D=\bar E^k$,
then 
  \[
  (\lfact1{\bar D},\lfact2{\bar D})\frown\cycle{\bar C}
    = \begin{dcases*}
      0 
        & if $k\leq w$; 
        \\
      (\lfact1{(\bar E^{k-w})},\lfact2{(\bar E^{k-w})})
        & if $k>w$.
      \end{dcases*}
  \]

\end{thmlist}
\end{lemma}

\begin{proof}
Let $m$ and~$r$ be the length and the period of the circuit~$C$, and let
$E$ be the primitive complete circuit of length~$r$ such that there is a
positive integer~$w$ with $\bar C=\bar E^w$. Since $C$ is
in~$\Cprim(\Gamma)$, we have that $w=2$ if $r$ is odd and the
characteristic of~$\kk$ is not~$2$, and that $w=1$ in any other case.
As in Theorem~\ref{thm:basishomology}, we consider the $m$-cocycle 
  \[
  \cycle{\bar C} = \sum_{i=0}^{r-1}
       (-1)^{im}\cdot (\rot^i(\bar C),s(\rot^i(\bar C)))
        \in \kk(\Gamma_{m}\parallel\cB).
  \]
On the other hand, the 
$n$-cycle $(\lfact1{\bar D},\lfact2{\bar D})$ corresponds to the complete circuit $D$.

Let us suppose there is an element~$i$ of~$\{0,\dots,r-1\}$ such that
the cycle
  \begin{equation}\label{eq:sigh:3}
  (\lfact1{\bar D},\lfact2{\bar D})\frown(\rot^i(\bar C),s(\rot^i(\bar C)))
  \end{equation}
is not zero in~$\kk(\cB\odot\Gamma)$. According to our
formula~\eqref{eq:frown:pairs} for the cap product, the path $\lfact2{\bar
D}$ then factorizes as $\rot^i(\bar C)\eta$ and we have that
$s(\lfact1{\bar D})=s(\rot^i(\bar C))$. It follows from this that
  \begin{equation}\label{eq:sigh:4}
  \rot(\bar D) 
        = \lfact2{\bar D}\lfact1{\bar D}
        = \rot^i(\bar C)\eta\lfact1{\bar D}
        = \rot^i(\bar E^w)\eta\lfact1{\bar D}
  \end{equation}
and, since the circuit~$E$ is primitive and the first and last members of
this chain of equalities are complete cycles, we see that there is a positive
integer~$k$ such that $k\geq w$ and $\rot(\bar D)=\rot^i(\bar E^{k})$ and that the
length and the period of the circuit~$D$ are~$kr$ and~$r$, respectively.

Notice that $\bar D=\rot^{i-1}(\bar E^k)$. As $\bar E^k$ is an element
of~$\overline{\C^\circ}(\Gamma)$ and it is conjugate to $\rot^{i-1}(\bar
E^k)$, which is also an element of~$\overline{\C^\circ}(\Gamma)$, we must
have that $\rot^{i-1}(\bar E^k)=\bar E^k$: this tells us that necessarily
  \begin{equation}\label{eq:sigh:4.5}
  \claim{either $i=1$ and $r>1$ or $i=0$ and $r=1$.}
  \end{equation}
In particular, of course, there is exactly one element~$i$ in~$\{0,\dots,r-1\}$ such
that the product~\eqref{eq:sigh:3} is non-zero, and therefore
  \begin{align}
  (\lfact1{\bar D},\lfact2{\bar D})\frown\cycle{\bar C}
       &= (-1)^{im}\cdot
          (\lfact1{\bar D},\lfact2{\bar D})
          \frown
          (\rot^i(\bar C),s(\rot^i(\bar C)))
          \notag
          \\
       &= (-1)^{im}\cdot
          (\lfact1{\bar D},\eta) 
       = 
          (\lfact1{\bar D},\eta)
          \label{eq:sigh:5}
  \end{align}

If $k$ were equal to~$w$, then the cycle~$(\lfact1{\bar D},\lfact2{\bar
D})$ would have degree~$wr-1$ and, as the cocycle $\cycle{\bar C}$ has
degree~$wr$, the product in~\eqref{eq:sigh:3} ought to be zero simply
because it has negative degree. As it is not zero, we see that
$k>w$ and then, in view of the equality~\eqref{eq:sigh:4}, we have that
$\rot^i(\bar E^{k-w})=\eta\lfact1{\bar D}$. This implies that 
$\rot^{i-1}(\bar E^{k-w})=\lfact1{\bar D}\eta$ and therefore that
  \[
  \lfact1{\bar D} = \lfact1{\rot^{i-1}(\bar E^{k-w})},
  \qquad
  \eta = \lfact2{\rot^{i-1}(\bar E^{k-w})}.
  \]
Using this in~\eqref{eq:sigh:5} and recalling~\eqref{eq:sigh:4.5} we can conclude that
  \begin{align*}
  (\lfact1{\bar D},\lfact2{\bar D})\frown\cycle{\bar C}
       &= (-1)^{im}\cdot
          (\lfact1{\rot^{i-1}(\bar E^{k-w})},\lfact2{\rot^{i-1}(\bar E^{k-w})})
          \\
       &= \begin{cases*}
          (\lfact1{(\bar E^{k-w})},\lfact2{(\bar E^{k-w})})
                & if $r=1$; 
                \\
          (-1)^{m}\cdot
          (\lfact1{(\bar E^{k-w})},\lfact2{(\bar E^{k-w})})
                & if $r>1$.
          \end{cases*}
  \end{align*}
Moreover, since $m$ es even if the characteristic of~$\kk$ is not~$2$, this
is simply
  \[
  (\lfact1{(\bar E^{k-w})},\lfact2{(\bar E^{k-w})})
  \]
independently of the value of~$r$.
\end{proof}

The very final piece of our calculation is very simple and
confronts us with a new exceptional gentle presentation:

\begin{lemma}\label{lemma:cap:7}
Let us suppose that the presentation~$(Q,I)$ is \emph{not} isomorphic to  
the one with
  \[
  Q =
  \begin{tikzpicture}[
        auto, 
        thick, 
        font=\footnotesize,
        scale=0.8, 
        baseline=(1.base),
        ]
  \node (1) at (0,0) {$1$};
  \draw [->] (1) edge[out=180-45, in=180+45, loop, swap] node {$a$} (1);
  \draw [->] (1) edge[out=0-45, in=0+45, loop, swap] node {$b$} (1);
  \pgfresetboundingbox
  \useasboundingbox (-2,-1) rectangle (2,1);
  \end{tikzpicture}
  \qquad
  \qquad
  \qquad
  I = \langle a^2, b^2\rangle
  \]
If $C$ is an element of~$\Cprim(\Gamma)$ and $D$ is a cocomplete
circuit in the gentle presentation~$(Q,I)$, then 
  \(
  \hcycle{\bar D}\frown\cycle{\bar C} = 0
  \)
in the complex~$\kk(\cB\odot\Gamma)$ and, \emph{a fortiori}, also
in its homology~$\HH_*(A)$.
\end{lemma}

\begin{proof}
Let $m$ and~$r$ be the length and the period of the circuit~$C$, and,
as in Proposition~\ref{prop:hh:m} and Lemma~\ref{lemma:sum}, let
  \[
  \cycle{\bar C} = \sum_{i=0}^{r-1}
        (-1)^{im}\cdot(\rot^i(\bar C),s(\rot^i(\bar C)))
        \in \kk(\Gamma_{m}\parallel\cB)
  \]
be the $m$-cocycle associated to~$C$. Let, on the other hand, $t$ be the
period of the complete circuit~$D$ and let
  \[
  \hcycle{\bar D} 
        \coloneqq \sum_{j=0}^{t-1}(\rfact1{\rot^j(\bar D)},
                                   \rfact2{\rot^j(\bar D)})
        \in\kk(\cB\odot\Gamma_1),
  \]
be the $1$-cycle in the complex~$\kk(\cB\odot\Gamma)$ described in
Theorem~\ref{thm:basishomology}. Let us suppose that the product
$\hcycle{\bar D}\frown\cycle{\bar C}$ is not zero in~$\HH_{m-1}(A)$. In
particular, there are indices
$i\in\{0,\dots,r-1\}$ and
$j\in\{0,\dots,t-1\}$ such that  \[
  (\rfact1{\rot^j(\bar D)},\rfact2{\rot^j(\bar D)})
  \frown (\rot^i(\bar C),s(\rot^i(\bar C)))
  \]
is a non-zero element of~$\kk(\cB\odot\Gamma)$. According to our
formula~\eqref{eq:frown:pairs}, this implies that the
path~$\rfact2{\rot^j(\bar D)}$ factorizes as $\rot^i(\bar C)\eta$ for some
path~$\eta$ and, since $\rfact2{\rot^j(\bar D)}$ has length~$1$ and $\bar
C$ has positive length, this tells us that $\bar C$ itself has length~$1$,
so $m=1$, $r=1$, $i=0$, $C=\rfact2{\rot^j(\bar D)}$ and, of course, that
$\eta=s(\rfact2{\rot^j(\bar D)})=s(\rot^j(\bar D))$. 

As the number $t$ is the period of~$\bar D$, which is a cocomplete cycle,
using the gentleness of~$(Q,I)$ we can see that the $t$ arrows
$\rfact2{\rot^0(\bar D)}$,~\dots,~$\rfact2{\rot^{t-1}(\bar D)}$ are
pairwise different, so that in fact the index~$j$ is uniquely determined.
It follows from all this that in fact
  \[
  \hcycle{\bar D}\frown\cycle{\bar C}
        = (\rfact1{\rot^j(\bar D)},\rfact2{\rot^j(\bar D)})
          \frown(\bar C,s(\bar C))
        = (\rfact1{\rot^j(\bar D)},s(\rot^j(\bar D))).
  \]
This last $0$-chain is homogeneous with respect to the direct sum
decomposition~\eqref{eq:hom:desc} of the complex~$\kk(\cB\odot\Gamma)$ that
we described in Chapter~\ref{chapter:homology}, and is in the direct
summand corresponding to the element of~$\C'$ that contains
$\rfact1{\rot^j(\bar D)}$, and therefore, according to
Lemma~\ref{lemma:homology:neither}, it is a coboundary unless either 
$\rfact1{\rot^j(\bar D)}$ has length zero or it has positive length and is 
complete or cocomplete. Let us analize the first two of these possibilities.
\begin{itemize}

\item The path $\rfact1{\rot^j(\bar D)}$ cannot be of length zero, for then
$\bar D$, a cocomplete cycle, would coincide with~$\bar C$, a complete
cycle.

\item Suppose that $\rfact1{\rot^j(\bar D)}$ has positive length and that
it is a cocomplete cycle. As $\rot^j(\bar D)=\rfact1{\rot^j(\bar
D)}\rfact2{\rot^j(\bar D)}$ is also a cocomplete cycle, the
gentleness of~$(Q,I)$ allows us then to conclude that $\rot^j(\bar D)$ is a
power of~$\rfact2{\rot^j(\bar D)}$. As $\bar C=\rfact2{\rot^j(\bar D)}$, in
this situation we have that $\bar D$, a cocomplete cycle, is a power
of~$\bar C$, a complete cycle: this cannot happen.

\end{itemize}
We thus see that $\rfact1{\rot^j(\bar D)}$ necessarily has positive length
and is a complete cycle. As it is a factor of~$\rot^j(\bar D)$, which is
also a cocomplete cycle, the length of $\rfact1{\rot^j(\bar D)}$ is
exactly~$1$. We thus see that $a=\rfact1{rot^j(\bar D)}$ and
$b=\rfact2{rot^j(\bar D)}$ are two different loops such that $a^2\in I$,
$b^2\in I$, $ab\not\in I$, $ba\not\in I$, $\bar C=a$ and $\rot^j(\bar D)=ab$.
The gentleness of~$(Q,I)$ and the connectedness of~$Q$ therefore
imply that $s(a)$ is the unique vertex of the quiver~$Q$, that~$a$
and~$b$ are its only two arrows, and that the ideal~$I$ is generated by
their squares: in other words, the presentation~$(Q,I)$ is isomorphic to
the one described in the statement of the lemma. The claim of the latter is
thus true.
\end{proof}

We can summarize our findings as follows:

\begin{theorem}\label{thm:cap}
Let $(Q,I)$ be a gentle presentation in which $Q$ has more than one
vertex, and let $T$ be a spanning tree of~$Q$. If $u$ is an element of the
basis of~$\HH^*(A)$ described at the beginning of
Section~\ref{sect:corollaries} that belongs to the generating set~$\G$ of
Proposition~\ref{prop:G} and $v$ one of the basis elements of~$\HH_*(A)$ described
in Theorem~\ref{thm:basishomology} such that the cap product $v\frown u$ is
not zero in~$\HH_*(A)$, then one of the following statements holds.
\begin{thmlist}

\item\label{it:cap:1} There is an arrow~$c$ in~$Q_1\setminus T$ and
a cocomplete  circuit~$C$ in which the arrow~$c$ appears such that 
  \[
  u=(c,c),
  \qquad
  v=\hcycle{\bar C},
  \qquad
  v\frown u = (\bar C,s(\bar C)).
  \]

\item\label{it:cap:2} There is an arrow~$c$ in~$Q_1\setminus T$ and
a complete  circuit~$C$ in which~$c$ appears and whose length~$m$ and
period~$r$ are such that $(-1)^{(m+1)r}=1$ in~$\kk$ such that 
  \[
  u=(c,c),
  \qquad 
  v=\hcycle{\bar C},
  \qquad
  v\frown u = (-1)^{m+1}(\lfact1{\bar C},\lfact2{\bar C}).
  \]

\item\label{it:cap:3} There is an element~$\alpha$ of~$\Crepprim(\cB)$ and a vertex~$e$
in~$Q$ through which the cycle~$\alpha$ passes such that 
  \[
  u=\cycle{\alpha},
  \qquad 
  v=(e,e),
  \qquad
  v\frown u = (\alpha,s(\alpha)).
  \]

\item\label{it:cap:4} There is an element~$\alpha$ of~$\Crepprim(\cB)$ and a positive
integer~$k\in\NN$ such that 
  \[
  u=\cycle{\alpha},
  \qquad 
  v=(\alpha^k,s(\alpha^k)),
  \qquad
  v\frown u = (\alpha^{k+1},s(\alpha^{k+1})).
  \]

\item\label{it:cap:5} There is an element~$\alpha$ of~$\Crepprim(\cB)$ and a positive
integer~$k\in\NN$ such that 
  \[
  u=\cycle{\alpha},
  \qquad
  v=\hcycle{\alpha^k},
  \qquad
  v\frown u = \hcycle{\alpha^{k+1}}.
  \]

\item\label{it:cap:6} There is an element~$C$ in the set~$\Cprim(\Gamma)$, a complete
circuit~$D$ whose length~$n$ and period~$r$ are such that  $(-1)^{(n+1)r}=1$
in~$\kk$, a primitive circuit~$E$, and non-negative integers~$w$ and~$k$
such that 
  \[
  k\geq w, 
  \quad
  \bar C=\bar E^w, 
  \quad
  \bar D=\bar E^k, 
  \quad
  u = \cycle{\bar C},
  \quad
  v = \hcycle{\bar D},
  \]
and we have that
  \[
  v\frown u
    = \begin{dcases*}
      \sum_{i=0}^{r-1}(-1)^i\cdot(s(\rot^i(\bar E)),s(\rot^i(\bar E)))
        & if $k=w$; 
        \\
      \hcycle{\bar E^{k-w}}
        & if $k>w$.
      \end{dcases*}
  \]

\item\label{it:cap:7} There is an element~$C$ in the set~$\Cprim(\Gamma)$, a complete
circuit~$D$ whose length~$n+1$ and period~$t$ are such that $(-1)^{nt}=1$
in~$\kk$, a primitive circuit~$E$, and non-negative
integers~$w$ and~$k$ such that 
  \[
  k>w, 
  \quad
  \bar C=\bar E^w, 
  \quad
  \bar D=\bar E^k,
  \quad
u=\cycle{\bar D}, 
  \quad
  v=(\lfact1{\bar D},\lfact2{\bar D})
  \]
and we have that
  \[
  v\frown u
    = (\lfact1{(\bar E^{k-w})},\lfact2{(\bar E^{k-w})}).
  \]

\end{thmlist}
\end{theorem}

\begin{proof}
This is the information that is contained in Lemmas~\ref{lemma:cap:1},
\ref{lemma:cap:2}, \ref{lemma:cap:3}, \ref{lemma:cap:4}, \ref{lemma:cap:5},
\ref{lemma:cap:6}, and \ref{lemma:cap:7}.
\end{proof}

\begin{remark}\label{rem:exception-homology}
Let us suppose that $Q$ has exactly one vertex and one arrow a. When the presentation $(Q,I)$ is \fd gentle, we have $a^2\in R$ and depending of the ground field there are two cases. 
\begin{itemize}
	\item If the characteristic of~$\kk$ is not~$2$, by  Corollary~\ref{coro:homology:basis:tr} we know that the graded vector space~$\HH_*(A)$ is freely spanned by the homology class of~$(s(a),s(a))$, and for each $m\geq0$ the classes of $\hcycle{a^m}$ if $m$ is odd and by $(a, a^m)$ if $m$ is even. According to Remark~\ref{rem:exception} the vector space~$\HH^0(A)$ is freely spanned by~$\one$ and~$(s(a),a)$, and for each $n\geq1$ the vector space $\HH^n(A)$ is freely spanned by~$(a^n,a)$ if $n$ is odd, and by $\cycle{a^n}$ if $n$ is even.
	
	Using these bases and Equation \ref{eq:frown:pairs}, we obtain for all integers $m\geq 0$ and $n \geq 1$:
	


\begin{table}[H]
	\centering
	\setlength\extrarowheight{2pt}
	\setlength{\tabcolsep}{7pt}
	\setlength{\aboverulesep}{0pt}
	\setlength{\belowrulesep}{0pt}
	\begin{tabular}{c@{\hskip1.3em}*{3}{c}}
		$\frown$
		& $(s(a),a)$
		& $(a^{n},a) $
		& $\cycle{a^{n}}$
		\\ \toprule
		$(s(a),s(a))$
		& $ (a,s(a))$
		& $0$
		& $0$
		
		\\ \midrule
		$\hcycle{a^m}$
		& $0$
		& $(a,a^{m-n})$
		& $\hcycle{a^{m-n}}$
		
		\\ \midrule
		$(a,a^m)$
		& $0$
		& $0$
		& $(a,a^{m-n})$
		\\ \bottomrule
	\end{tabular}
	\bigskip
	\caption{}
	\label{tbl:x}
\end{table}
	
	\item If the characteristic of~$\kk$ is~$2$, by  Corollary~\ref{coro:homology:basis:tr} \sloppy the graded vector space~$\HH_*(A)$ is freely spanned by the homology class of~$(s(a),s(a))$, and, for each $m\geq1$, the classes of $\hcycle{a^m}$ and, for $m \geq0$ the classes of $(a, a^{m})$. On the other hand, by   Remark~\ref{rem:exception}  the vector space~$\HH^0(A)$ is freely spanned by~$\one$ and~$(s(a),a)$. For  each $n\geq1$ the vector space~$\HH^n(A)$ is freely spanned by~$\cycle{a^n}$ and by $(a^n,a)$. For all integers $m\geq n\geq 1$ we have that

\begin{table}[H]
	\centering
	\setlength\extrarowheight{2pt}
	\setlength{\tabcolsep}{7pt}
	\setlength{\aboverulesep}{0pt}
	\setlength{\belowrulesep}{0pt}
	\begin{tabular}{c@{\hskip1.3em}*{3}{c}}
		$\frown$
		& $(s(a),a)$
		& $(a^{n},a) $
		& $\cycle{a^{n}}$
		\\ \toprule
		$(s(a),s(a))$
		& $ (a,s(a))$
		& $0$
		& $0$
		
		\\ \toprule
		$(a,s(a))$
		& $0$
		& $0$
		& $0$
		
		\\ \midrule
		$\hcycle{a^m}$
		& $(a,a^m)$
		& $(a,a^{m-n})$
		& $\hcycle{a^{m-n}}$
		
		\\ \midrule
		$(a,a^m)$
		& $0$
		& $0$
		& $(a,a^{m-n})$
		\\ \bottomrule
	\end{tabular}
	\bigskip
	\caption{}
	\label{tbl:xx}
\end{table}

\end{itemize}

Finally, if the representation $(Q, I)$ is not \fd gentle, then there are no relations; the graded vector space $\HH_*(A)$ is freely spanned by the homology classes of $(s(a),s(a))$,  the elements $(a^m, s(a))$ and  $\hcycle{a^{m+1}}$ one for each $m\geq 1$. Using  Remark~\ref{rem:exception}  the vector space~$\HH^0(A)$ is freely spanned by~$\one$ and the elements~$\cycle{a^n}$, one for each $n>0$, the vector space~$\HH^1(A)$ is freely spanned by the pairs~$(a,a^n)$, one for each $n\geq0$ and  for all $n>1$ we have $\HH^n(A)=0$. For all $m\geq1$ and for appropriate values of $n$ we obtain 

\begin{table}[H]
	\centering
	\setlength\extrarowheight{2pt}
	\setlength{\tabcolsep}{7pt}
	\setlength{\aboverulesep}{0pt}
	\setlength{\belowrulesep}{0pt}
	\begin{tabular}{c@{\hskip1.3em}*{3}{c}}
		$\frown$
		& $\cycle{a^n}$
		& $(a,a^n)$
		\\ \toprule
		$(s(a),s(a))$
		& $(a^n,s(a))$
		& $0$
		
		\\ \midrule
		$(a^m,s(a))$
		& $(a^{m+n},s(a))$
		& $0$
		
		\\ \midrule
		$\hcycle{a^{m+1}}$
		& $\hcycle{a^{m+n+1}}$
		& $(a^{m+n},s(a))$
		\\ \bottomrule
	\end{tabular}
	\bigskip
	\caption{}
	\label{tbl:y}
\end{table}

\end{remark}

In the following description, we assume that the reader is familiar with the length and characteristic restrictions for each of the elements in both homology and cohomology.

\begin{remark}
	This computation concludes the omitted case in Lemma \ref{lemma:cap:7}. Let us suppose that Q has two loops, namely, $a$ and $b$, which coincide at the same vertex. Let I be the ideal generated by the relation $x^2$ where $x\in \{a,b\}$. Using Corollary~\ref{coro:homology:basis:tr} the graded vector space~$\HH_*(A)$ is freely spanned by the homology classes of~$(s(a),s(a))$, $(x,s(a))$, $\hcycle{(ba)^n}$, $((ba)^n,s(a))$, $\hcycle{x^n}$ and $(x,x^{n})$ for each $n\geq1$; notice that the last two elements appear under certain conditions of length and characteristic, but if $n=1$ the last element only appears if the characteristic is 2. On the other hand, the graded vector space ~$\HH^*(A)$ is freely spanned by the cohomology classes $\one$, $\cycle{(ab)^m}$, $(a,a(ba)^m)$, $(x,x)$, $\cycle{x^m}$, $(x^m,x)$ for all $m\geq1$, it is worth mentioning that the characteristic, as well as the length, influence the last two types of elements. For appropriate values $n$ and $m$ we obtain:
	
	{\scriptsize 
	\begin{table}[H]
		\centering
		\setlength\extrarowheight{2pt}
		\setlength{\tabcolsep}{7pt}
		\setlength{\aboverulesep}{0pt}
		\setlength{\belowrulesep}{0pt}
		\begin{tabular}{c@{\hskip1.3em}*{8}{c}}
			$\frown$
			& $\cycle{(ab)^m}$
			& $(a, a(ba)^m)$
			& $(x,x)$
			& $\cycle{x^m}$
			& $(x^m,x)$
			
			\\ \toprule
			$(s(a),s(a))$
			& $2((ba)^n,s(a))$
			& 0
			& 0
			& 0
			& 0
			
			\\ \toprule
			$((ba)^n,s(a))$
			& $((ba)^{n+m},s(a))$
			& 0
			& 0
			& 0
			& 0
			
			\\ \midrule
			$(x,s(x))$
			& 0
			& 0
			& 0
			& 0
			& 0
			
			\\ \midrule
			$\hcycle{(ab)^{n}}$
			& $\hcycle{(ab)^{n+m}}$
			& $((ab)^{n+m},s(a))$
			& $((ab)^{n},s(a))$
			& 0
			& 0
			
			\\ \midrule
			$\hcycle{x^n}$
			& 0
			& 0
			& $(x,x^{n-1})$
			& $(s(a),x^{n-m})$
			& $(x,x^{n-m})$
			
			\\ \midrule
			$(x,x^{n})$
			& 0
			& 0
			& 0
			& $(x,x^{n-m})$
			& $0$
			
			\\ \bottomrule
		\end{tabular}
		\bigskip
		\caption{}
		\label{tbl:dobleloop}
	\end{table}
}
\end{remark}

\section{Some consequences}
\label{sect:cup:consequences}

In this section, we fix a  gentle presentation~$(Q,I)$, set
$A\coloneqq\kk Q/I$, and explore some consequences of our calculation of
the algebra structure on Hochschild cohomology $\HH^*(A)$ of~$A$. 

\bigskip

First we consider  the graded Jacobson radical of the
algebra~$\HH^*(A)$, that is, the intersection of its maximal left
homogeneous ideals.

\begin{lemma}\label{lemma:radical}
The graded Jacobson radical $\rad\HH^*(A)$ of the algebra
$\HH^*(A)$ is the subspace spanned by 
\begin{itemize}

\item the positive part~$\HH^+(A)$ and 

\item the pairs~$(s(\alpha),\alpha)\in\HH^0(A)$ with $\alpha$ a
$\cB$-maximal cycle in~$(Q,I)$. 

\end{itemize}
\end{lemma}

\begin{proof}
Since the algebra~$\HH^*(A)$ is non-negatively graded, the sum
of any of its proper graded ideals with~$\HH^0(A)$ is also a proper graded
ideal: this tells us that the subspace~$\HH^+(A)$ is contained in the
radical~$J$ and that $J$~is in fact the preimage of the Jacobson radical of
the algebra~$\HH^0(A)$, isomorphic to  the center of~$A$, by the
projection~$\HH^*(A)\to\HH^*(A)/\HH^+(A)=\HH^0(A)$. It follows from
Theorem~\ref{thm:algebra} that~$\HH^0(A)$ is the commutative algebra freely
generated by 
\begin{itemize}

\item the pairs~$(s(\alpha),\alpha)$ with $\alpha$ a $\cB$-maximal cycle
in~$(Q,I)$, and 

\item the sums~$\cycle{\alpha}$ with $\alpha\in\Crepprim(\cB)$.

\end{itemize}
subject to the following relations: the product of any two generators is
zero and the squares of the generators in the first group are all zero. Now,
$\HH^0(A)$ is a finitely generated and commutative algebra over a field, so
its Jacobson radical coincides with its nilradical: one sees at once using
this that its radical is the ideal generated by the generators in the first
group, and the lemma follows from this.
\end{proof}

\begin{corollary}\label{coro:top-algebra}
Let $n$ be the number of primitive cocomplete cycles in the  gentle
presentation~$(Q,I)$. The algebra $\HH^*(A)/\rad\HH^*(A)$ is isomorphic to
the quotient of the polynomial algebra~$\kk[x_1,\dots,x_n]$ by the
quadratic monomial ideal
  \[
  (x_ix_j:1\leq i<j\leq n).
  \]
In particular, the number~$n$ is a derived invariant of the algebra $A$:
two  gentle presentations giving derived equivalent algebras have
the same number of primitive cocomplete cycles.
\end{corollary}

Notice that the quotient $\HH^*(A)/\rad\HH^*(A)$ is in principle a graded
algebra, but in our situation it is concentrated in degree~$0$.

\begin{proof}
The description of the algebra~$B\coloneqq\HH^*(A)/\rad\HH^*(A)$ follows
immediately from that of the radical given by Lemma~\ref{lemma:radical} and
the presentation of~$\HH^*(A)$ given by Theorem~\ref{thm:algebra}. There is
in~$B$ a unique maximal ideal~$\m$ such that the localization
$B_{\m}$ is not regular, and its Zariski cotangent space $\m B_\m/\m^2
B_\m$ has dimension~$n$ as a vector space over~$\kk$: this shows that we
can compute the number~$n$ from the algebra~$\HH^*(A)$, which is a derived
invariant of~$A$, and proves the last claim in the corollary.
\end{proof}

Since the algebra~$\HH^*(A)/\rad\HH^*(A)$ is simply the quotient of the
center of~$A$ by its radical, we do not really need to know the Hochschild
cohomology of~$A$ to prove the derived invariance of the number of
primitive cocomplete cycles. On the other hand, we cannot take the next
natural step without knowing that cohomology.

\begin{corollary}
Let $(Q,I)$ be a gentle presentation and let $A\coloneqq\kk Q/I$.
The quotient
  \[
  \mathcal{T}(A) \coloneqq \frac{\rad\HH^*(A)}{\rad^2\HH^*(A)}
  \]
is a finite-dimensional graded vector space, and its
Hilbert-Poincar\'e polynomial
  \[
  h_{\mathcal{T}(A)}(t) \coloneqq
    \sum_{m\geq0} \dim\mathcal{T}^m(A)\cdot t^m
    \in\ZZ[t]
  \]
is a derived invariant of the algebra~$A$.
\end{corollary}

\begin{proof}
In view of Lemma~\ref{lemma:radical} and the form of the basis
for~$\HH^*(A)$ that we gave at the beginning of
Section~\ref{sect:corollaries}, the vector space~$\rad\HH^*(A)$ is
freely spanned by the classes of the following elements.
\begin{thmlist}

\item The pairs $(s(\alpha),\alpha)$ with $\alpha$ a
$\cB$-maximal path in~$(Q,I)$.

\item The pairs $(c,c)$ with $c$ an arrow in the
complement of the spanning tree~$T$.

\item The pairs $(\gamma,\alpha)$ with $\gamma$ a
$\Gamma$-maximal element of~$\Gamma$ and $\gamma$ and~$\alpha$ neither
beginning nor ending with the same arrow.

\item The sums $\cycle{C}$ with $C\in\Crep(\Gamma)$,
as defined in Lemma~\ref{lemma:sum}.

\item The pairs $(bC,b)$ with $C\in\Crep(\Gamma)$ and
$b$ the first arrow of~$C$.

\end{thmlist}
and $\rad^2\HH^*(A)$ is spanned by those elements in this list
that are a scalar multiple of a product of the elements of the list.
 Clearly, the elements in (v) and those in (iv) that correspond
to a cycle~$C$ that is not in $\Crepprim(\Gamma)$. Moreover, considering all possible factorisations, we will see that none of the other elements are in~$\rad^2\HH^*(A)$.
We thus see that the graded vector
space~$\mathcal{T}(A)$ is freely generated by the classes of the following
homogeneous elements.

\begin{itemize}

\item The pairs $(s(\alpha),\alpha)$ with $\alpha$ a
$\cB$-maximal path in~$(Q,I)$.

\item The pairs $(c,c)$ with $c$ an arrow in the
complement of the spanning tree~$T$.

\item The pairs $(\gamma,\alpha)$ with $\gamma$ a
$\Gamma$-maximal element of~$\Gamma$ and $\gamma$ and~$\alpha$ neither
beginning nor ending with the same arrow.

\item The sums $\cycle{C}$ with $C\in\Crepprim(\Gamma)$.

\end{itemize}
These are finitely many classes, so $\mathcal{T}(A)$ is finite dimensional.
The Hilbert--Poincar\'e polynomial~$h_{\mathcal{T}(A)}(t)$ can be computed
purely in terms of the graded algebra structure of~$\HH^*(A)$, so it is a
derived invariant of the algebra~$A$.
\end{proof}

\chapter{The Gerstenhaber bracket}
\label{chapter:structure}

The Gerstenhaber bracket on the Hochschild cohomology algebra~$\HH^*(A)$ 
was constructed by M.\,Gerstenhaber originally in terms of the
standard Hochschild complex of~$A$, and this is problematic when doing
explicit computations: when we use that construction we are forced to deal
with comparison morphisms between the standard resolution of the algebra
and the projective resolution that we used to compute~$\HH^*(A)$ --- which
in essentially all cases is not the standard one. This is the motivation of
recent work of C.\,Negron and S.\,Witherspoon \cite{Negron-Witherspoon} and
Yu.\,Volkov \cite{Volkov} that develops methods to do the computation
directly in terms of an arbitrary resolution. Similarly,
M.\,Su\'arez-{\'A}lvarez presented in~\cite{MSA} a different idea that allows
for a reasonably practical calculation of the Lie action of~$\HH^1(A)$
on~$\HH^*(A)$. We will combine these two approaches to exhibit the Lie
algebra structure on~$\HH^*(A)$ for gentle algebras.

\section{Computation of the Gerstenhaber bracket.}
The method of Negron and Witherspoon in our case amounts to the following.
We write as before ${\mathbb B}A$ for the bar resolution of~$A$ and $\epsilon:{\mathbb B}A\to A$
for its augmentation.  There is an injective morphism 
$\iota:\cR\to {\mathbb B}A$ 
of complexes
of $A$-bimodules such that $\epsilon\circ\iota=\mu$,
the augmentation of~$\cR$, with a section ${\mathbb B}A\to\cR$ that is also
a morphism of complexes of~$A$-bimodules. Moreover, if
$\Delta_{{\mathbb B}}:{\mathbb B}A\to {\mathbb B}A\otimes_A {\mathbb B}A$ is the standard diagonal map
of~${\mathbb B}A$, with
  \[
  \Delta_{\mathbb B}(1\otimes a_1\otimes\cdots\otimes a_{p}\otimes1)
        = \sum_{i=0}^{p} 
                (1\otimes a_1\otimes\cdots\otimes a_i\otimes 1)
                \otimes_A
                (1\otimes a_{i+1}\otimes\cdots\otimes a_p\otimes 1)
  \]
for each $p\geq0$ and each elementary tensor $1\otimes a_1\otimes\cdots\otimes
a_{p}\otimes 1$ in~${\mathbb B}_p(A)$. The diagonal
map~$\Delta:\cR\to\cR\otimes_A\cR$ that we constructed in
Chapter~\ref{chapter:cup} has the property that
$\Delta_{{\mathbb B}}\circ\iota = \iota\otimes\iota\circ\Delta$.
The complexes~$\cR\otimes_A\cR$ and~$\cR$ are $A^e$-projective resolutions
of~$A$, and the morphism of complexes
$F\coloneqq\mu\otimes\id_{\cR}-\id_{\cR}\otimes\mu:\cR\otimes_A\cR\to\cR$
is a lift of the zero map $0:A\to A$ to these resolutions. So there exists
a homotopy $\phi:\cR\otimes_A\cR\to\cR[-1]$ such that 
$F = \phi\circ d_{\cR\otimes_A\cR} + d_{\cR}\circ\phi$. Let now $f:\cR_p\to
A$ and $g:\cR_q\to A$ be a $p$- and a $q$-cocycle in the
complex~$\Hom_{A^e}(\cR,A)$, which we view as morphisms of complexes
$\cR\to A[p]$ and $\cR\to A[q]$. We can then define 
$f\circ_\phi g:\cR_{p+q-1}\to A$ to be the $(p+q-1)$-cocycle corresponding
to the composition
  \[
  \begin{tikzcd}
  \cR \arrow[d, swap, "\Delta^{(2)}"]
    &[2em]
    &[-1.5em] A[p+q-1]
    \\
  \cR\otimes_A\cR\otimes_A\cR \arrow[r, "\id_\cR\otimes g\otimes\id_\cR"]
    &[2em] \cR\otimes_AA[q]\otimes_A\cR 
      = (\cR\otimes_A\cR)[q] \arrow[r, "{\phi[q]}"]
    &[-1.5em] \cR[q-1] \arrow[u, swap, "{f[q-1]}"]
  \end{tikzcd}
  \]
with $\Delta^{(2)} \coloneqq \Delta\otimes\id_\cR\circ\Delta :
\cR\to\cR\otimes_A\cR\otimes_A\cR$, and then put
  \[
  [f,g]_\phi \coloneqq f\circ_\phi g
                       - (-1)^{(p-1)(p-1)}g\circ_\phi f
        : \cR_{p+q-1}\to A.
  \]
It is important to remember that when one evaluates compositions like this
on an element of~$\cR$ implicit Koszul signs appear.

Negron and Witherspoon show in~\cite{Negron-Witherspoon} that this bracket
operation~$[\place,\place]_\phi$ on the cocycles of the
complex~$\Hom_{A^e}(\cR,A)$ descends to its cohomology, which we are
identifying canonically with $\HH^*(A)$, and gives there the Gerstenhaber
bracket.

To carry this procedure out, we need to choose a
homotopy~$\phi:\cR\otimes_A\cR\to\cR[-1]$. To define one, we note that 
the complex of $A$-bimodules $\cR\otimes_A\cR$ is freely
spanned as a vector space by the elementary tensors of the form $(u\otimes
a\otimes b)\otimes_A(1\otimes c\otimes v)$ with $u$,~$b$,~$v\in\cB$,
$a$,~$c\in\Gamma$, and $t(v)=s(c)$, $t(c)=s(b)$, $t(b)=s(a)$
and $t(a)=s(u)$, and we make the convention that whenever we write an
elementary tensor in~$\cR\otimes_A\cR$ these conditions are satisfied. With
this in mind, we put, for each choice of integers $m$,~$n$,~$r\geq0$ and
paths $a=a_m\cdots a_1\in\Gamma_m$, $b=b_r\cdots b_1\in\cB_r$,
$c=c_n\cdots c_1\in\Gamma_n$, and $u$,~$v\in\cB$,
  \begin{multline*}
  \phi\bigl(
      (u\otimes a\otimes b) \otimes(1\otimes c\otimes v)
      \bigr) \\
      \coloneqq
      \begin{dcases*}
      \sum_{i=1}^r ub_r\cdots b_{i+1}\otimes b_i\otimes b_{i-1}\cdots b_1v
                & if $n=m=0$; \\
      (-1)^mu\otimes ab_r\otimes b_{r-1}\cdots b_1v
        & if $m>0$, $n=0$, $r\geq1$ and $a_1b_r\in R$; \\
      ub_r\cdots b_2\otimes b_1c\otimes v
        & if $m=0$, $n>0$, $r\geq1$ and $b_1c_n\in R$; \\
      (-1)^mu\otimes abc\otimes v
        & if $m$,~$n>0$, $r=1$, and $a_1b_1$,~$b_1c_n\in R$; \\
      0 & in any other case.
      \end{dcases*}
  \end{multline*}
A rather annoying calculation, which we omit here, shows that with this
choice of~$\phi$ we indeed have that $F=\phi\circ
d_{\cR\otimes_A\cR}+d_{\cR}\circ\phi$.

With this homotopy at hand, we can easily compute Gerstenhaber brackets of
elements of any degree, but for us here it will be convenient to use it only
when both elements do not have degree~$1$. The result we get is the
following.

\begin{proposition}
Let $(Q,I)$ be a  gentle presentation. Suppose that the quiver~$Q$ is not
the one with one vertex and one arrow. The $\circ_\phi$-compositon of the
representing cocycles of two elements in our generating set~$\G$ for the
algebra~$\HH^*(A)$ that are not of degree~$1$ is a coboundary and, in
particular, so is their Gerstenhaber bracket.
\end{proposition}

\begin{proof}
According to Proposition~\ref{prop:G}, 
the elements in our generating set for the algebra~$\HH^*(A)$ that are not of
degree~$1$ are the classes of the following elements
of~$\kk(\Gamma\parallel\cB)$:
\begin{itemize}

\item the pairs $(s(\alpha),\alpha)$ with $\alpha$ a $\cB$-maximal element
of~$\cB$;

\item the cocycles~$\cycle{\alpha}$ with $\alpha\in\Crepprim(\cB)$;

\item the pairs $(\gamma,\alpha)$ with $\gamma$ a $\Gamma$-maximal element
of~$\Gamma$ of length different from~$1$, and such that $\gamma$ and
$\alpha$ neither begin nor end with the same arrow;

\item the cocycles~$\cycle{C}$ with $C\in\Crepprim(\Gamma)$ of length
different from~$1$.

\end{itemize}
To check the statement of the proposition we will compute
$\circ_\phi$-compositions directly. Let us fix two non-negative
integers~$m$ and~$n$, both different from~$1$.

Let first $C$ be a $\Gamma$-complete cycle of length~$n$, and let
$f=(\gamma,\alpha)$ be an element of~$\Gamma_m\parallel\cB$. If $\eta$ is a
path in~$\Gamma_{m+n-1}$, then
  \begin{align*}
  \MoveEqLeft
  \bigl(f\circ_\phi(C,s(C))\bigr)(1\otimes\eta\otimes1) \\
  &= \sum_{\eta_1C\eta_2=\eta}
    (-1)^{n\abs{\eta_1}}
    f\Bigl(\phi\bigl(
        (1\otimes\eta_1\otimes s(C))\otimes(1\otimes\eta_2\otimes1)
        \bigr)\Bigr)
  = 0
  \end{align*}
because $\phi$ vanishes on all elementary tensors of the form $(1\otimes
a\otimes b)\otimes(1\otimes c\otimes 1)$ in~$\cR\otimes_A\cR$ with $b$ of
length~zero; the signs appearing here are Koszul's fault. It follows
immediately from this that $\circ_\phi$-multiplication on the right by the
cocycle~$\cycle{C}$ is identically zero.

Next, let $(\delta,\beta)$ be an element of~$\Gamma_n\parallel\cB$ with
$\Gamma$-maximal first component and $\delta$ and~$\beta$ neither beginning
nor ending with the same arrow. Let $f=(\gamma,\alpha)$ be an element
of~$\Gamma_m\parallel\cB$. If $\eta$ is a path in~$\Gamma_{m+n-1}$, then 
  \begin{align}
  \MoveEqLeft
  \bigl(f\circ_\phi(\delta,\beta)\bigr)(1\otimes\eta\otimes1)  \notag \\
     &= \sum_{\eta_1\delta\eta_2=\eta}
        (-1)^{n\abs{\eta_1}}
        f\Bigl(\phi\bigl(
            (1\otimes\eta_1\otimes\beta)\otimes(1\otimes\eta_2\otimes1)
            \bigr)\Bigr). \label{eq:some}
  \end{align}
Since $\delta$ is $\Gamma$-maximal, this sum has no terms unless
$\eta=\delta$, and in that case only one term in which both $\eta_1$
and~$\eta_2$ are of length~$0$. Therefore $\phi\bigl(
(1\otimes\eta_1\otimes\beta)\otimes(1\otimes\eta_2\otimes1) \bigr)$ is of
degree~$1$, which is not the degree of~$f$, and the sum~\eqref{eq:some} is
zero.

Finally, suppose that $n=0$, let $\alpha$ be a cycle in~$(Q,I$) such that
either $\alpha$ is $\cB$-maximal or $\alpha$ is a cocomplete primitive
cycle, and let $f=(\gamma,\beta)$ be  an element of~$\Gamma_m\parallel\cB$.
If $\eta$ is a path in~$\Gamma_{m+n-1}$, then this set is not empty, so
that $m\geq2$, and
  \begin{align}
  \MoveEqLeft \notag
  \bigl(f\circ_\phi(s(\alpha),\alpha)\bigr)(1\otimes\eta\otimes1)  \\
  &= \sum_{\substack{\eta_1\eta_2\eta_3=\eta\\\eta_2=s(\alpha)}}
  (-1)^{n\abs{\eta_1}}
  f\Bigl(\phi\bigl(
  (1\otimes\eta_1\otimes\alpha)\otimes(1\otimes\eta_3\otimes1)
  \bigr)\Bigr). \label{eq:last}
  \end{align}
Suppose there is a factorization~$\eta_1\eta_2\eta_3$ of~$\eta$
with~$\eta_2=s(\alpha)$ and such that the corresponding term of this sum is
not zero. 
\begin{itemize}

\item If both~$\eta_1$ and~$\eta_3$ have positive length, then the
definition of the map~$\phi$ tells us that $\alpha$ has length~$1$, so that
it is a loop in~$Q$, and that $\gamma=\eta_1\alpha\eta_3$: we thus have
that both $\eta_1\alpha\eta_3$ and $\eta_1\eta_3$ are elements of~$\Gamma$
of length at least~$2$, and the gentleness of~$(Q,I)$ implies
that $\eta_1$ and~$\eta_3$ are powers of~$\alpha$ and that $\alpha^2\in I$. Lemma~\ref{lemma:exception} tells us that this cannot occur, since we are
supposing that the quiver is not the one with one vertex and one arrow. 

\item On the other hand, if either $\eta_1$ or ~$\eta_3$ has length zero, then by
definition of $\phi$ and the gentleness of $(Q,I)$ we have that
$\gamma$ divides $\eta_1\alpha\eta_3$ when $\alpha$ is a cocomplete cycle.
Let $\delta$ be in $\cB$ such that $\delta \gamma$ or $\gamma\delta$ is
$\eta_1\alpha_1\eta_3$. There are three cases now: 
\begin{itemize}

\item if $\gamma=C$ is a
$\Gamma$-complete cycle then 
  \[
  \bigl(\cycle{C}\circ_\phi(s(\alpha),\alpha)\bigr)(1\otimes\eta\otimes1)
         =\delta+(-1)^{\abs{\eta_1}}\delta=0,
  \]
because $|\eta|$ is odd.;

\item in the second case, when $\gamma$ is
$\Gamma$-maximal and $\alpha$ is not a loop, we have 
  \[
  \bigl(f\circ_\phi(s(\alpha),\alpha)\bigr)(1\otimes\eta\otimes1)=0,
  \]
because $\delta\beta=\beta\delta=0$;

\item finally, if $\alpha$ is a loop  and
either the first arrow of $\gamma$ or the last one is $\alpha$, then
  \[
  \bigl(f\circ_\phi(s(\alpha),\alpha)\bigr)(1\otimes\eta\otimes1)
     =(-1)^{\abs{\eta_1}}\beta,
  \]
so $f\circ_\phi(s(\alpha),\alpha)=(\eta,\beta)$, that is a coboundary. 

\end{itemize}
\end{itemize}
\end{proof}

To complete the computation of the Lie bracket of~$\HH^*(A)$ we need to
deal with the brackets of elements of~$\HH^1(A)$ with other elements, and
for this we will use the approach of~\cite{MSA}. This is the content of the
next three propositions.

\bigskip

If $c$ is an arrow in~$Q$ and $\gamma=c_n\cdots c_1$ is a path, we set
  \[
  \deg_c(\gamma) \coloneqq \#\{i\in\{1,\dots,n\}:c_i=c\},
  \]
the number of times the path~$\gamma$ passes through~$c$, and if
$(\gamma,\alpha)$ is an element of~$\Gamma\parallel\cB$ we let 
  \[
  \deg_c(\gamma,\alpha) \coloneqq \deg_c(\alpha) - \deg_c(\gamma).
  \]

\medskip

\begin{proposition}\label{nonzerobracket}
If $(Q,I)$ is a gentle presentation, $c$ an arrow~$Q$, and $u$ the pair
$(c,c)\in\Gamma_1\parallel\cB$, then 
  \[
  [u,v] = \deg_c(v)\cdot v
  \]
for all $v\in\Gamma\parallel\cB$.
\end{proposition}

\begin{proof}
There is a unique derivation~$d:A\to A$ vanishing on the subalgebra~$E$
spanned by the vertices and such that on each path~$\gamma\in\cB$ takes
the value $d(\gamma) = \deg_c(\gamma) \gamma$. From~$d$ we construct  the
derivation $$d^e\coloneqq d\otimes1+1\otimes d:A^e\to A^e,$$ so that
$d:A\to A$ is now a $d^e$-operator on the left $A^e$-module~$A$. We now
define a $d^e$-lift $f_\bullet:\cR\to\cR$ of that $ d^e$-operator to the
$A^e$-projective resolution~$\cR$ of~$A$ such that for each $m\geq0$ and
each $\gamma\in\Gamma_m$ we have
  \[
  f_m(1\otimes\gamma\otimes 1) 
        = \deg_c(\gamma)\cdot 1\otimes\gamma\otimes 1.
  \]
Using this $d^e$-lift we can immediately compute the brackets that appear in
the statement of the proposition following the procedure of~\cite{MSA}.
\end{proof}

Given an element  $(c,\alpha)$  in $\Gamma_1\parallel\cB$ and $\beta$  in
$\cB$,  we write
  \[
  \beta^{c,\alpha} = \sum_{\beta_2c\beta_1=\beta}\beta_2\alpha\beta_1,
  \]
where the sum runs over all factorizations of~$\beta$ of the
form~$\beta_2c\beta_1$. Note that if the presentation~$(Q,I)$ is \fd gentle,
then there is in fact at most one such factorization.

\medskip

\begin{proposition}\label{zerobracket2}
Let (Q,I) be a gentle presentation and suppose that $Q$ is neither
the Kronecker quiver nor a loop without relations. Let $(c,\alpha)$ be an
element of~$\Gamma_1\parallel\cB$ where $c$ is $\Gamma$-maximal and $c$ is
neither the first nor the last arrow of $\alpha$. 
\begin{thmlist}

\item If $v$ is an element of the generating set~$\G$ of degree different
from~$1$, then the bracket $[(c,\alpha),v]$ is a coboundary in the
complex~$\kk(\Gamma\parallel\cB)$.

\item If $v$ is an element of~$\G$ of degree~$1$ such that 
\begin{itemize}
	\item it is either of the form $(d,\beta)$ where $d$ is a $\Gamma$-maximal arrow that is neither the first nor the last arrow of $\beta$,
	\item or it is of the form $(d,s(d))$ with
	$d\in\Crepprim_1(\Gamma)$
\end{itemize}
then the bracket $[(c,\alpha),v]$ is a coboundary.

\end{thmlist}
\end{proposition}

\medskip

We are excluding here the Kronecker quiver and the quiver with exactly one
vertex and one arrow because for them the second part of the proposition
does not hold. We will treat these special cases separately.

\medskip

\begin{proof}
Note that the path $\alpha$ has positive length  when $Q$ is not a loop without relations: 
if its length were zero, then $c$ would be a loop. Since $c$ is $\Gamma$-maximal, $c^2$ would not be in~$\Gamma_2$ and also $c$ is the only arrow by the $\Gamma$-maximality of~$c$.

There is a $\kk Q_0$-linear derivation~$D:A\to A$ 
such that for all arrows~$a\in Q_1$ we have
  \[
  D(a) = \begin{cases*} 
         \alpha & if $a=c$; \\
         0 & if not.
         \end{cases*}
  \]
If $D^e:A^e\to A^e$ is the corresponding derivation on the enveloping
algebra of~$A$, $r$~the length of~$\alpha$ and $\alpha=a_r\cdots a_1$, 
then there is a $D^e$-lift $f_\bullet:\cR\to\cR$ of the
$D^e$-operator $D:A\to A$ to the Bardzell resolution~$\cR$ 
such that for all $m\geq0$ and all $\gamma\in\Gamma_m$ we have
  \[
  f_m(1\otimes\gamma\otimes 1)
        = \begin{dcases*}
          \sum_{i=1}^r
                a_r\cdots a_{i+1}\otimes a_i\otimes a_{i-1}\cdots a_1
             & if $m=1$ and $\gamma=c$; \\
          0 & if not.
          \end{dcases*}
  \]
It follows from this using~\cite{MSA} that if $n\geq0$ and $(\delta,\beta)$
is an element of~$\Gamma_n\parallel\cB$ then
  \begin{equation}\label{eq:br}
  [(c,\alpha),(\delta,\beta)]
      = \begin{cases*}
          (\delta,\beta^{c,\alpha}) - (c,\alpha^{\delta,\beta })
                & if $n=1$; \\
          (\delta,\beta^{c,\alpha})
                & if not.
          \end{cases*}
  \end{equation}

To prove the first part of the proposition, let us now suppose that
$n\neq1$ and that $(\delta,\beta)\in\Gamma_n\parallel\cB$, and show that in
this case we have that
  \begin{equation} \label{eq:c-alpha}
  \claim{if $(\delta,\beta^{c,\alpha})$ is not zero, then $n\neq0$, the
  path~$\beta$ has length~$1$, and either $[(c,\alpha),(\delta,\beta)]$ is
  a coboundary or the path~$\delta$ is not $\Gamma$-maximal.}
  \end{equation}
To do that, let us suppose that $(\delta,\beta^{c,\alpha})\neq0$. The
arrow~$c$ then appears in the path~$\beta$, so that $\beta$ has positive
length. The presentation~$(Q,I)$ is  gentle and~$\alpha$ neither begins nor ends
 with the arrow~$c$,  thus $\beta$ must have length~$1$ and
be equal to~$c$. If~$n=0$, then this tells us that $c$ is a loop and by hypothesis $c^2\in I$. But this is
not possible, since $c$ is a $\Gamma$-maximal path. We thus have $n\geq2$. At
this point we have that  the paths $c$, $\delta$ and $\alpha$ are
parallel and that $[(c,\alpha),(\delta,\beta)]$ is $(\delta,\alpha)$.

Since $c$ is $\Gamma$-maximal and $\delta$ has length at
least~$2$, the arrow~$c$ does not appear in~$\delta$, and then the
gentleness of~$(Q,I)$ implies that $\delta$ and~$\alpha$ begin with the
same arrow and end with the same arrow. If $\alpha$ has length at
least~$2$, then the pair~$(\delta,\alpha)$ is a coboundary. If
instead~$\alpha$ has length~$1$, so that it is  just an arrow, then
the path~$\delta$ starts and ends with that arrow and, in particular, it is
not $\Gamma$-maximal. The claim~\eqref{eq:c-alpha} is thus proved. 

Using it we see at once that
$[(c,\alpha),v]$ is a coboundary whenever $v$ is an element of the set~$\G$
of degree different from~$1$. Indeed, such an element of~$\G$ has one  of
the following three forms. Either it is given by $(\delta,\beta)$ with $\delta$ a $\Gamma$-maximal path, or by  $\cycle{C}$ with $C\in\Crepprim(\Gamma)$, that is a
linear combination of pairs of~$\Gamma\parallel\cB$ with second component
of length~$0$, or it has degree~$0$ in cohomology. In each of these three
cases by~\eqref{eq:c-alpha} we have that $[(c,\alpha),v]$ is a coboundary.
This shows  part~\thmitem{1} of the proposition.

In order to prove part~\thmitem{2} now, let this time $(d,\beta)$ be an
element of~$\Gamma_1\parallel\cB$, and suppose that
$[(c,\alpha),(d,\beta)]=(d,\beta^{c,\alpha})-(c,\alpha^{d,\beta})\neq0$. We
consider two cases.
\begin{itemize}

\item Suppose first that $d$ is $\Gamma$-maximal and that $d$ is neither the first nor the last arrow of~$\beta$. By symmetry, we can assume that 
$\beta^{c,\alpha}\neq0$. The arrow~$c$ then appears in
the path~$\beta$ and, since $(Q,I)$ is gentle and by the fact that $c$ is neither the first nor the last arrow of $\alpha$, we see that
$\beta$ has length~$1$. Therefore $\beta=c$ and thus the arrows~$c$
and~$d$ are parallel. As both $c$ and $d$ are $\Gamma$-maximal, gentleness implies 
that $Q$ is a Kronecker quiver, contrary to
our hypothesis. This case does therefore not occur.

\item Next, suppose that $d\in\Crepprim_1(\Gamma)$ and that $\beta=s(d)$.
As  $c$ does not appear in~$\beta$, we have that
$[(c,\alpha),(d,\beta)]=-(c,\alpha^{d,\beta})\neq0$. Therefore the
path~$\alpha$, which is in~$\cB$, goes through the loop~$d$ and
$\alpha^{d,\beta}\neq0$. Since the presentation~$(Q,I)$ is  gentle and  since $\alpha^{d,\beta}\neq0$, 
we have that $d$ is either the first 
or the last arrow in~$\alpha$.

Let us consider the case in which $\alpha$ ends with~$d$, the other case being
 similar. As $d$ is a loop but $c$ is not a loop, gentleness implies 
that there is a cycle~$\gamma$ (possibly of length zero)  starting
and ending at~$s(c)$ such that $\alpha=dc\gamma$. Therefore
$(c,\alpha^{d,\beta})=(c,c\gamma)$. If $\gamma$ has positive length, then
 gentleness implies that the only arrow with target $s(c)$ is the last arrow
of~$\gamma$.  Then  the pair $(c,c\gamma)$ is the coboundary of
$(s(c),\gamma)$. If instead $\gamma$ has length zero, then using  gentleness
and the $\Gamma$-maximality of~$c$ we can see that $c$ and $d$ are the only
arrows incident with the vertex~$t(c)$. This implies  that the coboundary of
$(t(c),t(c))$ is $(c,c)$. In all cases, the bracket $[(c,\alpha),(d,\beta)]$
is a coboundary.

\end{itemize}
The proposition is proved.
\end{proof}

We have only one more computation left.

\begin{proposition}\label{zerobracket3}
Let (Q,I) be a gentle presentation. Suppose that the quiver~$Q$ is not the one with one vertex and one arrow.
Let $c$ be an arrow  in~$\Crepprim_1(\Gamma)$. 
\begin{thmlist}

\item If $v$ is an element of the generating set~$\G$ of degree different
from~$1$, then $[(c,s(c)),v]$ is a coboundary.

\item If $f$ is an element of~$\Crepprim_1(\Gamma)$, then
$[(c,s(c)),(f,s(f))]=0$.

\end{thmlist}
\end{proposition}

\begin{proof}
Note that since $c$ is a path of length~$1$ that belongs
to~$\Crepprim(\Gamma)$, the characteristic of the ground field~$\kk$
is~$2$. There is a unique $(\kk Q_0)^e$-linear derivation~$D:A=\kk Q/I\to A$ such
that $D(c)=s(c)$ and $D(a)=0$ for all arrows~$a$ in~$Q$ different from~$c$.
From~$ D$ we can construct the derivation $D^e\coloneqq D\otimes1+1\otimes
D:A^e\to A^e$, and a $D^e$-lift $f_\bullet:\cR\to\cR$ of the $D^e$-operator
$D:A\to A$. It is easy to check that such a lift exists and that 
$f_n(1\otimes\gamma\otimes1)=0$, for all $n\geq0$ and all
$\gamma\in\Gamma_n$. Using~\cite{MSA}, it follows 
that, for all $(h,\beta)\in\Gamma\parallel\cB$, we have 
  \[
  [(c,s(c)),(h,\beta)] = (h,\beta^{c,s(c)}).
  \]
If $\beta$ has length zero, then this is zero. This implies that
$[(c,s(c)),\cycle{C}]=0$ for all $C\in\Crepprim(\Gamma)$ and
part~\thmitem{2} of the proposition holds.

Suppose now that $\beta$ has positive length, that $(h,\beta)$ does not
have degree~$1$ ---so that $h$ does not have length~$1$--- and that
$\beta^{c,s(c)}\not\in I$. In either case the loop~$c$ is in~$\beta$. By the gentleness, $\beta=\beta_2c\beta_1$ with one of~$\beta_1$ or~$\beta_2$
of length zero.  Indeed, there are two paths $\beta_1$ and~$\beta_2$ such
that $\beta=\beta_2c\beta_1$ and $\beta_2\beta_1\in \cB$. Let us consider,
for example, the case in which $\beta_1$ has length zero. We have three
cases.
\begin{itemize}

\item Suppose first that $\beta$ is $\cB$-maximal and so $h=s(\beta)$. By
 gentleness $\beta=c$ and $Q$ is a loop, which contradicts the
hypotheses. 

\item Next, suppose that $\beta$ is a primitive cocomplete cycle of period
$r$ and that $h=s(\beta)$. So $\beta^{c,s(c)}=\beta_2$ is a cycle and
$s(\beta_2)=s(c)$. Thus \[[(c,s(c)),\cycle{\beta}]=(s(c),
\beta_2)+(s(c), \beta_2)=0.\]

\item Finally, if $h$ is $\Gamma$-maximal, then
$(h,\beta^{c,s(c)})=(h,\beta_2)$. Since the presentation $(Q,I)$ is 
gentle and $Q$ is not a loop,  $h$ and $\beta_2$ begin with the same arrow.
Indeed, if $\beta_2$ is a vertex then $h$ is a cycle and $s(h)=s(c)$, so
either $h$ is a power of $c$ or $h$ is a complete cycle. That is not possible.

\end{itemize}
\end{proof}

Now, we treat the excluded cases. 

\begin{remark}
Suppose that the quiver $Q$ has exactly one vertex and one arrow~$a$. Using
Remark~\ref{rem:exception} we obtain:
\begin{itemize}

\item if $a^2\in I$ and the characteristic of~$\kk$ is not~$2$, then
$\mathscr{G}=\{(s(a),a),(a,a)\}$ and the only non-zero bracket between
elements of~$\mathscr{G}$ is
  \[
  [(a,a),(s(a),a)] = (s(a),a),
  \]
        
\item if either $a^2\notin I$ or $a^2\in I$ and $\kk$ is of characteristic
~$2$, then the generating set of~$\HH^*(A)$ is
$\mathscr{G}=\{(s(a),a),(a,s(a))\}$, and the only non-zero bracket between
elements of that set is 
  \[ 
  [(a, s(a)), (s(a),a)] = (s(a),s(a)).
  \]

\end{itemize}

Suppose now that $Q$ is the Kronecker quiver with arrows $a$ and $b$. Then
$$\mathscr{G}=\{(a,a), (a,b),(b,a)\}$$ and the non-zero brackets are
\begin{align*}
& [(a,a),(a,b)]=-(a,b),
& [(a,a),(b,a)]=(b,a),
& & [(b,a),(a,b)]=2(a,a).
\end{align*}	
\end{remark}

 We can summarize the previous results as follows.
 \begin{theorem}\label{thm:HH1Lie}
 Suppose that $Q$ is not the quiver with one loop and not the Kronecker quiver.
 \begin{enumerate}
\item If $(c,c)\in \HH^1(A)$ corresponds to an arrow $c$ in the complement of the spanning tree $T$ and $v\in \HH^n(A)$, then $[(c,c),v]=\deg_c(v)v$.
\item All other brackets amongst elements in~$\G$ are zero. 
\end{enumerate} 
Moreover, in (1) if $v$ is also in $\HH^1(A)$, then $\deg_c(v)$ is always $0$ or $1$, hence the Lie algebra structure of $\HH^1(A)$ does not  depend on 
$\charact(\kk)$.
\end{theorem}
\begin{proof}
The theorem follows from Propositions  \ref{nonzerobracket}, \ref{zerobracket2} and
\ref{zerobracket3}.
\end{proof}

\section{The shifted Hochschild cohomology as a graded Lie algebra}
In this section, we will describe the Lie algebra structure of the Hochschild cohomology of a gentle algebra and exhibit some consequences of the results obtained when  
computing the Gerstenhaber bracket.

The Hochschild cohomology $\HH^*(A)$, for a finite dimensional algebra $A$, is a Lie algebra with Lie bracket given by the Gerstenhaber bracket. However, this bracket is not graded with respect to the cohomological degree. In order to obtain a graded Lie algebra, we need to consider ${\mathfrak g} = \HH^*(A)[1]$, given by the shifted cohomology spaces $\HH^n(A)[1] = \HH^{n+1}(A)$. The  following result on the graded Lie algebra structure of the Hochschild cohomology of a gentle algebra then follows from our computations above. 

For the following theorem, we fix a spanning tree $T$ of $Q$. 

\begin{theorem}\label{thm:Lie-cohomology}
Let $(Q,I)$ be a gentle presentation and $A = \kk Q/I$ where $Q$ is not a loop and not the Kronecker quiver and suppose the spanning tree $T$ has been chosen so that condition ($\star$) of page~\pageref{eq:star} is satisfied. Then there is a Lie algebra isomorphism 
\[{\mathfrak g} \cong \mathfrak{l}_1 \times \ldots \times \mathfrak{l}_n\]
where $n$ is the number of cycles of the graph underlying the quiver and for $K_i$ such a cycle, ${\mathfrak l}_i$ is a Lie subalgebra of $\mathfrak g$ belonging to one of the following four families. 
\begin{enumerate}
\item Suppose $K_i$ corresponds to an element of the form $(\gamma, \alpha) \in \kk(\Gamma || \cB)$ where either $\gamma$ is a $\Gamma$-maximal path and $\gamma$ and $\alpha$ do not start or end with the same arrow or $\gamma = s(\alpha)$ and $\alpha$ is a $\cB$-maximal path. Then ${\mathfrak l}_i$ is generated by the class $(\gamma, \alpha)$ and the classes of  elements of the form $(a,a)$ where $a$ belongs either to $\gamma$ or to $\alpha$ and $a$ is in the complement of the spanning tree. In particular, $${\mathfrak l}_i \simeq \langle (a,a) \rangle \ltimes \langle (\gamma, \alpha) \rangle.$$ 
\item Suppose $K_i$ corresponds to an element of the form $(\gamma, \alpha)$  where $\gamma$ is a complete cycle and $\alpha$ is a vertex in $\gamma$. Then  ${\mathfrak l}_i$ is generated by the non-zero classes $\cycle{\gamma}$ and $(a \gamma^k, a)$, for $k \in \mathbb Z_{>0}$, and the class of an element $(a,a)$ where  $a$ is in the complement of the spanning tree and where without loss of generality we assume that $a$ is the first arrow of $\gamma$. In particular, 
 \begin{align*} {\mathfrak l}_i & \simeq \langle (a,a) \rangle \ltimes  \langle \cycle{\gamma},  (a\gamma^k, a), k \geq 1 \rangle \\ & \simeq \langle (a,a) \rangle \ltimes ( \langle (a\gamma^k, a), k \geq 1  \rangle \ltimes \langle \cycle{\gamma^k},  k \geq 1\rangle).\end{align*} 
\item Suppose $K_i$ corresponds to an element of the form $(\gamma, \alpha)$  where  $\alpha$ is a cocomplete cycle and $\gamma$ a vertex in $\alpha$. Then  ${\mathfrak l}_i$ is generated by the non-zero classes $\cycle{\alpha^k}$  and $(a , a\alpha^k)$, for $k \in \mathbb Z_{>0}$, and the class of an element $(a,a)$ where $a$ is in the complement of the spanning tree and where without loss of generality we assume that $a$ is the first arrow of $\alpha$. In particular, \begin{align*} {\mathfrak l}_i & \simeq \langle (a,a) \rangle \ltimes  \langle \cycle{\alpha},  (a\alpha^k, a), k \geq 1 \rangle \\ & \simeq \langle (a,a) \rangle \ltimes ( \langle (a\alpha^k, a), k \geq 1  \rangle \ltimes \langle \cycle{\alpha^k},  k \geq 1\rangle).\end{align*} 
\item Suppose $K_i$ corresponds to none of the above then it only gives rise to an element of the form $(a,a)$ where $a$ is an arrow in $K_i$ which is in the complement of the spanning tree. In particular, it is abelian and we have
$$ {\mathfrak l}_i \simeq \langle (a,a) \rangle.$$ 
\end{enumerate}

\end{theorem}


We now consider the connection of the dimensions of the Hochschild cohomology spaces and the AAG derived invariant defined in \cite{AAG}.

Let us first suppose that $\charact(\kk)\neq 2$. We know from 
\cite{Ladkani1} and \cite{Redondo-Roman}  that given a \fd gentle presentation $(Q,I)$ of $A$, \sloppy $\dim \HH^1(A)=
1- \chi(Q)+ \phi_A(1,1)$, where $\chi(Q)= |Q_0|- |Q_1|$ and given non-negative intergers $i,j$, we denote $\phi_A(i,j)$ the  AAG-invariant corresponding to the pair $(i,j)$ as defined in \cite{AAG}. It follows from our computations that If $\charact(\kk)= 2$, then the dimension of $ \HH^1(A)$ is  $1- \chi(Q)+ \phi_A(1,1) + \phi_A(0,1)$.

\begin{proposition}
Let $(Q,I)$ be a \fd gentle presentation and $A = \kk Q/I$ the associated algebra,  
 \begin{itemize}
\item if $\charact(\kk)\neq 2$, then we have 
\[\phi_A(1,1)= \dim [\HH^1(A), \HH^1(A)],\] 
so $\phi_A(1,1)$ equals the dimension of the derived Lie algebra $\HH^1(A)'$ of $\HH^1(A)$;
\item if $\charact(\kk)= 2$, then we have
\[\phi_A(1,1) + \phi_A(0,1)= \dim [\HH^1(A), \HH^1(A)].\] 
\end{itemize} 
\end{proposition}
\begin{proof}
For $A$ \fd gentle, the AAG-invariant $\phi_A(1,1)$ counts the number of classes of pairs $(a, \alpha)$, where $a$ is an arrow and $\alpha$ is a non zero path in $A$
of positive length, which is maximal, while $\phi_A(0,1)$ counts the number of classes of pairs $(b, s(b))$, where $b$ is a loop with source $s(b)$ and $b^2\in I$. As we have just computed, these are exactly those elements in~$\G$ of cohomological degree $1$ belonging to $\HH^1(A)'$ - notice that the pairs $(b, s(b))$ only appear in case the characteristic is $2$.
\end{proof}

Thus, the previous  proposition provides a good candidate for $\phi_A(1,1)$ when $A$ is infinite  dimensional gentle. Namely, set $$\phi_A(1,1)= \dim [\HH^1(A), \HH^1(A)].$$

\begin{corollary} Suppose $\charact (\kk)\neq 2$.
The number of arrows in $Q$ is a derived invariant of the algebra $A$.
\end{corollary}
\begin{proof}
We have already recalled that \sloppy
$$\dim \HH^1(A)= 1- (|Q_0|- |Q_1|) + \dim \HH^1(A)'$$ and \sloppy that  $\phi_A(1,1)= \dim \HH^1(A)'$ in the \fd gentle case. 
It is well known that $\dim \HH^1(A)$ and
$\dim \HH^1(A)'$ are derived invariants. Moreover, $|Q_0|$ is also a derived invariant of a gentle algebra, since it is the number of simple $A$-modules. In case $A$ is infinite dimensional gentle, $|Q_0|$ is also a derived invariant using $K$-theory arguments.
It follows from $\dim \HH^1(A)= 1- (|Q_0|- |Q_1|) +\dim \HH^1(A)'$ that  $|Q_1|$ is also a derived invariant of $A$.
\end{proof}

Now suppose that $\charact (\kk)$ is an odd integer $p$. Let $\cycle{C}$ be a primitive cycle of length $m$ that belongs to the basis of $\HH^n(A)$. Suppose first that $n$ is even and $m=n$. Let $c$ be an arrow in the complement of a spanning tree such that $c\in C$. We know that
\[[(c,c), \cycle{ C}]= \deg_c(\cycle{C})\cdot \cycle{C}= n\cdot \cycle{C}.\]
Now, if $p$ divides $n$, then $[(c,c), \cycle{C}]=0$. If not, then $p$ and $n$ are coprime and $[(c,c), \cycle{ C^p}]=0$.
In case $n$ is odd, the cycle $C$ is already a square, so $\deg_c( \cycle{C })=2n$. 
These comments can be summarized in the following corollary.

\begin{corollary}
Suppose that $\kk$ is a field of odd characteristic and that $(Q,I)$ is a finite dimensional gentle presentation of the algebra $A$. Let $(c,c)$ be such that $c$ is an arrow in the complement of the spanning tree $T$ and $\cycle{ C }$ be an element in the basis 
of $\HH^n(A)$ corresponding to a basic complete cycle.
Then  
\[[(c,c), \cycle{ C }]= \deg_c(\cycle{ C })\cdot \cycle{ C }.\]
From this, using different powers of the cycle and the fact that the Gerstenhaber bracket $[(c,c), -]$ is a
derivation with respect to the cup product, we can deduce the characteristic of $\kk$.  
 \end{corollary}

 \begin{proof}
Let $\cycle{ C }$ be a basic complete cycle of length $m$ corresponding to an element in the basis 
of $\HH^n(A)$. We recall that  $n=m $ if $m$ is even  or $n =2m$ otherwise.
Suppose further that $c$ is an arrow in the complement of the spanning tree $T$ of $Q$ and that $c$ is an arrow in $C$. Then
 \[[(c,c), \cycle{ C^w }]= \deg_c(\cycle{ C^w })\cdot \cycle{ C^w }\]
 where $\deg_c(\cycle{ C^w }) = w$ where $w$ is equal to 1 if $m$ is even  and  $w$ is equal to $2$ if $m$ is odd. Moreover, 
 \[[(c,c), \cycle{ C^{wr} }]= \deg_c(\cycle{ C^{wr} })\cdot \cycle{ C^{wr}}\]
 where $\deg_c( \cycle{ C^{wr} })$ is equal to $r$  if $m$ is even or it is equal to $2r$ if $m$ is odd. 
 Now let $p$ be the smallest positive integer such that $[(c,c), \cycle{ C^{wp}}] =0$. Then the characteristic of $\kk$ is $p$.
\end{proof}




\chapter{Geometric surface interpretation of Hochschild (co)homology }
\label{chapter:geometric}

In this section we introduce the geometric model for 
gentle algebras following the ideas in \cite{HKK, LP, OPS, PPP} based on the
ribbon graph associated to a gentle algebra in \cite{S}. 

A \newterm{ribbon graph}\index{Ribbon graph} is  a finite, undirected graph with a
cyclic ordering on the set of edges adjacent to each vertex. A
\newterm{marked ribbon graph}\index{Marked ribbon graph} is a ribbon graph on which we have
additionally chosen, at some of the vertices , a pair of cyclically consecutive edges
incident to it. We make the convention that whenever we draw (a part of) a
ribbon graph in the plane, the cyclic ordering of the edges at an embedded
vertex is implicitly the clockwise one. 

If $(Q,I)$ is a gentle
presentation, then we 
construct  the marked ribbon graph $G$ associated to~$(Q,I)$ as follows:
\begin{itemize}

\item The edges of~$G$ are the vertices of~$Q$.

\item The vertices of~$G$ are 
\begin{itemize}

\item the  maximal non-zero paths in~$(Q,I)$ of finite length, that is, those paths~$q$
in~$Q$ such that for all arrows $a\in Q_1$ we have $aq$,~$qa\in I$, 

\item the trivial paths in~$Q$ corresponding to vertices~$v\in Q_0$ through
which passes exactly one maximal non-zero path of~$(Q,I)$, and

\item the infinite paths which are powers of a  cycle in $(Q,I)$.

\end{itemize}

\item An edge~$v$ of~$G$ is incident to a vertex~$q$ of~$G$ if and only if
the path~$q$ passes through~$v$  in Q.

\item The cyclic ordering of the edges adjacent to a vertex~$q$ of~$G$ is given by, 
 \begin{itemize}
\item if $q$ is of finite length, the cyclic closure of the linear ordering on the edges determined by the order
in which~$q$  viewed as a path in $Q$ visits them, and the marked pair of cyclically adjacent edges
adjacent to~$q$ is precisely the one closing the `cycle',
\item if $q$ is of infinite length, the cyclic order is determined by the order in which the cyclic path generating $q$ visits them. We note that such a vertex does not have any associated marked pair of edges.
\end{itemize}

\end{itemize}
Note that this makes sense: an edge of~$G$ is incident to exactly two
vertices of~$G$ precisely because the presentation $(Q,I)$ is   gentle.

From a marked ribbon graph~$G$ which is not the trivial graph with one
vertex and no edges we can construct a \newterm{marked ribbon
surface}\index{Marked ribbon surface} $S_G$ by a process of glueing as
follows. If $q$ is a vertex of~$G$ of valency~$n$, then  if $q$ corresponds to a path of
finite length the neighborhood of~$q$ in $G$ is as on the left side of Figure~\ref{finite
paths}.
 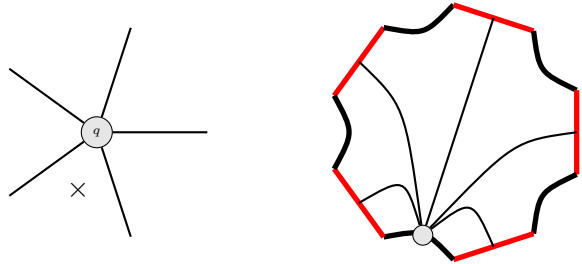
\begin{figure}[H]
  \[
  \begin{tikzpicture}[auto, thick, scale=0.8, font=\tiny,
        baseline={(0,0)}
        ]
  \node[cg] (0) at (0,0) {$q$};
  \foreach \i/\n in {2/a, 1/b, 0/c, -1/d, -2/x}
    {
    \node (x\i) at (\i*360/5:2) {};
    \draw (0) to (x\i);
    \ifnum\i>-2
    \else
      \node at (\i*360/5+360/5/2:1) {\large$\times$};
    \fi
    }
  \end{tikzpicture}
  \qquad\qquad
  \begin{tikzpicture}[auto, thick, scale=0.8, font=\tiny,
        baseline={(0,0)}
        ]
\foreach \i/\n in {2/a, 1/b, 0/c, -1/d, -2/x}
    {
    \coordinate (x\i) at (\i*360/5:2);
    \draw[red!100!black, line width=2pt] 
        ($(x\i)-(\i*360/5-90:0.7)$) 
        -- ($(x\i)+(\i*360/5-90:0.7)$) coordinate (q);
    \draw[line width=2pt] (q)
                .. controls (\i*360/5-360/5/2:1.7) 
                .. ($(\i*360/5-360/5:2)-(\i*360/5-360/5-90:0.7)$)
                coordinate (w)
                ;
    }
  \node[cg] (0) at (4*360/5-360/5/2:1.8) {};
  \foreach \i/\n in {2/a, 1/b, 0/c, -1/d, -2/x}
    {
    \draw[black, thick] (0) 
        .. controls ($(x\i)!0.5!(0) + (2*360/5-360/5:0.75)$)
        .. (x\i);
    }
  \end{tikzpicture} \]
   \caption{Neighborhood of {\newsib a vertex $q$ in the ribbon graph $G$ corresponding to} a finite path  and the corresponding polygon $P_q$. }
  \label{finite paths}
  \end{figure}

   If $q$ corresponds to a path of infinite length, then the neighbourhood of $q$ in $G$ is as on the left side of Figure~\ref{infinite paths}.
	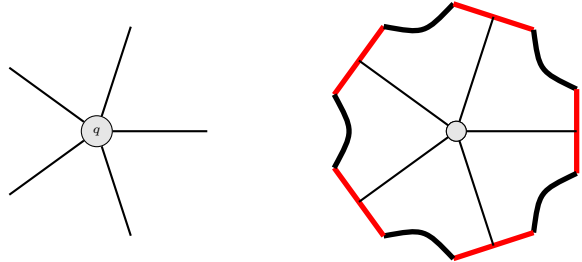
\begin{figure}[H]
	\[
	\begin{tikzpicture}[auto, thick, scale=0.8, font=\tiny,
	baseline={(0,0)}
	]
	\node[cg] (0) at (0,0) {$q$};
	\foreach \i/\n in {2/a, 1/b, 0/c, -1/d, -2/x}
	{
		\node (x\i) at (\i*360/5:2) {};
		\draw (0) to (x\i);
	}
	\end{tikzpicture}
	\qquad\qquad
	\begin{tikzpicture}[auto, thick, scale=0.8, font=\tiny,
	baseline={(0,0)}
	]
	\foreach \i/\n in {2/a, 1/b, 0/c, -1/d, -2/x}
	{
		\coordinate (x\i) at (\i*360/5:2);
		\draw[red!100!black, line width=2pt] 
		($(x\i)-(\i*360/5-90:0.7)$) 
		-- ($(x\i)+(\i*360/5-90:0.7)$) coordinate (q);
		\draw[line width=2pt] (q)
		.. controls (\i*360/5-360/5/2:1.7) 
		.. ($(\i*360/5-360/5:2)-(\i*360/5-360/5-90:0.7)$)
		coordinate (w)
		;
	}
	\node[cg] (0) at (0,0) {};
	\foreach \i/\n in {2/a, 1/b, 0/c, -1/d, -2/x}
	{
		\draw[black, thick] (0,0) 
		.. controls (0,0)
		.. (x\i);
	}
	\node[cg] (0) at (0,0) {};
	\end{tikzpicture} \]
	\caption{Neighborhood of  {\newsib a vertex $q$ in the ribbon graph $G$ corresponding to} an infinite path $q$ and the corresponding polygon $P_q$. }
	\label{infinite paths}
	\end{figure}
  
We then define a $2n$-gon $P_q$  as follows:  the $2n$ sides of $P_q$ are given by alternating black and red line segments where the $n$ red line segments are in one to one correspondence with the edges of $G$ incident to~$q$. ordered in the cyclic ordering at $q$.
If $q$ corresponds to a path of finite length in $(Q,I)$, then we add a marked point in the middle of the  (black) line segment of $P_q$ corresponding to the marking at~$q$. We then embed the
neighborhood of~$q$ in~$G$ as on the right side of Figure~\ref{finite paths}, connecting that
marked point to the middle points of the sides of~$P_q$ which correspond to
edges of~$G$. If $q$ corresponds to a path of infinite length in $(Q,I)$ then we embed the neighborhood of $q$ in $G$ as on  the right in Figure~\ref{infinite paths}. We construct the disjoint union $\bigsqcup_{q\in G}P_q$ and
identify, for each edge $v$ of~$G$ connecting vertices~$p$ and~$q$, the
sides of~$P_p$ and~$P_q$ corresponding to~$v$ with opposite orientations,
so that the resulting surface~$S$ 
is oriented. Clearly, this construction also gives us an embedding of the
graph~$G$ into~$S$  in such a way that the faces of $G$ correspond to the boundary components of $S$ and we refer to $S$ as the ribbon surface of $G$.

The point of these constructions is that in
\cite{OPS, PPP} it is shown that the isomorphism classes of  gentle algebras are in bijection with surface dissections as described above. Moreover,  it is shown in \cite{OPS}  that starting with a gentle  presentation~$(Q,I)$, if we
construct the corresponding marked ribbon graph~$G$ and from it the
corresponding ribbon surface $S$, with an embedding
of~$G$, then the indecomposable objects in the bounded derived category
$D^b(\rsmod{A})$ of the gentle algebra~$A=\kk Q/I$ are in bijection with
certain  homotopy classes of curves on~$S$  together with an indecomposable local system. For more information on this, the
papers~\cite{LP} and~\cite{HKK}  where  similar constructions in the context of partially wrapped Fukaya category of surfaces with stops and graded gentle algebras are considered, are also relevant. 

 Furthermore, in \cite{APS}, for every curve in the surface associated to a gentle algebra, a combinatorial winding number  is defined, see also \cite{LP} for another definition of a combinatorial winding number which coincides with the one in \cite{APS} on the set of closed curves. In order to define this winding number, we first define the dual graph $G^*$  of $G$ embedded in $S$ as follows. The vertices of $G^*$ correspond to marked points on the surface, placing one such marked point between any two marked points corresponding to vertices of $G$. Furthermore, if there is a boundary component with no marked points on it, it is sometimes convenient to replace it by a vertex of $G^*$ and treat it as an endpoint compactification of a puncture. The edges of $G$ are in bijection with the edges of $G^*$ where the bijection is given by sending an edge $v$ in $G$, to the unique edge of $G^*$ connecting two new marked points (corresponding to the vertices of $G^*$) and crossing $v$. In what follows, we will refer to the vertices of $G$ as marked points in the surface and if we need to refer to the marked points on $S$ corresponding to the vertices of $G^*$, we will explicitly say so. 

 We consider both open and closed curves in the surface. A \newterm{curve} in $S$ is a continuous map $\ssC:[0,1] \to S$.  We say that the 
curve is \newterm{closed} if $\ssC(0)= \ssC(1)$ and if $C(x) \in S \setminus \partial S$, for all $x \in [0,1]$, and that it is \newterm{open} otherwise. We only consider certain open curves, namely those such that $\ssC(0)$ and 
$\ssC(1)$ correspond to marked points associated to vertices of $G$ or to vertices of $G^*$ in the interior of $S$.  We note that when speaking about a curve, we usually identify it with its image in $S$. Furthermore, we always assume that the curves which we are considering are in minimal position with regards to  $G$ and $G^*$. 

With this we can define the combinatorial winding number 
$w(\ssC)$ of a curve $\ssC$. 
We first define, for all $n \in \mathbb Z$,  a \newterm{grading} on a curve $\ssC$ given by the function $f_n:  \ssC  \cap G^* \to \mathbb{Z}$, where the set $\ssC  \cap G^*$ is ordered by the direction of $\ssC$ defined as follows. 
Let $x_1, x_2 \in [0,1]$ be such that $\ssC(x_1) $ and $\ssC(x_2)$ are the  first and second crossing of $\ssC$ with $G^*$. Then by 
construction of $G^*$ the corresponding edges of $G^*$ are edges of a polygon with exactly one boundary segment $\ell$. Now the grading $f_n$  of $\ssC$ is defined by assigning the integer $n$ to $\ssC(x_1)$, that is $f_n(\ssC(x_1))=n$  and if when travelling along $\ssC$ from  
$\ssC(x_1)$ to $\ssC(x_2)$ the boundary segment $\ell$ lies to the right of $\ssC$ then $f_n(\ssC(x_2))=n+1$ and if $\ell$ lies to the left of $\ssC$ we set $f(\ssC(x_2))=n-1$. Propagating this along the whole of $\ssC$, defines the grading $f_n$. 

Given a curve $\ssC$, the \newterm{combinatorial winding number} $w(\ssC)$ is defined to be \sloppy $w(\ssC) =  f_n(\ssC(x_m))-n$ where for   $x_1, \ldots, x_m \in [0,1]$ the ordered set $\ssC  \cap G^*$ is given by $ \{\ssC(x_1), \ldots, \ssC(x_m)\}$ and where if $\ssC$ is a closed curve,  we fix $x_1$ to be any one of the intersections of  $\ssC$ and $G^*$ and let $x_m$, be the first $i \neq 1$ such that $\ssC (x_i) = \ssC(x_1)$. In particular, we note that $w(\ssC)$ is independent of $n$. 

 For every connected boundary component $B$ in $S$, we write $\ssC_B$ for a representative (in minimal position) of the homotopy class of curves homotopic to $B$. If the surface $S$ arises from a gentle algebra then we consider the vertices of $G$ in the interior of $S$ as (endpoint compactifications of) punctures. To distinguish the punctures corresponding to vertices of $G$ from the punctures corresponding to vertices of  $G^*$ arising from unmarked boundary components, we will call such a puncture a $G$-puncture.  For a $G$-puncture $P$ we denote by $\ssC_P$ the corresponding primitive closed curve and we set $w(\ssC_P) = 0$.  We note that this convention agrees with the winding number in terms of a line field associated to $(S,M)$  in \cite{HKK}, see also \cite{APS, LP}.   By construction of $(S,M)$, every curve $\ssC$ corresponds to a path in $\kk Q/I$, which we will denote by $p_\ssC$.   Furthermore,  for $(\gamma, \alpha)$ in $\kk (\Gamma\parallel\cB)$,    the winding number of the curve $\ssC$ associated to the path $p_\ssC = \gamma \alpha^{-1}$ is equal to $|\gamma| - |\alpha|$, that is, the negative of  the weight of the pair  $(\gamma, \alpha)$.

 \section{Geometric interpretation: cohomology, cup product and bracket} 
 Let $(Q,I)$ be a gentle presentation and $A = \kk Q/I$. Let $\G$ be the generating set of $\HH^*(A)$ described in Proposition~\ref{prop:G} in terms of cocycles in  $\kk(\Gamma\parallel\cB)$ and let $\cF \subset \G$ be a fixed set of derivations arising from the complement of a spanning tree of $Q$. Then we have the following result.

%
%
%
%
%

\begin{theorem}\label{thm:generators in gentle surface}
Let $(Q,I)$ be a gentle presentation and $A = \kk Q/I$.  Let $(S,M)$ be the marked surface induced by the ribbon graph $G$ of $A$. Then there is a one to one correspondence between the  generators of $\HH^*(A)$ in  $\G \setminus \cF$   and the set of  boundary components  with 0 or 1 marked point and $G$-punctures in $(S,M)$, that is punctures corresponding to vertices of $G$ in the interior of $S$. 

\end{theorem}

\begin{proof}
We begin by showing that every generator $g$ of $\HH^*(A)$ in $\cG \setminus \cF$ as described in  Theorem~\ref{thm:cohomology:basis} can naturally be associated to either a boundary component with 0 or 1 marked point or a $G$-puncture.

Suppose that $g = (s(\alpha), \alpha) \in \HH^0(A)$ with $\alpha$ a $\cB$-maximal path. Then $\alpha = \alpha_m ....\alpha_1 $ is a cycle such that $\alpha_1 \alpha_m \in I$. Since $\alpha$ is $\cB$-maximal, no other arrow starts and ends at $s(\alpha)$. Since $A$ is connected, $G$ is connected and this implies that locally in  $S$, we have the following configuration


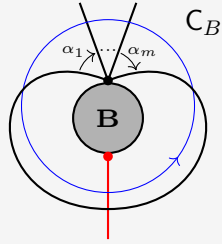
\begin{figure}[H]
	\[
	\begin{tikzpicture}[auto, thick, scale=1, font=\tiny,baseline={(0,0)}]
	
	\node[B] (0) at (0,0) {\textbf{B}};
	\node[b] (b) at ($(0)+(90:0.5)$) {};
	\node[r] (r) at ($(0)+(270:0.51)$) {};
	\node[font=\large] () at ($(0)+(45:1.8)$){$\ssC_B$};
	
	\draw (b) to node {} ($(b)+(70:1.1)$);
	\draw (b) to node {} ($(b)+(110:1.1)$);
	\draw[red] (r) to node {} ($(r)+(270:1.1)$);
	
	\draw[blue, thin, decoration={markings, mark=at position -0.1 with {\arrow{>}}},
		postaction={decorate}
		]
		($(0)+(90:0.16)$) circle (1.15);
	
	\draw[-] (b) to [out=20, in=0, looseness=2.6] node   {} ($(r)+(270:0.7)$);
	\draw[-] (b) to [out=160, in=180, looseness=2.6] node   {} ($(r)+(270:0.7)$);
	
	\node () at ($(b)+(140:0.6)$){$\alpha_1$};
	\draw[->, thin] (b)+(160:.4cm) arc (160:115:.4cm);
	
	\node () at ($(b)+(90:0.4)$){$...$};
	
	\node () at ($(b)+(40:0.6)$){$\alpha_m$};
	\draw[->, thin] (b)+(65:.4cm) arc (65:20:.4cm);
	
	\end{tikzpicture}
	\]
\end{figure}

It then follows from the definitions that $g$ corresponds to the boundary curve $\ssC_B$ with $w(\ssC_B) = -1$ and $p_{\ssC_B} = \alpha$. 

Suppose that $g=\cycle{\alpha} \in \HH^0(A)$ with $\alpha\in\Crepprim(\cB)$. Then $\alpha$ is a cycle in $Q$ and $\alpha \in \cB$. Thus it corresponds to a vertex $v$ of $G$. Then $v$ corresponds to a $G$-puncture $P$ in $(S,M)$ and $\alpha$ gives rise to the boundary curve $\ssC_P$, that is $w(\ssC_P)= 0$ and $p_{\ssC_P} = \alpha$. Locally, in $(S,M)$, this corresponds to the following configuration 


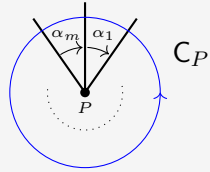
\begin{figure}[H]
	\[
	\begin{tikzpicture}[auto, thick, scale=1, font=\tiny,baseline={(0,0)}]
	
	\node[b] (0) at (0,0) {};
	\node () at ($(0)+(-90:0.2)$) {$P$};	
	\node[font=\large] () at ($(0)+(20:1.5)$){$\ssC_P$};
	
	\draw (0) to node {} ($(0)+(55:1.2)$);
	\draw (0) to node {} ($(0)+(90:1.2)$);
	\draw (0) to node {} ($(0)+(125:1.2)$);

	\draw[blue, thin, decoration={markings, mark=at position 0 with {\arrow{>}}},
	postaction={decorate}
	]
	(0) circle (1);
	
	\node () at ($(0)+(72:0.8)$){$\alpha_1$};
	\draw[->, thin] (0)+(87:.6cm) arc (87:58:.6cm);
	
	\node () at ($(0)+(108:0.8)$){$\alpha_m$};
	\draw[->, thin] (0)+(123:.6cm) arc (123:93:.6cm);
	
	\draw[dotted, thin] (0)+(10:.5cm) arc (10:-190:.5cm);
	
	\end{tikzpicture}
	\]
\end{figure}

Suppose now that $g =  (\gamma,\alpha) \in \HH^n(A)$ with $\gamma = \gamma_n \ldots \gamma_1$, where $\gamma_i \in Q_1$,  a  $\Gamma$-maximal element
of~$\Gamma$ and $\gamma$ and~$\alpha$ neither beginning nor ending with the
same arrow and $\alpha = \alpha_m \ldots \alpha_1 \in \cB$. Then first following $\gamma$ and then following $\alpha$ in the reversed order, traces out a path in $(S,M)$ which by construction  corresponds to a  closed curve $\ssC_{\gamma\alpha^{-1}}$ with $w(\ssC_{\gamma\alpha^{-1}}) = n-1$  and $p_{\ssC_{\gamma\alpha^{-1}}} = \gamma \alpha^{-1}$. Furthermore, $\alpha$ corresponds to a maximal fan $F$ at a vertex $v$ of $G$ with $m +1$ edges on some boundary component of $(S,M)$ and $\gamma$ corresponds to a maximal fan $F^*$ at a vertex $v^*$ of $G$ with $n+1$ edges of $G^*$ on the some boundary component. Since the first edge of $F$ (corresponding to $s(\alpha)$) intersects the first edge (corresponding to $s(\gamma)$) of $F^*$ and the last edge of $F$ (corresponding to $t(\alpha)$) of $F$ intersects the last edge (corresponding to $t(\gamma)$) of $F^*$,  by construction of $(S,M)$,  the vertices $v$ and $v^*$ correspond to marked points on the same boundary component $B$ and there are no other marked points on $B$. Thus $\ssC_{\gamma\alpha^{-1}}$ corresponds to the boundary curve $\ssC_B$ and  locally in $(S,M)$ we have the following configuration. 
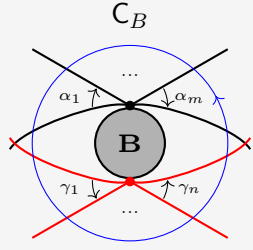
\begin{figure}[H]
	\[
	\begin{tikzpicture}[auto, thick, scale=1, font=\tiny,baseline={(0,0)}]
	
	\node[B] (0) at (0,0) {\textbf{B}};
	\node[b] (b) at ($(0)+(90:0.5)$) {};
	\node[r] (r) at ($(0)+(270:0.51)$) {};
	\node[font=\large] () at ($(0)+(90:1.7)$){$\ssC_B$};
	
	\draw[-] (b) to [out=10, in=120, looseness=0.5] node   {} ($(b)+(-20:1.7)$);
	\draw (b) to node {} ($(b)+(30:1.5)$);
	\draw (b) to node {} ($(b)+(150:1.5)$);
	\draw[-] (b) to [out=170, in=60, looseness=0.5] node   {} ($(b)+(200:1.7)$);
	
	\draw[-,red] (r) to [out=-10, in=-120, looseness=0.5] node   {} ($(r)+(20:1.7)$);
	\draw[red] (r) to node {} ($(r)+(-30:1.5)$);
	\draw[red] (r) to node {} ($(r)+(-150:1.5)$);
	\draw[-, red] (r) to [out=-170, in=-60, looseness=0.5] node   {} ($(r)+(-200:1.7)$);
	
	\draw[blue, thin, decoration={markings, mark=at position 0.08 with {\arrow{>}}},
	postaction={decorate}
	]
	($(0)$) circle (1.3);

	\node () at ($(b)+(173:0.8)$){$\alpha_1$};
	\draw[->, thin] (b)+(182:.5cm) arc (182:153:.5cm);
	
	\node () at ($(b)+(90:0.4)$){$...$};
	
	\node () at ($(b)+(7:0.8)$){$\alpha_m$};
	\draw[->, thin] (b)+(27:.5cm) arc (27:-2:.5cm);
	\node () at ($(r)+(-173:0.8)$){$\gamma_1$};
	\draw[->, thin] (r)+(-182:.5cm) arc (-182:-153:.5cm);
	
	\node () at ($(r)+(-90:0.4)$){$...$};
	
	\node () at ($(r)+(-7:0.8)$){$\gamma_n$};
	\draw[->, thin] (r)+(-27:.5cm) arc (-27:2:.5cm);
	
	\end{tikzpicture}
	\]
\end{figure}

Finally suppose that $g=\cycle{C} \in \HH^{n}(A)$ with $C = c_n \ldots c_1 \in\Crepprim(\Gamma)$. Then $C$ gives rise to a $G^*$-puncture in $(S,M)$, which in turn corresponds to a boundary component $B$ with no marked points. If $C$ is a primitive cycle, then  the path traced out by $C$ in $(S,M)$  corresponds to the boundary curve $\ssC_B$ with $w(\ssC_B) = |C| $ and $p_{\ssC_B}= C$. When $C$ is not a primitive cycle, then it is a square of a primitive cycle $D$ and the path traced out by $D$ corresponds to the boundary curve $\ssC_B$, that is $p_D=\ssC_B$. Then $C$ corresponds to $\ssC^2_B$ with $p_{\ssC^2_B} = C$ and $w(\ssC^2_B) = |C| $. That is we have the following local configuration in $(S,M)$. 
 

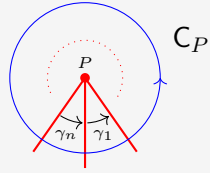
\begin{figure}[H]
	\[
	\begin{tikzpicture}[auto, thick, scale=1, font=\tiny,baseline={(0,0)}]
	
	\node[r] (0) at (0,0) {};
	\node () at ($(0)+(90:0.2)$) {$P$};	
	\node[font=\large] () at ($(0)+(20:1.5)$){$\ssC_P$};
	
	\draw[red] (0) to node {} ($(0)+(-55:1.2)$);
	\draw[red] (0) to node {} ($(0)+(-90:1.2)$);
	\draw[red] (0) to node {} ($(0)+(-125:1.2)$);

	\draw[blue, thin, decoration={markings, mark=at position 0 with {\arrow{>}}},
	postaction={decorate}
	]
	(0) circle (1);
	
	\node () at ($(0)+(-72:0.8)$){$\gamma_1$};
	\draw[->, thin] (0)+(-87:.6cm) arc (-87:-58:.6cm);
	
	\node () at ($(0)+(-108:0.8)$){$\gamma_n$};
	\draw[->, thin] (0)+(-123:.6cm) arc (-123:-93:.6cm);
	
	\draw[dotted, thin, red] (0)+(-30:.5cm) arc (-30:210:.5cm);
	
	\end{tikzpicture}
	\]
\end{figure}

Conversely, suppose that $B$ is a boundary component with one marked point corresponding to a vertex $v$ of $G$. Then by construction,  this boundary also contains a vertex $v^*$ of $G^*$ and the maximal fan $F$ in $G$ at $v$ corresponds to a maximal path $\alpha \in \cB$ and the maximal fan $F^*$ in $G^*$ at $v^*$ corresponds to a maximal path $\gamma$ in $\Gamma_{|\gamma|}$ where $|\gamma| +1$ is equal to the number of edges of $F^*$. 
 Then the boundary curve $\ssC_B $ is homotopic to the curve $\ssC_{\gamma \alpha^{-1}}$, that is $p_{\ssC_B} =\gamma \alpha^{-1}$ and  $w(\ssC_B)=|\gamma|-1$. Since both $F$ and $F^*$ are maximal, we see that  $g = (\gamma, \alpha)$ is a generator in $\cG$ in degree $\HH^{|\gamma|}(A)$. We note in particular, that $F^*$ might only contain a single edge in which case $\gamma$ is a vertex and $|\gamma| = 0$. 

Let $P$ be a $G$-puncture in $(S,M)$. Then by definition the curve $\ssC_P$ has winding number $w(\ssC_P)=0$. By construction $p_{\ssC_P}$ is primitive cyclic path in $\cB$ and by Proposition \ref{prop:G} $\cycle{p_{\ssC_P}}$ is a generator of $\HH^*(A)$ in degree 0. 

Finally suppose $B$ is a boundary component with no marked points. By construction $B$ corresponds to a vertex $v^*$ of $G^*$ and the winding number $w(\ssC_B)$ of the boundary curve $\ssC_B$ is equal to the valency $n$ of $v^*$. Then $p_\ssC$ is a complete $\Gamma$ cycle and thus if $n$ is even or $\charact(\kk) =2$ then $\cycle{p_\ssC}$ is a generator in $HH^n(A)$ and if $n$ is odd then $\cycle{p^2_\ssC}$ is a generator in $\HH^{2n}(A)$.

\end{proof}

It follows from the proof of Theorem~\ref{thm:generators in gentle surface}  that the degree of a generator $g$ in $\G\setminus \cF$ of $\HH^*(A)$ is either given  by the sum of the  winding number of the corresponding boundary curve (or the square of the boundary curve) plus the number of marked points on  the associated  boundary component or,  in case $g$ corresponds to a curve around a puncture which by definition has winding number zero, then the degree of $g$ is again the winding number of this curve. More precisely, we have the following. 

\begin{corollary}
Let $(Q,I)$ be a gentle presentation and $A=\kk Q/I$. Given a generator $g \in \G\setminus \cF$, let $B$ be the corresponding boundary component or $G$-puncture in $(S,M)$ as in the proof of Theorem~\ref{thm:generators in gentle surface}. Then the degree of $g$ is $  w(\ssC_B) +b$, where $\ssC_B$ is the boundary curve of $B$ and $b$ is the number of marked points on $B$,  except if $B$ is an unmarked boundary corresponding to a vertex of $G^*$ with an odd number of edges and the characteristic of $\kk$ is not 2,  then the degree of $g$ is $ 2 w(\ssC_B)$.    
\end{corollary}

In Table~\ref{Table:geometric}, we give the explicit bijective correspondence between the boundary components with zero or one marked point and the generators of $\HH^*(A)$. 
\strutlongstacks{T}
\begin{center}
\begin{table}[ht]\label{Table:geometric}
\begin{tabular}{|c|c|}
  \hline
  Boundary component &  Generator of $\HH^*(A)$ \\ 
  \hline
  &\\
	$
	\begin{tikzpicture}[auto, thick, scale=0.85, font=\tiny,baseline={(0,0)}]
	
	\node[B] (0) at (0,0) {\textbf{B}};
	\node[b] (b) at ($(0)+(90:0.5)$) {};
	\node[r] (r) at ($(0)+(270:0.51)$) {};
	
	\draw (b) to node {} ($(b)+(70:1.1)$);
	\draw (b) to node {} ($(b)+(110:1.1)$);
	\draw[red] (r) to node {} ($(r)+(270:1.1)$);
	\node () at ($(b)+(-100:1.9)$){$s(\alpha_1)$};

	
	\draw[-] (b) to [out=20, in=0, looseness=2.6] node   {} ($(r)+(270:0.7)$);
	\draw[-] (b) to [out=160, in=180, looseness=2.6] node   {} ($(r)+(270:0.7)$);
	
	\node () at ($(b)+(140:0.6)$){$\alpha_1$};
	\draw[->, thin] (b)+(160:.4cm) arc (160:115:.4cm);

	\node () at ($(b)+(90:0.4)$){$...$};
	
	\node () at ($(b)+(40:0.6)$){$\alpha_m$};
	\draw[->, thin] (b)+(65:.4cm) arc (65:20:.4cm);
	
	\end{tikzpicture}
	$
 & \small{$(s(\alpha), \alpha) \in \HH^0(A) $ where $\alpha= \alpha_m \ldots \alpha_1$ is $\Gamma$-maximal} \\ 
 &\\
  \hline
&\\
 $\begin{tikzpicture}[auto, thick, scale=0.85, font=\tiny,baseline={(0,0)}]
	
	\node[b] (0) at (0,0) {};
	\node () at ($(0)+(-90:0.2)$) {$P$};	
	
	\draw (0) to node {} ($(0)+(55:1.2)$);
	\draw (0) to node {} ($(0)+(90:1.2)$);
	\draw (0) to node {} ($(0)+(125:1.2)$);

	
	\node () at ($(0)+(72:0.8)$){$\alpha_1$};
	\draw[->, thin] (0)+(87:.6cm) arc (87:58:.6cm);
	
	\node () at ($(0)+(108:0.8)$){$\alpha_m$};
	\draw[->, thin] (0)+(123:.6cm) arc (123:93:.6cm);
	
	\draw[dotted, thin] (0)+(10:.5cm) arc (10:-190:.5cm);
	
	\end{tikzpicture}$  &  
	\Centerstack{
	 \small{$\cycle{\alpha}  = \sum_{i=0}^{r-1}(s(\rot^i(\alpha)),\rot^i(\alpha))  \in \HH^0(A) $} \\ \small{ where $\alpha= \alpha_m \ldots \alpha_1$ is a co-complete cycle} \\ 
	 \footnotesize{Note: if $\cycle{\alpha}$ is viewed as a simple closed curve around $P$,}\\ \footnotesize{ then  for all $k>1$, the non-zero element $\cycle{\alpha^k}$  in $\HH^0(A)$} \\ \footnotesize{ can be viewed as $k$-th power of the closed curve around $P$} } 
	 \\
&\\
  \hline
  &\\
  $
  \begin{tikzpicture}[auto, thick, scale=0.85, font=\tiny,baseline={(0,0)}]
	
	\node[B] (0) at (0,0) {\textbf{B}};
	\node[b] (b) at ($(0)+(90:0.5)$) {};
	\node[r] (r) at ($(0)+(270:0.51)$) {};
	
	\draw[-] (b) to [out=10, in=120, looseness=0.5] node   {} ($(b)+(-20:1.7)$);
	\draw (b) to node {} ($(b)+(30:1.5)$);
	\draw (b) to node {} ($(b)+(150:1.5)$);
	\draw[-] (b) to [out=170, in=60, looseness=0.5] node   {} ($(b)+(200:1.7)$);
	
	\draw[-,red] (r) to [out=-10, in=-120, looseness=0.5] node   {} ($(r)+(20:1.7)$);
	\draw[red] (r) to node {} ($(r)+(-30:1.5)$);
	\draw[red] (r) to node {} ($(r)+(-150:1.5)$);
	\draw[-, red] (r) to [out=-170, in=-60, looseness=0.5] node   {} ($(r)+(-200:1.7)$);
	

	\node () at ($(b)+(173:0.8)$){$\alpha_1$};
	\draw[->, thin] (b)+(182:.5cm) arc (182:153:.5cm);
	
	\node () at ($(b)+(90:0.4)$){$...$};
	
	\node () at ($(b)+(7:0.8)$){$\alpha_m$};
	\draw[->, thin] (b)+(27:.5cm) arc (27:-2:.5cm);
	\node () at ($(r)+(-173:0.8)$){$\gamma_1$};
	\draw[->, thin] (r)+(-182:.5cm) arc (-182:-153:.5cm);
	
	\node () at ($(r)+(-90:0.4)$){$...$};
	
	\node () at ($(r)+(-7:0.8)$){$\gamma_n$};
	\draw[->, thin] (r)+(-27:.5cm) arc (-27:2:.5cm);
	
	\end{tikzpicture}
$ & \Centerstack{ 
\small{$(\gamma, \alpha) \in \HH^n(A)$ where $\gamma$ is a $\Gamma$-maximal path with} \\ \small{ $\alpha$ and $\gamma$ not starting or ending with the same arrow.} 
 } \\
&\\
  \hline
  &\\

  $\begin{tikzpicture}[auto, thick, scale=0.85, font=\tiny,baseline={(0,0)}]
	
	\node[r] (0) at (0,0) {};
	\node () at ($(0)+(90:0.2)$) {$P$};	
	
	\draw[red] (0) to node {} ($(0)+(-55:1.2)$);
	\draw[red] (0) to node {} ($(0)+(-90:1.2)$);
	\draw[red] (0) to node {} ($(0)+(-125:1.2)$);

	
	\node () at ($(0)+(-72:0.8)$){$\gamma_1$};
	\draw[->, thin] (0)+(-87:.6cm) arc (-87:-58:.6cm);
	
	\node () at ($(0)+(-108:0.8)$){$\gamma_n$};
	\draw[->, thin] (0)+(-123:.6cm) arc (-123:-93:.6cm);
	
	\draw[dotted, thin, red] (0)+(-30:.5cm) arc (-30:210:.5cm);
	
	\end{tikzpicture}
$ &  \Centerstack{  \small{$\cycle{C} =  \sum_{i=0}^{r-1} (\rot^i(C),s(\rot^i(C))) \in \HH^{\varepsilon n}(A)$ } \\ \small{ where $C = D^\varepsilon$ and $D = \gamma_n \ldots \gamma_1$ is a complete cycle} \\ \small{ and $\varepsilon =1$ if $n$ is even or $\charact(\kk)=2$ and $\varepsilon =2$ otherwise.}}  \\
&\\
\hline
\end{tabular}
\vspace{0.5cm}\caption{Correspondence of boundary components with zero or one marked point and the generators of $\HH^*(A)$.}
\end{table}
\end{center}
\begin{remark}\label{rk:geom} (1)  \emph{ The geometric interpretation of the generators in $\cF$:}
From the construction of $G$ on $(S,M)$, we immediately see that an element in $(c, c)$ in $\cF$ corresponds to a curve $\ssC_{(c,c)}$ with $p_{\ssC_{(c,c)}} =c $ connecting two marked points on the boundary and that $w(\ssC_{(c,c)}) = 1$.

(2) \emph{ The geometric interpretation of the cup product:} Given the interpretation of elements in $\cF$ in (1), we note that any basis element in Theorem~\ref{thm:cohomology:basis} which is not a generator is a (cup) product of generators. It  follows from Theorem~\ref{thm:generators in gentle surface} and the description of the cup product in Chapter~\ref{chapter:cup} that it can be described by a curve obtained from concatenating the curves of the corresponding generators.

(3) \emph{ The geometric interpretation of the Gerstenhaber bracket:} Recall from Theorem~\ref{thm:HH1Lie} that the Gerstenhaber bracket of the generators of the Hochschild cohomology is almost always zero and that the only non-zero brackets arise from brackets of the form $[\cF, \cG\setminus \cF]$. More precisely, the bracket $[(c,c), v]$, for $(c,c) \in \cF$ and $v \in \cG \setminus \cF$ is non-zero if $\deg_c(v)$ is non-zero, that is $c$ appears in $v$ and in this case, 
$[(c,c), v] = \deg_c(v)v$. In terms of the geometric model, suppose that $(c,c)$ is given by an open curve   $\ssC_{(c,c)}$ with  $w(\ssC_{(c,c)}) = 1$.  The fact that $\deg_c(v)$ is non-zero then corresponds to the curve $\ssC_{(c,c)}$ and the curve associated to $v$ running parallel between two consecutive edges in a fan in the dual graph $G^*$ (embedded in $S$) and $\deg_c(v)$ counts the number of times the curves run parallel in such a way. 

\end{remark}

\begin{remark}
In \cite{OPS, LP} it was shown that the derived invariant for a \fd gentle algebra $A$  constructed by Avella-Alaminos and Geiss in \cite{AAG}, the AAG-invariant, is given by the boundary components of the marked surface $(S,M)$ associated to $A$. In particular, each boundary component gives to one non-zero entry for the AAG-invariant. On the other hand, we have seen that the Hochschild cohomology corresponds to  only boundary components with zero or one marked point. As a consequence, the Hochschild cohomology is a much weaker invariant than the AAG-invariant, since it is easy to construct two algebras that have the same number of boundary components with zero and one marked points, giving rise to isomorphic Hochschild cohomology but where the surfaces of the two algebras have different number of boundary components with more than one marked point. 
\end{remark}

\newpage
\section{Geometric interpretation of Hochschild homology} We now give a geometric interpretation of the basis of the Hochschild homology presented in Theorem~\ref{thm:basishomology} in the case of a  gentle algebra. We note that while the interpretation in the surface of the generators of the Hochschild cohomology as well as the cup product are very natural, for the interpretation of Hochschild homology there is not really a canonical choice to represent the basis elements. In what follows, we make a particular choice, but we note that other choices would also be possible.  Given this choice, the cap product can be interpreted as unwinding curves around $G^*$-punctures whereas the Connes differential, on the other hand, can be interpreted as winding a rotation of angle $2\pi/r$ around  each $G^*$-puncture of valence $r$, that is if the $G^*$-puncture corresponds  to a primitive complete cycle of length $r$. 

Recall that if $C$ is a complete circuit of period $r$ then $\bar  C$ is the $n$th power of some primitive cycle $\bar  C_{prim}$ of length $r$, that is $\bar  C = \bar C_{prim}^n$.  Furthermore,   $\bar  C_{prim}$ and hence also $ C$ corresponds to a $G^*$-puncture $P$ of valency $r$. In Figure~\ref{Figure:homologycurves} we define two curves $\ssC_n$ and $\ssC'_n$ associated to $\bar C$ corresponding to a particular choice of one of the $G^*$-edges incident  with $P$, namely such that the first intersection of $\ssC_n$ and $\ssC'_n$ with $G^*$ corresponds to the vertex $s(\bar C)$. 

\begin{figure}[H]
	\[
	\begin{tikzpicture}[auto, thick, scale=1, font=\tiny,baseline={(0,0)}]
	
	\node[r] (0) at (0,0) {};
	\node () at ($(0)+(240:0.2)$) {$P$};
	\node[b] (1) at ($(0)+(-90:1.5)$) {};
	\node[b] (2) at ($(0)+(-30:1.5)$) {};
	\node[b] (3) at ($(0)+(30:1.5)$) {};
	\node[b] (4) at ($(0)+(90:1.5)$) {};
	\node[b] (5) at ($(0)+(150:1.5)$) {};
	
	\draw[red] (0) to node [pos=0.8, right] {$3$} ($(0)+(-60:1.5)$);
	\draw[red] (0) to node [pos=0.8, above] {$2$} ($(0)+(0:1.5)$);
	\draw[red] (0) to node [pos=0.9, left] {$1$} ($(0)+(60:1.5)$);
	\draw[red] (0) to node [pos=0.8, left] {$r$} ($(0)+(120:1.5)$);
	
	\draw[black] (1) to [bend left] node {} (2);
	\draw[black] (2) to [bend left] node {} (3);
	\draw[black] (3) to [bend left] node {} (4);
	\draw[black] (4) to [bend left] node {} (5);
	
	\draw[dotted, thin, black] (0)+(160:1.5cm) arc (160:260:1.5cm);
	\draw[dotted, thin, red] (0)+(140:.33cm) arc (140:285:.33cm);

	\node (c) at ($(0)+(210:0.2)$) {};
	\draw[green] (4) to node [bend left] {} ($(c)+(10:0.6)$);
	\spiral[green](c)(270:10)(0.4:.6)[2];
	\spiral[green, dotted](c)(180:270)(0.4:.4)[0];
	\spiral[green](c)(110:180)(0.4:.4)[0];
	\draw[green] (3) to node {} ($(c)+(110:0.4)$);
	
	\node () at ($(1)+(270:0.6)$) {\large  $\ssC_n$ };
	
	
	\node[r] (0) at ($(0)+(0:5.5)$) {};
	\node () at ($(0)+(240:0.2)$) {$P$};
	\node[b] (1) at ($(0)+(-90:1.5)$) {};
	\node[b] (2) at ($(0)+(-30:1.5)$) {};
	\node[b] (3) at ($(0)+(30:1.5)$) {};
	\node[b] (4) at ($(0)+(90:1.5)$) {};
	\node[b] (5) at ($(0)+(150:1.5)$) {};
	
	\draw[red] (0) to node [pos=0.8, right] {$3$} ($(0)+(-60:1.5)$);
	\draw[red] (0) to node [pos=0.8, above] {$2$} ($(0)+(0:1.5)$);
	\draw[red] (0) to node [pos=0.9, left] {$1$} ($(0)+(60:1.5)$);
	\draw[red] (0) to node [pos=0.8, left] {$r$} ($(0)+(120:1.5)$);
	
	\draw[black] (1) to [bend left] node {} (2);
	\draw[black] (2) to [bend left] node {} (3);
	\draw[black] (3) to [bend left] node {} (4);
	\draw[black] (4) to [bend left] node {} (5);
	
	\draw[dotted, thin, black] (0)+(160:1.5cm) arc (160:260:1.5cm);
	\draw[dotted, thin, red] (0)+(140:.33cm) arc (140:285:.33cm);
	
	\node (c) at ($(0)+(210:0.2)$) {};
	\draw[green] (4) to node [bend left] {} ($(c)+(10:0.6)$);
	\spiral[green](c)(310:10)(0.4:.6)[2];
	\spiral[green, dotted](c)(180:310)(0.4:.4)[0];
	\spiral[green](c)(170:180)(0.4:.4)[0];
	\draw[green] (4) to node {} ($(c)+(170:0.4)$);
	
	\node () at ($(1)+(270:0.6)$) {\large  $\ssC_n'$ };
	
	\end{tikzpicture}
	\]
	\caption{Given a $G^*$-puncture $P$ of valency $r$,  the example on the left is a curve of the form $\ssC_n$ wrapping $n$ times around $P$ and the example  on the right is a curve of the form $\ssC'_n$, also wrapping $n$ times around $P$. }\label{Figure:homologycurves}
\end{figure}
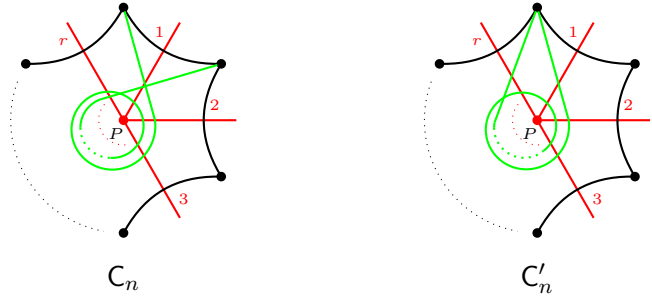

Note that the winding number of $\ssC_n$ is $w(\ssC_n) = nr$ and the winding number of $\ssC'_n$ is $w(\ssC'_n) = nr-1$.

\begin{theorem}\label{thm:geometric1}
Let $(Q,I)$ be a \fd gentle presentation and let $A\coloneqq\kk
Q/I$ be the finite dimensional algebra it presents with surface $(S,M)$ induced by the associated ribbon graph $G$. 
Assume further that the characteristic of $\kk$ is not equal to 2 and let $\cB$ be the basis of $\HH_*(A)$ in Theorem~\ref{thm:basishomology}. 
\begin{enumerate}
\item Let  $v$ be an edge in  $G$. Then $v$ is  an open curve of winding number  0 and it corresponds to the basis element $(e_v, e_v)$ in $\HH_0(A)$ where $e_v$ is the vertex in $Q_0$ corresponding to $v$. 
\item Let $P$ be a $G^*$-puncture of valency $r \geq 1$ and let $\ssC_n$ and $\ssC'_n$ be the corresponding curves of winding number $nr$ and $nr-1$ as defined in Figure \ref{Figure:homologycurves} for a fixed labelling of $G^*$. 
Then if $nr \geq2$, 
\begin{itemize} 
\item $\ssC_n$ corresponds to the basis element of $\HH_{nr}(A)$ of the form $
  \hcycle{\bar C}\coloneqq
  \sum_{i=0}^{r-1}(-1)^{(nr+1)i}\cdot
        \bigl(s(\rot^i(\bar C)), \rot^i(\bar C))$ if $(-1)^{(nr+1)r}=1$ in $\kk$ and 
\item $\ssC'_r$ corresponds to the basis element of $\HH_{nr-1}(A)$ of the form $(\lfact1{\bar C},\lfact2{\bar C})$ if 
$(-1)^{(nr-1)r}=1$ in $\kk$, 
\end{itemize}
 where in both cases $C$ is the complete  circuit of length $nr$ and of period $r$ corresponding to $P$.  
\end{enumerate}
Furthermore, every basis element in $\cB$ appears in this way. 
\end{theorem}

\begin{proof}
Statement (1) directly follows from the fact that the vertices of $Q$ are in bijection with the edges of $G$ and that  by construction every edge of $G$ crosses exactly one edge of $G^*$ and therefore has winding number zero. 
For the statement of (2), note that the choice of the element $(\lfact1{\bar C},\lfact2{\bar C})$ (in its rotation equivalence class) corresponds to choosing a particular rotation of the corresponding cycle $\bar C$ in $Q$. Furthermore, by Theorem~\ref{thm:basishomology} we have that $\bar  C$ is the $n$th power of a primitive cycle of length $r$. Up to relabelling of the edges of $G$, the curve $\ssC_{\lfact2{\bar C}}$ defined by the path 
$\lfact2{\bar C}$ corresponds exactly to a  curve of the form   $\ssC'_n$ and has winding number $nr-1$. On the other hand, without loss of generality we can assume that $\bar C= \lfact2{\bar C}\lfact1{\bar C}$ and thus corresponds exactly to one of the cycles in the expression $\sum_{i=0}^{r-1}(-1)^{(nr+1)i}\cdot
        \bigl(s(\rot^i(\bar C)), \rot^i(\bar C))$. Thus the curve $\ssC_{\bar C}$ traced out by the path of $\bar C=\lfact2{\bar C}\lfact1{\bar C}$  corresponds exactly to the curve $\ssC_n$ and has winding number $nr$. Comparing with Theorem~\ref{thm:basishomology} we see that all basis elements are accounted for. 
\end{proof}

\begin{remark}
 Using the above geometric interpretation of the Hochschild homology and cohomology together with Theorem~\ref{thm:cap}, the cap product for a \fd gentle algebra can be interpreted as unwrapping or unwinding curves around $G^*$-punctures.  More precisely, in the case of a \fd gentle algebra $A$ given by a gentle presentation $(Q,I)$, we have the following three cases of cap products to consider (using the notation of Theorem~\ref{thm:cap}):
\setlength{\leftmargini}{0cm}  
\begin{itemize} 
\item Let~$c$ be an arrow in~$Q_1\setminus T$ where $T$ is a spanning tree of $Q$  let and let~$C$ be
a complete  circuit in which~$c$ appears. Suppose  that~$C$ is of length~$m$ and
period~$r$ and $(-1)^{(m+1)r}=1$ in~$\kk$. Then   $(c,c) \in \HH^1(A)$ and $\hcycle{\bar C} \in \HH_m(A)$, and we have 
  \begin{equation*}\label{Eq:GeomCap1}
   \hcycle{\bar C}\frown (c,c) = (-1)^{m+1}\cdot(\lfact1{\bar C},\lfact2{\bar C}) \in \HH_{m-1}(A).
  \end{equation*}
  According to Remark~\ref{rk:geom}(1) and Theorem~\ref{thm:geometric1} the Hochschild cocycle $(c,c)$ corresponds to an open curve of winding number one and the Hochschild cycle $\hcycle{\bar C}$ corresponds to a curve $\ssC_{n}$ starting and ending on two 'consecutive' marked points on the boundary and wrapping around a $G^*$-puncture $n=m/r$ times, where $r$ is the valency of the $G^*$-puncture (that is the period of $C$). The cap product of these two elements then corresponds to the curve $\ssC_{n}'$ as described in Theorem~\ref{thm:geometric1}(3), wrapping around the $G^*$-puncture $n$ times and starting and ending at the same marked point on the boundary. 
  
  
  \item Let $E$ be a primitive circuit in~$\Cprim(\Gamma)$ of length $r $ (that is $E$ corresponds to a $G^*$-puncture of valency $r$) and let $n  \geq w >0$  be such that $(-1)^{(nr+1)r}=1$. 
  Then $\hcycle{\bar E^r} \in \HH_{nr}(A)$  and $\cycle{\bar E^w} \in \HH^{rw}(A)$ and 
  \begin{equation*}\label{Eq:GeomCap2} \hcycle{\bar E^n}
  \frown \cycle{\bar E^w}
    = \begin{dcases*}
      \sum_{i=0}^{r-1}(-1)^i\cdot(s(\rot^i(\bar E)),s(\rot^i(\bar E))) \in \HH_0(A)
        & if $n=w$; 
        \\
      \hcycle{\bar E^{n-w}} \in \HH_{r(n-w)}(A)
        & if $n>w$.
      \end{dcases*}
  \end{equation*}
 Following  Theorem~\ref{thm:geometric1}, $\hcycle{\bar E^n} \in \HH_{nr}(A)$  corresponds to a curve of the form $\ssC_n$ wrapping $n$ times around the $G^*$-puncture corresponding to $E$ and by the proof of Theorem~\ref{thm:generators in gentle surface}, $\cycle{\bar E^w} \in \HH^{rw}(A)$ corresponds to a closed curve wrapping $w$ times around the same $G^*$-puncture. The cap product $\hcycle{\bar E^n}
  \frown \cycle{\bar E^w}$ then can be seen as corresponding the curve $\ssC_{n-w}$ which is obtained from $\ssC_n$ by unwinding it $w$ times if $n>w$ and as the (open) curve corresponding to the edge of the ribbon graph $G$ corresponding to the first edge of $G^*$ crossed by $\ssC_n$ (in the left picture of Figure~\ref{Figure:homologycurves} this would correspond to the edge labelled 1).  
 
If $n>w>0$ then we also have
  \begin{equation*}\label{Eq:GeomCap3}
  (\lfact1{{(\bar E^n)}},\lfact2{{(\bar E^n)}}) \frown \cycle{\bar E^w}
    = (\lfact1{(\bar E^{n-w})},\lfact2{(\bar E^{n-w})}) \in \HH_{(r(n-w))-1}(A),
  \end{equation*}
\end{itemize}
where $(\lfact1{{(\bar E^n)}},\lfact2{{(\bar E^n)}}) \in \HH_{nr-1}(A)$. Now by Theorem~\ref{thm:geometric1}, $(\lfact1{{(\bar E^n)}},\lfact2{{(\bar E^n)}})$ corresponds to a curve of the form $\ssC_n'$ that is wrapping $n$ times around the $G^*$-puncture corresponding to $E$ and $\cycle{\bar E^w}$ again corresponds to a closed curve winding $w$ times around the same $G^*$-puncture. Then the cap product  $(\lfact1{{(\bar E^n)}},\lfact2{{(\bar E^n)}}) \frown \cycle{\bar E^w}
    = (\lfact1{(\bar E^{n-w})},\lfact2{(\bar E^{n-w})})$ can be seen as the w-fold unwinding of $\ssC_n'$ resulting in the curve $\ssC_{n-w}'$. 
\end{remark}

Furthermore, it  follows directly from Theorem~\ref{thm:geometric1} and Proposition~\ref{prop:B} that the Connes differential can be interpreted in the following way.

\begin{corollary}
With the notation of Theorem~\ref{thm:geometric1},  the Connes differential $B: \HH_{nr-1}(A) \to \HH_{nr}(A)$ sends the element in $\HH_{nr-1}(A)$  corresponding to $\ssC'_n$ to the element in $\HH_{nr}(A)$  corresponding to $\ssC_n$. It thus can be interpreted as   wrapping the curve $\ssC'_n$ by a `further' angle of $2 \pi/r $ in the  anti-clockwise direction around the $G^*$ puncture giving rise to $\ssC'_n$.
\end{corollary}

\backmatter

\bibliographystyle{amsalpha}
\bibliography{biblio}

\printindex[notations]
\printindex

\end{document}